\documentclass{amsart}
\usepackage{amssymb, amsmath,amsfonts,epsfig}
\usepackage{color}
\usepackage{enumerate}

\usepackage[T1]{fontenc}
\usepackage{yfonts}

\usepackage{tikz}
\usetikzlibrary{arrows,decorations.pathmorphing,backgrounds,positioning,fit, through}

  \usepackage[margin=1.3in]{geometry}
\newcommand{\mc}[1]{{}}
\mathchardef\GG="321D

\newcommand{\Ebar}{{\overline{E}}}

\newcommand{\Qbar}{{\overline{Q}}}

\newcommand{\oDelta}{{\overline \Delta}}
\newcommand{\oT}{{\overline T}}
\newcommand{\oalpha}{{\overline \alpha}}
\newcommand{\oS}{{\overline \cS}}
\newcommand{\oomega}{{\overline \omega}}
\newcommand{\oE}{{\overline E}}

\newcommand{\oW}{{\overline W}}

\newcommand{\cBt}{{\widetilde{\cB}}}
\newcommand{\qt}{{\tilde{q}}}
\newcommand{\omegat}{{\tilde{\omega}}}
\newcommand{\Sigmat}{{\widetilde{\Sigma}}}
\newcommand{\St}{{\widetilde{S}}}

\newcommand{\alphabar}{{\overline{\alpha}}}
\newcommand{\omegabar}{{\overline{\omega}}}
\newcommand{\betabar}{{\overline{\beta}}}

\newcommand{\pp}{{\bf p}}

\newcommand{\cA}{{\mathcal A}}
\newcommand{\cB}{{\mathcal B}}

\newcommand{\cE}{{\mathcal E}}
\newcommand{\cF}{{\mathcal F}}
\newcommand{\cG}{{\mathcal G}}
\newcommand{\cH}{{\mathcal H}}
\newcommand{\cI}{{\mathcal I}}

\newcommand{\cK}{{\mathcal K}}
\newcommand{\cL}{{\mathcal L}}
\newcommand{\cM}{{\mathcal M}}
\newcommand{\cN}{{\mathcal N}}

\newcommand{\cP}{{\mathcal P}}
\newcommand{\cQ}{{\mathcal Q}}
\newcommand{\cR}{{\mathcal R}}
\newcommand{\cS}{{\mathcal S}}
\newcommand{\cT}{{\mathcal T}}
\newcommand{\cU}{{\mathcal U}}

\newcommand{\cW}{{\mathcal W}}

\renewcommand{\gg}{{\mathfrak g}}
\newcommand{\gh}{{\mathfrak h}}

\newcommand{\bg}{{\mathbf g}}

\newcommand{\bj}{{\mathbf j}}

\newcommand{\bp}{{\mathbf p}}

\DeclareMathOperator{\twist}{twist}

\DeclareMathOperator{\Mod}{Mod}
\DeclareMathOperator{\Ext}{Ext}

\DeclareMathOperator{\I}{i}
\DeclareMathOperator{\hol}{hol}

\DeclareMathOperator{\Log}{Log}

\DeclareMathOperator{\SL}{SL}

\DeclareMathOperator{\II}{Im}

\newcommand{\half}{{\mathbb H}}
\newcommand{\reals}{{\mathbb R}}
\newcommand{\natls}{{\mathbb N}}
\newcommand{\cx}{{\mathbb C}}

\newcommand{\T}{\mathcal{T}}

\newtheorem{theorem}{Theorem}[section]
\newtheorem{proposition}[theorem]{Proposition}
\newtheorem{corollary}[theorem]{Corollary}
\newtheorem{lemma}[theorem]{Lemma}

\theoremstyle{definition}
\newtheorem{definition}[theorem]{Definition}
\newtheorem{example}[theorem]{Example}

\theoremstyle{remark}
\newtheorem{remark}[theorem]{Remark}

\newcommand{\thmref}[1]{Theorem~\ref{Thm:#1}}
\newcommand{\propref}[1]{Proposition~\ref{Prop:#1}}
\newcommand{\secref}[1]{\S\ref{Sec:#1}}
\newcommand{\lemref}[1]{Lemma~\ref{Lem:#1}}
\newcommand{\corref}[1]{Corollary~\ref{Cor:#1}}
\newcommand{\figref}[1]{Fig.~\ref{Fig:#1}}
\newcommand{\exref}[1]{Example~\ref{Exmp:#1}}
\newcommand{\remref}[1]{Remark~\ref{Rem:#1}}
\newcommand{\eqnref}[1]{Equation~\eqref{Eq:#1}}
\newcommand{\defref}[1]{Definition~\ref{Def:#1}}

\newcommand{\param}{{\mathchoice{\mkern1mu\mbox{\raise2.2pt\hbox{$
\centerdot$}}
\mkern1mu}{\mkern1mu\mbox{\raise2.2pt\hbox{$\centerdot$}}\mkern1mu}{
\mkern1.5mu\centerdot\mkern1.5mu}{\mkern1.5mu\centerdot\mkern1.5mu}}}

\newcommand{\emul}{\stackrel{{}_\ast}{\asymp}}
\newcommand{\gmul}{\stackrel{{}_\ast}{\succ}}
\newcommand{\lmul}{\stackrel{{}_\ast}{\prec}}
\newcommand{\eadd}{\stackrel{{}_+}{\asymp}}
\newcommand{\gadd}{\stackrel{{}_+}{\succ}}
\newcommand{\ladd}{\stackrel{{}_+}{\prec}}

\newcommand{\C}{{\mathbb C}}

\newcommand{\Z}{{\mathbb Z}}
\newcommand{\R}{{\mathbb R}}

\newcommand{\bA}{{\bf A}}
\newcommand{\bB}{{\bf B}}
\newcommand{\bC}{{\bf C}}

\newcommand{\Bers}{{\sf B}}
\newcommand{\CNM}{{\sf M_\tau}}
\newcommand{\CK}{{\sf c_{\cK}}}
\newcommand{\CX}{{\sf c_X}}
\newcommand{\const}{{c}}

\newcommand{\cross}{\times}
\newcommand{\ep}{\epsilon}
\newcommand{\vs}{\varsigma}
\newcommand{\Teich}{Teich\-m\"uller~}
\newcommand{\from}{\colon \thinspace}
\newcommand{\ST}{{\: \Big| \:}}
\newcommand{\st}{{\: | \:}}

\newcommand{\Stratum}{{\mathcal C}}
\newcommand{\Moduli}{{\cM(S)}}

\newcommand{\QM}{\cQ^1\cM(S)}
\newcommand{\QT}{\cQ^1\cT(S)}
\newcommand{\Qs}{{\cQ(\sigma)}}
\newcommand{\QMs}{\cQ^1\cM(\sigma)}
\newcommand{\QTs}{\cQ^1\cT(\sigma)}

\newcommand{\Cje}{\Stratum_{j,\epsilon}}
\newcommand{\Qje}{\cQ_{j,\epsilon}}

\newcommand{\Bj}{B_{j} (\Qs, }

\newcommand{\Aj}{A^{\tau}_{j,\epsilon}}

\newcommand{\Pjt}{\cP_{\theta,\tau} (\Qje, }

\newcommand{\Njt}{N_{\theta}(\Qje, }

\newcommand{\Pjht}{\widehat \cP_{\theta,\tau} (\Qje,}

\newcommand{\Nt}{{\widetilde{\cN}}}

\newcommand{\Sqn}{{\cS_q^{\, {\scriptscriptstyle \le} \tau}}}
\newcommand{\Sqc}{{\cS_q^{\, {\scriptscriptstyle \ge} \tau}}}
\newcommand{\Sn}[1]{{\cS_{#1}^{\, {\scriptscriptstyle \le} \tau}}}
\newcommand{\Sc}[1]{{\cS_{#1}^{\,  {\scriptscriptstyle \ge} \tau}}}

\begin{document}

  \title    {Counting closed geodesics in strata}

  \author[A.~Eskin]{Alex Eskin}
  \address{Dept. of Mathematics,
  University of Chicago,
  Chicago, Illinois 60637
   \newline \phantom{x} \qquad \tt http://math.uchicago.edu/$\sim$eskin/}
   \email{\tt eskin@math.uchicago.edu}

  \author[M.~Mirzakhani]{Maryam Mirzakhani}
  \address{Department of Mathematics,   Stanford University,
   Building 380, Stanford, California 94305}
   \email{\tt mmirzakh@math.stanford.edu}
  
  \author[K.~Rafi]{Kasra Rafi}
  \address{Department of Mathematics, University of Toronto,
  Toronto, Ontario, Canada M5S 2E4
  \newline \phantom{x} \qquad \tt http://www.math.toronto.edu/$\sim$rafi/}
  \email{\tt rafi@math.toronto.edu}

\begin{abstract}
We compute the asymptotic growth rate of the number $N(\Stratum, R)$ of 
closed geodesics of length $\leq R$ in a connected component $\Stratum$ of 
a stratum of quadratic differentials. We prove that, for any $0\leq \theta\leq 1$, 
the number of closed geodesics $\gamma$ of length at most $R$ such that 
$\gamma$ spends at least $\theta$--fraction of its time outside of a compact  
subset of $\Stratum$ is exponentially smaller than $N(\Stratum, R)$. 
The theorem follows from a lattice counting statement. For points $x, y$ 
in the moduli space $\Moduli$ of Riemann surfaces, and for $0 \le \theta \le1$ 
we find an upper-bound for the number of geodesic paths of length $\leq R$ in 
$\Stratum$ which connect a point near $x$ to a point near $y$ and spend at
least a $\theta$--fraction of the time outside of a compact subset of $\Stratum$. 
\end{abstract}

\maketitle
 
 
\section{Introduction}
Let $S=S_{g,p}$ be a surface of genus $g$ with $p$ punctures and let $\Moduli$ 
be the moduli space of Riemann surfaces homeomorphic to $S$.  
The co-tangent bundle of $\Moduli$ is naturally identified with $\cQ\Moduli$ 
the space of finite area quadratic differentials on $S$. Let $\QM$ be 
subspace of quadratic differentials of area 1. There is a natural $\SL(2, \reals)$ 
action on the $\QM$. The orbits of the diagonal flow, 
$g_{t}=\begin{bmatrix} e^{t} & 0 \\ 0 & e^{-t}\end{bmatrix}$, projects 
to geodesics in $\Moduli$ equipped with the \Teich metric.
For $R>0$, let $N(R)$ be the number of closed \Teich geodesics of 
length less than or equal to $R$ on $\QM$. It was shown in 
\cite{EM:geodesic-counting} that, as $R \rightarrow \infty$, the number $N(R)$ 
is asymptotic to $e^{hR}/hR,$ where $h=6g-6+2p$. 

The moduli space of quadratic differentials is naturally stratified: to each 
quadratic differential $(x,q) \in \cQ\cM(S)$ we can associate 
$\sigma(q)=(\nu_i, \ldots, \nu_k, \vs)$ 
where $\nu_1, \ldots, \nu_k$ are the orders of the zeros and poles of $q$, and 
$\vs \in \{-1, 1\}$ is equal to $1$ if $q$ is the square of an abelian differential 
and $-1$ otherwise. For a given tuple $\sigma$, we say a quadratic differential 
$(x,q) \in \cQ\cM(S)$ is of type $\sigma$ if $\sigma(q)= \sigma$. The space 
$\cQ\cM(\sigma)$ of all quadratic differentials in $\cQ\cM(S)$ of type $\sigma$ 
is called the stratum of quadratic differentials of type $\sigma$. 
The stratum $\cQ\cM(\sigma)$ is an analytic orbifold of real dimension 
$4g+2k+ \vs-3$.

Let $\QMs$ be the space of quadratic differentials in $\cQ\cM(\sigma)$
of area $1$.  It is not necessarily connected (see \cite{K:Z} 
and \cite{Lanneau} for the classification of the connected components),
however, each connected component is $\SL(2,\reals)$ invariant. 
Let $\Stratum$ be a connected component of $\QMs$. In this paper, we study 
the asymptotic growth rate of the number $N (\Stratum, R)$
of closed \Teich geodesics of length less than or equal to $R$ in 
$\Stratum$. Our main tool is estimating the number $N(\Stratum_{\cK},R)$ 
of closed geodesics that stay completely outside of a large compact set 
$\cK \subset \Stratum$.  

\begin{theorem}\label{Thm:short:saddle}
Given $\delta> 0$ there exists a compact subset $\cK \subset \Stratum$ 
and $R_0>0$ such that for all  $R>R_0$,
\begin{displaymath}
N(\Stratum_\cK,R) \le e^{(h-1+\delta)R}.
\end{displaymath}
\end{theorem}
This result implies that:
\begin{theorem}  \label{Thm:asymp:all:geodesics}
As $R \to \infty$, we have
\begin{displaymath}
N(\Stratum, R) \sim \frac{e^{hR}}{hR},
\end{displaymath}
where $h = \frac{1}{2} [1+\dim_{\reals}(\Stratum)]$ and
the notation $A \sim B$ means that the ratio $A/B$ tends to $1$
as $R$ tends to infinity. 
\end{theorem}
\begin{remark}
In the case of abelian differentials, $h$ is equal to the dimension of the 
relative homology of $S$ with respect to the set of singular points of 
$(x,q) \in \Stratum$, otherwise $h$ is one less.  
\end{remark}

\subsection*{Recurrence of geodesics.}
We prove a stronger version of \thmref{short:saddle}. Every 
quadratic differential defines a singular Euclidean
metric on the surface $S$ and for every compact set $\cK \subset \Stratum$, 
there is a lower bound for the $q$--length of a saddle connection
where $q \in \cK$. Here, we restrict attention to closed geodesics where more 
than one simple closed curve or saddle connection is assumed to be short; 
in this case the growth rate is of even lower order. 

Let $\T(S)$ be the \Teich space, the universal cover of $\Moduli$.
Let $\cQ\T(S)$ and $\QT$ be defined similarly. 
To distinguish between points in the Moduli space and \Teich space, 
we use $x \in \Moduli$ and $X \in \T(S)$. Also, we use the notation 
$(x,q)$ for points in  $\QM$ and $(X,q)$ for points in $\QT$. We denote a geodesic
in $\QM$ by $\bg$ and a geodesic in $\QT$ by $\cG$. The space $\QT$ is also 
naturally stratified. We denote the space of quadratic differentials in $\QT$ of type 
$\sigma$ by $\QTs$. To simplify the notation, let 
$$
\Qs := \QTs.
$$
Recall that $\Ext_X(\alpha)$ denotes the extremal length of a 
a simple closed curve $\alpha$ on the Riemann surface $X \in \cT(S)$. 
(see \eqnref{ext} for definition). We introduce a notion of
{\it extremal length} for saddle connections on quadratic differentials. 
Essentially, the extremal length of a saddle connection 
$\omega$ in a quadratic differential $(X,q) \in \QT$ with distinct end points $p_1$ 
and $p_2$ is the extremal length of the associated curve in the ramified double cover 
of $X$ with simple ramification points at only $p_1$ and $p_2$ 
(see \secref{ext:saddle} for more details). 

\begin{definition}\label{Def:basic}
For $\ep>0$ and for any quadratic differential $(X,q)\in \Qs$, let 
$\Omega_q(\ep)$ be the set of saddle connections $\omega$ 
so that either $\Ext_q(\omega) \leq \ep$ or $\omega$ appears in a geodesic 
representative of a simple closed curve $\alpha$ with $\Ext_X(\alpha) \leq \ep$. 
Let  $\Qje(\sigma)$ be the set of quadratic  differentials $(X,q)$ 
of type $\sigma$ so that $\Omega_q(\ep)$ contains at least $j$ disjoint 
homologically independent saddle connections.  When $\sigma$ is fixed, 
we denote this set simply by $\Qje$. Let $\Cje$ be the set of points in $\Stratum$
whose lift to $\QT$ lies in $\Qje$. For $0\le\theta\le1$, define $N_{\theta}(\Cje,R)$ 
to be the number of closed geodesics of length at most $R$ in $\Stratum$ that spend 
at least $\theta$--fraction of their length in $\Cje$. 
\end{definition}
In this paper, we show: 
\begin{theorem}\label{Thm:j:short}
Given $\delta> 0$, there exist $\ep > 0$ small enough and $R>0$ large enough so that,
for all $j \geq 1$ and $0\leq\theta\leq1$,
\begin{displaymath}
N_{\theta}(\Cje,R) \le e^{(h-j \theta+\delta)R}.
\end{displaymath}
\end{theorem}

\begin{remark}
The condition on $\Omega_q(\ep)$ is necessary. Just assuming there are
$j$ saddle connections of $q$--length less than $\ep$ does not 
reduce the exponent by $j$. In fact, for any $\ep$, there is a closed 
geodesic $\bg$ in $\QM$ where the number of saddle connections with 
$q$--length less than $\ep$ is as large as desired at every quadratic
differential $(X, q)$ along $\bg$. This is because the Euclidean size
of a subsurface could be as small as desired (see \secref{thinthick}) and
short saddle connection can intersect. 
\end{remark}

\subsection*{Lattice counting in \Teich space.}
Let $\Gamma(S)$ denote the mapping class group of $S$ and 
let $B_{R}(X)$ denote the ball of radius $R$ in the \Teich space with 
respect to the \Teich metric centered at the point $X \in \cT(S)$. It is known 
(\cite{ABEM}) that, for and $Y \in \T(S)$,
$$
\big| \Gamma(S) \cdot Y \cap B_{R}(X)\big| \sim \Lambda^2 \, e^{(6g-6)R},
$$
as $R \rightarrow \infty.$ Here $\Lambda$ is a constant depending only on the topology of $S$ (See \cite{Dumas:Skinning}). 

The main theorem in this paper is a partial generalization of this result for the 
strata of quadratic differentials. Here we are interested in
the case where the \Teich geodesic joining $X$ to $\gg \cdot Y$, 
for $\gg \in \Gamma(S)$, is assumed to belong to the stratum $\Qs$ or stay 
close to it.

\begin{figure}[ht]
\setlength{\unitlength}{0.01\linewidth}
\begin{picture}(100, 33)

\put(30,0){
  \begin{tikzpicture}
  [thick, 
    scale=.04\unitlength,
    vertex/.style={circle,draw,fill=black,thick,
                   inner sep=0pt,minimum size=1mm}]
   \coordinate (A) at (2,21); 
   \node (X) at (15,15) [vertex, label=above:$X$]  {};  
   \node [draw,circle through=(A),label=left:$B_{R}(X)$] at (X) {};
   \draw[dashed] (X) to node [above] {$R$}   (A); 
   \node (L) at (11,8) [label=below:$\Qs$]  {};  

   \node (X1) at (14.2,12) [vertex, label=below:$X_1$]  {};
   \node (Y1) at (23.5,10) [vertex, label=below:$Y_1$]  {};
   \node (Z) at (24.1,12.5) [vertex, label=above:{$Z = \gg \cdot Y $}]  {};
   \node (Y) at (37,10) [vertex, label=above:$Y$]  {};
   
   \draw (X1) to node [below] {$\cG$}  (Y1); 
   \draw[dotted] (X) -- (X1); 
   \draw[dotted] (Z) -- (Y1); 

   \draw[very thin] (6,13.5) .. controls (15,14) .. (27,11);
   \draw[very thin] (27,11) -- (25,6);
   \draw[very thin] (5,8) .. controls (13,8.5) .. (25,6);
   \draw[very thin] (6,13.5) -- (5,8);
      
   \end{tikzpicture}}
\end{picture}

\caption{There is a geodesic $\cG$ in $\Qs$ connecting a point near $X$ 
to a point near $Z \in B_R(X)$ that is in the orbit of $Y$.}
\label{Fig:Z}
\end{figure}
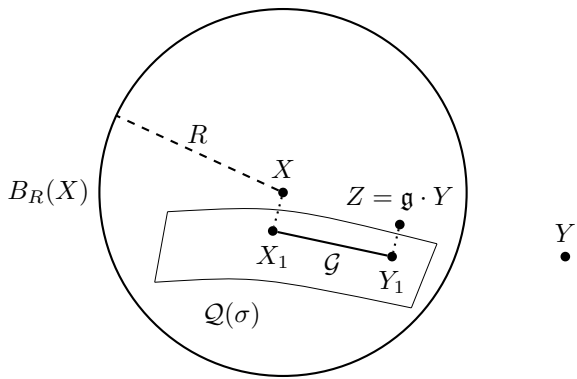

More precisely, for a fix $r_{0}>0$ (see \secref{Net}), let $N_{\theta}(\Qje, X, Y, R)$ 
be the number of points $Z \in \T(S)$ such that (See \figref{Z}):
\begin{itemize}
\item $Z \in B_{R}(X)$ and $Z = \gg \cdot Y$, for some $\gg \in \Gamma(S)$. 
\item there is a \Teich geodesic segment $\cG \subset \Qs$ joining 
$X_{1} \in B_{r_0}(X)$ to $Y_{1} \in B_{r_0}(Z)$
\item $\cG$ spends at least 
$\theta$--fraction of the time in $\Qje$. 
\end{itemize}
Also, for a fix $\ep_0$ (see \secref{Constants} below), we define $\cS_X$ to be the set of 
$\ep_0$--short curves in $X$ and 
$$
G(X) = 1+ \prod_{\beta \in \mathcal{S}_{X}}
\frac{1}{\sqrt{\Ext_X(\beta)}}.
$$

\begin{theorem}\label{Thm:j:short:XY}
Given $\delta> 0$, there exist $\ep > 0$  small enough and $R>0$ large enough 
such that, for every $0\le\theta\leq1$, $j \geq 1$ and $X, Y \in \cT(S)$,
we have
\begin{displaymath} \label{Eq:XYr}
N_\theta(\Qje, X,Y, R) \le G(X) G(Y) \, e^{(h-j\theta+\delta)R},
\end{displaymath}
\end{theorem}

Compare with Theorem $7.2$ in \cite{EM:geodesic-counting}. 

\subsection{Notes on the proof.} \label{Sec:Notes}
\subsection*{1.} Each stratum $\QMs$ has an affine integral structure, and carries 
a unique probability measure $\mu$, called the Masur-Veech measure, 
invariant by the \Teich flow which is equivalent to the Lebesgue measure. Moreover, 
the restriction of the \Teich flow to any connected component $\Stratum$ of 
$\QMs$ is mixing with respect to the Lebesgue measure class \cite{Masur1},
\cite{Veech}. In fact, the \Teich flow on $\Stratum$ is exponentially mixing with 
respect to $\mu$ \cite{Avila:Resende}, \cite{AGY:mixing}. However, we will only 
use the mixing property (as stated in \thmref{VM}) in this paper. 

\subsection*{2.} The main difficulty for proving \thmref{asymp:all:geodesics} 
is the fact that the Teichm\"uller flow is not uniformly hyperbolic. As in 
\cite{EM:geodesic-counting}, we show that the \Teich geodesic flow (or more 
precisely an associated random walk) is biased toward the part of $\Stratum$ 
that does not
contain short saddle connections (see \lemref{est:tilde:P:j}). 
Similar method has been used in \cite{EM:geodesic-counting} where it is enough 
to use Minsky's product region theorem (see \secref{extremal}) to prove 
the necessary estimates. In this paper, since the projection map from $\Stratum$ to 
$\Moduli$ is not easy to understand, we need different and more delicate methods 
to obtain similar results. 

\subsection*{3.} 
We define a notion of a 
\emph{$(q, \tau)$--regular} triangulation for a quadratic differential $(X,q)$ 
(\defref{nice:triangulation}). Such a triangulation captures the geometry
of singular Euclidean metric associated to $q$ in a way that is compatible
with the hyperbolic metric associated to $X$. We will show that
a set of disjoint saddle connections in $\Omega_q(\ep)$ can be included in a
$(q, \tau)$--regular triangulation (\lemref{Exist}).

\subsection*{4.}
In order to prove \thmref{LatticeCount} (\secref{main:counting})
we compute, given the triangulation $T_a$, the number of possible 
triangulations $T_b$ which have certain bounds on their intersection
number with $T_a$. It turns out that the number of possible triangulations
$T_b$ is related to the number of edges in $T_a$ that are
homologically independent. This is the main reason that the growth rate of 
$N_\theta(\Qje, X,Y, R)$ is related to $\dim_\reals \Stratum$. 
In \secref{geometry} we establish the basic properties of a $(q,\tau)$--regular 
triangulation and in \secref{Intersection-Bound} we establish the necessary 
bounds on the intersection number between $T_a$ and $T_b$
needed in \secref{main:counting}.

\subsection*{5.}
Theorems \ref{Thm:short:saddle} and \ref{Thm:j:short} are essentially corollaries 
of \thmref{j:short:XY}. In \secref{thin:part}, we use \thmref{LatticeCount} to prove \thmref{j:short:XY}. 
Here we describe the steps involved in the proof of 
\thmref{j:short:XY}. First, we fix a net $\cN$ in $\Moduli$ and its lift $\tilde \cN$ in $\T(S)$.  
For any constant $\tau$, we note that $N_\theta(\Qje, X,Y, R)$ is bounded above by 
the number of trajectories $\{\lambda_0, \ldots, \lambda_n \}$ in $\tilde \cN$ from 
$X$ to $Y$ so that the distance between $\lambda_i$ and $\lambda_{i+1}$ is at most 
$\tau$ and, for $\theta$ proportion of steps, the segment $[\lambda_i, \lambda_{i+1}]$ 
can be approximated by a path in $\Qje$. 

Given $\lambda_i$, we bound the number of possible choices for
$\lambda_{i+1}$ so that the segment $[\lambda_i, \lambda_{i+1}]$ 
can be approximated by a path in $\Qje$. The bound depends on the geometry of 
$\lambda_i$ (captured by the function $G(\param)$. 

On the other hand, if
$\cG \from [a,b] \to \Qje$ is a  geodesic segment with initial and terminal 
quadratic differentials $(X_a, q_a)$ and $(X_b, q_b)$ with $|b-a| \leq \tau$, one can find
a $(q_a, \tau)$--regular triangulation $T_a$ and $(q_b,\tau)$--regular
triangulation $T_b$ so that $T_a$ and $T_b$ have $j$ nomologically independent edges in common 
(See Lemma \ref{Lem:quad} for the precise statement). Then \thmref{LatticeCount} shows that
the number of choices for 
$\lambda_{i+1}$ is also reduced by a factor $e^{-j\tau}$.

\subsection*{6.} To obtain \thmref{asymp:all:geodesics}, we 
use the basic properties of the Hodge norm \cite{ABEM} to prove a 
closing lemma for the \Teich geodesic flow in \secref{closing}.
We remark that the Hodge norm behaves badly near smaller strata, i.e. near points 
with degenerating zeros of the quadratic differential, where quadratic differentials 
have small geodesic segments.

On the other hand, the set of quadratic differentials with no small geodesic 
segment is compact and 
in any compact subset of $\Stratum$, the geodesic flow is uniformly hyperbolic 
(See \cite{Veech}, \cite{forni} and \secref{Hodge}). 
Also, in view of \thmref{j:short}, for any $0 \leq \theta \leq 1$, 
the number of closed geodesics $\gamma$ of length at most $R$ such that 
$\gamma$ spends at least a $\theta$--fraction of the time outside of a compact 
subset of $\Stratum$ is exponentially smaller than $N (\Stratum,R)$. Therefore, 
"most" closed geodesics spend at least $(1-\theta)$--fraction 
of the time away from the degenerating locus. 
This allows us to prove  \thmref{asymp:all:geodesics} following the ideas from  Margulis' thesis \cite{Margulis:thesis}.

\subsection{Further remarks and references} \label{Sec:Remarks} 
\subsection*{1} According to the Nielsen-Thurston classification, every 
irreducible mapping class $\gg \in \Gamma(S)$ of
infinite order has a representative which is a pseudo-Anosov homeomorphism. 
Let $K_\gg$ denote the dilatation factor of $\gg$ \cite{Th}. 
By a theorem of Bers, every closed geodesic in $\Moduli$ is the unique loop of 
minimal length in its homotopy class. Also a pseudo-Anosov $\gg \in \Gamma(S)$ 
gives rise to a closed geodesic $\cG_{\gg}$ of length $\log(K_\gg)$ in $\QM.$
Hence $\log(K_\gg)$ is the translation length of $\gg$ as an isometry of 
$\cT(S)$ \cite{Bers}. In other words, 
$$
\cL(S)=\big\{ \log(K_\gg) \st \gg \in \Gamma(S) \quad
\text{pseudo-Anosov} \; \big\} 
$$
is the length spectrum of $\Moduli$ equipped with the \Teich metric.
By \cite{AY} and \cite{Iv}, $\cL(S)$ is a discrete subset of $\reals$. 
Hence the number of conjugacy classes of pseudo-Anosov elements of the group $
\Gamma(S)$ with dilatation factor $K_\gg \leq e^R$ is finite.
We remark that for any pseudo-Anosov $\gg \in \Gamma(S)$ the number $K_\gg$ 
is an algebraic number and $\log(K_\gg)$ is equal to the minimal topological 
entropy of any element in the same homotopy class \cite{FLP}. (See
\cite{P} and \cite{BC} for simple explicit constructions of pseudo-Anosov 
mapping classes.) In terms of this notation, $N(\Stratum,R)$ is the number of 
conjugacy classes of pseudo-Anosov elements $\gg$ in the mapping class group $
\Gamma(S)$ with expansion factor of at most $e^{R}$ such that 
$\cG_{\gg} \subset \Stratum.$

\subsection*{2}The first results on this problem are due to Veech \cite{Veech}. He
proved that there exists a constant $c$ such that
\begin{displaymath}
h \le \liminf_{R \to \infty} \frac{\log N(R)}{R} \le \limsup_{R
  \to \infty} \frac{\log N(R)}{R} \le c
\end{displaymath}
and conjectured that $c= h$.

Foliations fixed by pseudo-Anosov maps can be characterized by being
representable by eventually periodic "convergent" words \cite{PP1}. Moreover, there 
is an inequality relating the length of the repeating
part of the word corresponding to a pseudo-Anosov foliation and the
dilatation factor of a pseudo-Anosov map preserving that foliation \cite{PP2}. 
However, the estimates obtained using these inequalities are weaker. 

\subsection*{3}
The basic idea behind the proof of the main theorem in this paper is 
proving recurrence results for \Teich geodesics. Variations on this theme have been 
used in \cite{EMM}, \cite{EM}, \cite{Jayadev}, and \cite{EM:geodesic-counting}.
One reason the proof is different from \cite{EM:geodesic-counting} is 
that in general the projection map 
$\pi\from \QMs \rightarrow \Moduli$ is far from being a fibration: in many cases 
$\dim(\QMs)<\dim(\cM(S))$ and $\dim(\pi^{-1}(X) \cap \QMs)$ depends on the 
geometry of $X$. In this paper, we need to analyze the geometry of quadratic 
differentials more carefully. The results obtained in \secref{geometry}
allow us to deal with this issue. 

\subsection*{4} 
Our results are complimentary to the following result:
\begin{theorem}[Hammenstadt]
\label{Thm:rafi:hammenstadt}
There exists a compact $\cK \subset \Stratum$ such that  for $R$ sufficiently large, 
\begin{displaymath}
N(\Stratum_{\cK}, R) \ge e^{(h-1)R}.
\end{displaymath}
\end{theorem}
Also, by results in \cite{Hamenstadt:BM} the normalized geodesic flow 
invariant 
measure supported on the set of closed geodesics of length $\leq R$ in $\Stratum$ 
become equidistributetd
with respect to the Lebesgue measure $\mu$ as $R\rightarrow \infty$.

\subsection{Choosing constants} \label{Sec:Constants}  
We choose our constants as follows: We call a curve short if its extremal
length is less than $\ep_0$. This is a constant that depends on the topology 
of $S$ only (a uniform constant) and is chosen so that \thmref{PRT}
and the estimate in \eqnref{EFE} hold. We call any other constant that
depends in the topology of $S$ or the choice of $\ep_0$ a uniform constant. 
Most of these constants are hidden in notations $\emul$ and $\lmul$ (see 
the notation section below). For the arguments in \secref{thin:part} to work, 
we need to choose $\tau$ large enough depending on the value of $\delta$
(see proofs of \thmref{j:short} and \lemref{est:tilde:P:j}
in \secref{thin:part}). Then $\ep$ is chosen
small enough depending on the value of $\tau$. We need 
$\ep \le \ep_1=\ep_1(\tau)$ so that \lemref{Exist} holds and 
$\ep \leq \ep_2=\ep_{2}(\tau)$ so that \lemref{quad} holds. 
The dependence on the choice of $\tau$ and $\ep$ is always highlighted
and a constant that we call uniform does not depend on $\ep$ or $\tau$. 

\subsection{Notation} \label{Sec:Notation}
In this paper, the expression $\bA \lmul \bB$ 
means that $\bA < \const \, \bB$ and $\bA \ladd \bB$ means 
$\bA \leq \bB + \const$ for some uniform constant $\const$ which only depends 
on the topology of $S$ (a uniform constant). We write $\bA \emul \bB$ if we have both 
$\bA \lmul \bB$ and $\bB \lmul \bA$. Similarly, $\bA \eadd \bB$ if both $\bA \ladd \bC$ 
and $\bB \ladd \bA$ hold. The notation $\bA=O(\bB)$ means that $\bA \lmul \bB$.

\subsection*{Acknowledgements}
We would like to thanks the referee for many useful comments that have improve
the exposition of the paper at several places. 

\section{\Teich Space and Quadratic Differentials}
\label{Sec:back}
In this section, we recall some definitions and known results about the 
geometry of $\Moduli$ equipped with the \Teich metric. For more details, 
see \cite{Hubbard}, \cite{FarbMargalit} and \cite{Strebel}.

\subsection{\Teich space} 
Let $S$ be a connected oriented surface of genus $g$ with $p$ marked points. 
A point in the \emph{\Teich space} $\cT(S)$ is a Riemann surface $X$ of genus
$g$ with $p$ marked points equipped with a diffeomorphism 
$f \from S \rightarrow X$ sending marked points to marked points. 
The map $f$ provides a {\it marking} on $X$ by $S$. Two marked surfaces 
$f_1 \from S \rightarrow X$ and $f_2 \from S \rightarrow Y$ define the same 
point in $\cT(S)$ if and only if $f_1 \circ f_2^{-1}\from Y \rightarrow X$ is 
isotopic (relative to the marked points) to a holomorphic map. 
By the uniformization theorem, each point $X$ in $\cT(S)$ has a complete 
metric of constant curvature $-1$ with punctures at the marked points. The space 
$\cT(S)$ is a complex manifold of dimension $3g-3+p$, diffeomorphic to a cell. 
Let $\Gamma(S)$ denote the mapping class group of $S$, the group of isotopy 
classes of orientation preserving self-homeomorphisms of $S$ fixing the marked 
points point-wise. The mapping class group $\Gamma(S)$ acts on $\cT(S)$ by 
changing the marking. The quotient space 
$$\Moduli= \cT(S)/\Gamma(S)$$ 
is the moduli space of Riemann surfaces homeomorphic to $S$.

\subsection{\Teich distance and Teichm\"uller's theorem}\label{Sec:Tt}
The \Teich metric on $\cT(S)$ is defined by
$$
d_{\cT} \big( (f_1 \from S \rightarrow X_1), (f_2 \from S 
             \rightarrow X_2) \big) = \frac 12 \inf_f \log (K_f),
$$ 
where $f\from X_{1} \rightarrow X_{2}$ ranges over all quasiconformal maps 
isotopic to $f_{1} \circ f_{2}^{-1}$ and $K_f \geq 1$ is the dilatation of $f$.
For convenience, we will often omit the marking and write $X \in \T(S)$. 
To distinguish between a marked point and an un-marked point, 
we use small case letters for points in Moduli space and write $x \in \Moduli$. 

We recall the following important theorem due to Teichm\"uller. Given any 
$X_{1}, X_{2} \in \cT(S)$, there exists a unique quasi-conformal map $f$, 
called the \Teich map and quadratic differentials $(X_i,q_i) \in \cQ^{1}(X_{i})$ such 
that the map $f$ takes zeroes and poles of $q_{1}$ to zeroes and poles of $q_{2}$ 
of the same order and $d_\cT (X_{1}, X_{2}) = \frac{1}{2} \log(K_f)$.
%

\subsection{The space of quadratic differentials} \label{Sec:maps}
Let $\cQ(X)$ denote the vector space of quadratic differentials on $X$ with at most 
simple poles at the marked points of $X$. The cotangent space of $\cT(S)$ at a 
point $X$ can be identified with $\cQ(X)$ and the space 
$$
\cQ\T(S) = \Big\{(X,q) \ST X \in \T(S), \, q \in Q(X) \Big\}
$$ 
can be identified with the cotangent space of $\T(S)$. 

In local coordinates $z$, $q$ is the tensor 
given by $q(z) dz^{2}$, where $q(z)$ is a meromorphic function with poles 
of degree at most one at the punctures of $X$. In this setting, the \Teich metric 
corresponds to the norm 
$$
\parallel q \parallel_{\cT} = \int _{X} |q(z)| \, |dz|^{2}
$$
on $\cQ\T(S)$. Let $\cQ^{1}\cT(S)$ denote the
space of (marked) unit area quadratic differentials, or equivalently the
unit cotangent bundle over $\T(S)$.
Define 
$$
\cQ\Moduli\cong \cQ\T(S)/ \Gamma(S) \quad\text{and}\quad
\QM \cong \cQ^{1}\cT(S)/ \Gamma(S).
$$

To simplify the notation, in this paper, we let $\bp$ denote both projection maps 
$$
  \bp\from \cT(S) \rightarrow \cM(S),
  \qquad\text{and}\qquad \bp\from \cQ^{1}\cT(S)\rightarrow \cQ^{1}\cM(S).
$$
Similarly, $\pi$ will denote both projection maps:
$$
  \pi\from \cQ^{1}\cM(S) \rightarrow\cM(S), 
  \qquad\text{and}\qquad 
  \pi\from \cQ^{1}\cT(S) \rightarrow \cT(S).
$$

\subsection{Extremal and hyperbolic lengths of simple closed curves}
By a \emph{curve} we always mean the free homotopy class 
of a non-trivial, non-peripheral, simple closed curve on the surface $S$
where the homotopy is relative to the marked points. We denote the set of 
curves on $S$ by $\cS$ to emphasize that they are simple curves. 

Given a curve $\alpha$ on the surface $S$ and $X \in \cT(S)$, let $\ell_{X}(\alpha)$
denote the hyperbolic length of the unique geodesic in the homotopy class of $\alpha$
on $X$. The {\it extremal length} of a curve $\alpha$ on $X$
is defined by \begin{equation}\label{Eq:ext}
\Ext_X(\alpha):=\sup_{\rho}
\frac{\ell_{\rho}(\alpha)^{2}}{\operatorname{Area}(X,\rho)},
\end{equation}
where the supremum is taken over all metrics $\rho$ conformally equivalent
to $X$, and $\ell_\rho(\alpha)$ denotes the infimum of  $\rho$--lengths of 
representatives of $\alpha$. 

Here $X$ can be any Riemann surface, even an open annulus. Recall
that the modulus of an annulus $A$ is defined to 
$$
\Mod(A) := \frac 1{\Ext_A(\alpha)}, 
$$
where $\alpha$ is the core curve of $A$.

Given curves $\alpha$ and $\beta$ on $S$, the intersection
number $\I(\alpha,\beta)$ is the minimum number of points in which
representatives of $\alpha$ and $\beta$ must intersect. In general, by 
\cite{GM}
\begin{equation}\label{Eq:intersection}
  \I(\alpha, \beta) \leq \sqrt{\Ext_X(\alpha)} \cdot \sqrt{\Ext_X(\beta)}. 
\end{equation}
The following result \cite{Kerckhoff:Asm} relates the ratios of extremal lengths 
to the \Teich distance: 

\begin{theorem}[Kerckhoff] \label{Thm:ker} Given 
$X,Y \in \T(S)$, the \Teich distance between $X$ and $Y$ is given by 
$$
d_{\cT}(X,Y)=\sup_{\beta \in \cS}
   \log\left(\frac{\sqrt{\Ext_X(\beta)}}{\sqrt{\Ext_Y(\beta)}}\right).
$$ 
\end{theorem} 

The relationship between the extremal length and the hyperbolic length is
 complicated; in general, by the definition of extremal length,
$$ 
  \frac{\ell_{X}(\alpha)^{2}}{ \pi (2g-2+p)} \leq \Ext_X(\alpha).
$$
Also, for any $X \in \T(S)$, the extremal length can be extended continuously
to the space of measured laminations \cite{Kerckhoff:Asm} such that 
$$
  \Ext_X(r \cdot \lambda)=r^{2} \Ext_X(\lambda).
$$ 
As a result, since the space of projectivized measured laminations
is compact, for every $X$ there exists a constant $\CX$ so that 
$$
  \frac{1}{\CX} \ell_{X}(\alpha) \leq \sqrt{\Ext_X(\alpha)} 
  \leq \CX \, \ell_{X}(\alpha). 
$$
However, by \cite{Maskit}
\begin{equation}\label{Eq:Maskit}
 \frac{1}{\pi} \leq 
 \frac{\Ext_X(\alpha)} { \ell_{X}(\alpha)} \leq \frac{1}{2} e^{\ell_{X}(\alpha)/2}.
\end{equation}
Hence, as $\ell_{X}(\alpha) \rightarrow 0, $ 
$$
  \frac{\ell_{X}(\alpha)}{\Ext_X(\alpha)} \emul 1.
$$

\subsection{Minsky's product theorem} \label{Sec:extremal}
Let $\cA=\{\alpha_1, \ldots, \alpha_j\}$ be a collection of disjoint simple 
closed curves on $S$ and, for a fixed $\ep_0$,
$$
\cT_{\ep_0}(\cA)=
\Big\{ X \in \cT(S) \ST \Ext_X(\alpha_i) \leq \ep_0, 
\quad 1\leq i\leq j  \Big\}. 
$$
Then, using the Fenchel-Nielsen coordinates on $\cT(S)$, we can define 
$$ \Phi_\cA\from \T_{\ep_0}(\cA) \to (\half^2)^j $$
by
$$
\Phi_\cA(X)= \left( \theta_1(X), \frac{1}{\ell_X(\alpha_1)}, \ldots, 
\theta_j(X), \frac{1}{\ell_X(\alpha_j)} \right).
$$
Here, $\theta_i(\param)$ is the Fenchel-Nielson twist coordinate around 
$\alpha_i$ and represents the $x$--coordinate in upper-half plane $\half$
and the $y$--coordinate in $\half$ is the reciprocal of the hyperbolic length. 
Following Minsky, we get a map
$$
\Phi\from \cT_{\ep_0}(\cA) \rightarrow (\half^2)^j \times \cT(S \setminus \cA),
$$
where $\cT(S \setminus \cA)$ is the quotient \Teich space obtained by collapsing 
all the $\alpha_i$. The product region theorem \cite{Minsky}
states that for sufficiently small $\epsilon_0$ the \Teich metric on 
$\T_{\ep_0}(\cA)$ is within an additive constant of the supremum metric on
$(\half^2)^j \cross \cT(S \setminus \cA)$. More precisely, let 
$d_\cA(\param, \param)$ denote the supremum metric on 
$(\half^2)^j \cross \cT(S \setminus \cA)$. Then:

\begin{theorem}[Minsky]\label{Thm:PRT}
There is $\ep_0>0$ is small enough and $B > 0$ depending only on $S$ 
such that for all $X,Y \in \T_{\ep_0}(\cA)$, 
$$ 
\big| d_{\cT}(X,Y) - d_\cA \big( \Phi(X),\Phi(Y) \big) \big| < B.
$$ 
\end{theorem}
As mentioned in the introduction, we fix $\ep_0$ so that the above theorem
and the estimate in $\eqnref{EFE}$ hold. 

\subsection{Short curves on a surface} \label{Sec:bounded}
For $\ep_0$ as above, we say a curve $\alpha$ is {\it short} on $X$ if 
$\Ext_X(\alpha) \leq \epsilon_0$. From discussions in \cite{Minsky}, 
we know that, if two curves are short in $X$ they can not intersect. 
Let $\cS_X$ be the set of short curves on $X$. 
Define $G \from \cT(S) \rightarrow \reals_{+}$ by
\begin{equation}\label{Eq:GG}
G(X) = 1+ \prod_{\alpha \in \cS_X} \frac{1}{\sqrt{\Ext_X(\alpha)}}.
\end{equation}
If $d_{\cT}(X,Y)=O(1)$ then $G(X) \emul G(Y).$
The function $G$ is $\Gamma(S)$ invariant and 
induces a proper function on $\Moduli$. We also recall the following
lemma which, for example, follows from \cite{EM:geodesic-counting}.

\begin{lemma}\label{Lem:count} 
For any $X \in \T(S)$ let
$$
I_X = \big\{ \gg \in \Gamma(S) \, \big|\, d_\T(\gg \cdot X, X) = O(1) \big\}.
$$
be the set of mapping classes that move $X$ by a bounded amount. 
Then 
$$\big| I_X \big|\emul G(X)^2.$$
\end{lemma}


\subsection{Stratum of quadratic differentials} 
Although the value of $q\in \cQ(X)$ at a point in $X$ depends on the local 
coordinates, the zero set of $q$ is well defined.
As a result, there is a natural stratification of the space $\cQ\T(S)$ by 
the multiplicities of zeros of $q$. For $\sigma=(\nu_1,\dots,\nu_k,\varsigma)$
define $\cQ\T(\sigma) \subset \cQ\T(S)$ to be the subset
consisting of pairs $(X,q)$ of quadratic differentials on $X$
with zeros and poles of multiplicities $(\nu_{1},\ldots,\nu_{k})$.
The poles are always assumed to be simple and are located at the marked points, 
however, not all marked points have to be poles. The sign $\vs \in \{-1, 1\}$ 
is equal to $1$ if $q$ is the square of an abelian differential (an
abelian differential). Otherwise, $\vs=-1$. Then
$$
\cQ\T(S)=
\bigsqcup\limits_{\sigma} \cQ\T(\sigma).
$$ 
It is known that each $\cQ\T(\sigma)$ is an orbifold. See \cite{Masur1} 
and \cite{Masur:Smillie2} for more details.

\subsection{Flat lengths of simple closed curves and saddle connections.}
\label{Sec:Geodesic}
Let $(X,q)$ be a quadratic differential.  If we represent $q$ locally as 
$q(z) dz^{2}$ then $|q|=|q(z)|^{\frac 12} |dz|$ defines a singular Euclidean 
metric on $X$ with cone points at zeros and poles. The total angle at a singular 
point of degree $\nu$ is $(2+\nu)\pi$. (for more details, see \cite{Strebel}).
This is not a complete metric space since poles are a finite distance away.
However, one can still talk about the \emph{geodesic representative} of a curve 
that may \emph{pass through} the poles even though the poles. 
Namely, for a arc in $(X,q)$, consider the lift of this arc to the universal cover,
take the geodesic representative in the completion of the universal cover
and then project it back to $(X,q)$.  Following the discussion in \cite[Page 185]{Rafi:SC},
we can ignore this issue and treat these special geodesics as we would any 
other geodesic.

The homotopy class of an arc  (relative to its endpoints) has a
unique $q$--geodesic representative. Any curve $\alpha$ either has a unique 
$q$--geodesic representative or there is flat cylinder of parallel representatives. 
In this case, we say $\alpha$ is a \emph{cylinder curve} and we denote the 
cylinder of geodesics representatives of $\alpha$ by $F_\alpha$. We denote the 
Euclidean length of the $q$--representative of 
$\alpha$ by $\ell_q(\alpha)$. 

A {\it saddle connection} on $(X,q)$ is a $q$--geodesic segment which 
connects a pair of singular points without passing through one in its interior.
We denote the Euclidean length of a saddle connection $\omega$ on $q$ by 
$\ell_{q}(\omega)$.

\subsection{Period coordinates on the strata.} \label{Sec:Period}
 In general, any saddle connection $\omega$ joining two zeros of
a quadratic differential $q=\zeta dz^{2}$ determines a complex number 
$\hol_{q}(\omega)$ (after choosing a branch of $\sqrt \zeta$ and
an orientation of $\omega$) by
$$
\hol_q(\omega)=
\left(\int\limits_{\omega} \Re \sqrt \zeta \right) + 
\left(\int\limits_{\omega} \II \sqrt \zeta \right) i.
$$
We recall that for any $\sigma= (\nu_i, \ldots, \nu_k, \vs)$ the period coordinates 
gives $\cQ\T(\sigma)$ the structure of an affine manifold.  
Consider the first relative homology group $H_{1}(S, \Sigma, \reals)$ of the pair 
$(S, \Sigma)$ with $|\Sigma|=k.$  Let 
$$
h=(2g+k-1)=\dim\big(H_1(S, \Sigma, \reals)\big)
$$ 
if $\vs=1,$ and 
$$
h=(2g+k-2)=\dim\big(H_1(S, \Sigma, \reals)\big) -1
$$ 
if $\vs=-1.$ We recall that given $(X,q) \in \QTs$ 
there is a triangulation $T$ of the underlying surface by saddle connections
(see for example \cite[Proposition 3.1]{Vorobets} and  \cite[Proposition 3.1]{Thurston}). 
One can choose  $h$ directed edges $\omega_1,\ldots,\omega_h$ of $T$, 
and an open neighborhood $U_q \subset \cQ\T(\sigma)$ of $q$ 
such that the map 
$$
  {\boldsymbol \psi}_{T,q} \from \cQ\T(\sigma) \rightarrow \cx^h
\qquad\text{defined by}\qquad
  {\boldsymbol \psi}_{T,q}(q)= \big( \hol_{q}(\omega_{i}) \big)_{i=1}^{h}
$$
is a local homeomorphism. For any other geodesic triangulation $T'$, the map 
${\boldsymbol \psi}_{T',q}\circ {\boldsymbol \psi}_{T,q}^{-1}$ is linear.

In case of abelian differentials ($\vs=1$) it is enough to choose a basis for 
$H_{1}(S, \Sigma, \reals)$ from the edges of $T$.
Note that for non-orientable differentials ($\vs=-1$) there will be a linear relation 
between the holonomies of the vectors corresponding to a basis for the relative 
homology (see \secref{Relations}). In this case, it is enough to choose 
$\dim(H_{1}(S, \Sigma,  \reals))-1$ independent vectors of $T$.
For a more detailed discussion of the holonomy coordinates see \cite{Masur:Smillie1}. 

\subsection{\Teich geodesic flow}
We recall that when $3g+p>4$ the \Teich metric is not even Riemannian. 
However, geodesics in this metric are well understood. 
A quadratic differential $(X,q) \in \QT$ with zeros at $p_{1},\ldots p_{k}$ 
is determined by an atlas of charts mapping open subsets of 
$S -\{p_{1},\ldots,p_{k}\}$ to $\reals^{2}$ such that the change of 
coordinates are of the form $v \rightarrow \pm{v}+c.$ Therefore the group 
$\SL(2,\reals)$ acts naturally on $\QT$ by acting on 
the corresponding atlas; given $A \in \SL(2,\reals)$, 
$A \cdot q \in \QT$ is determined by the new atlas $\{ A \phi_{i}\}.$
The action of the diagonal subgroup
$g_{t}=\begin{bmatrix} e^{t}&0\\
0 & e^{-t} \end{bmatrix}$
is the \Teich geodesic flow for the \Teich metric.
In other words, in holonomy coordinates the \Teich flow is simply defined by 
$$ 
\Re \big(\hol_{g_{t}(q)}(\omega_{i})\big)=e^{t} \, 
  \Re\big(\hol_{q}(\omega_{i})\big),
$$ 
and 
$$ 
\Im\big(\hol_{g_{t}q}(\omega_{i})\big) =e^{-t} \, 
  \Im \big(\hol_{q}(\omega_{i})\big).
$$

This action descends to $\QM$ via the projection map 
$\bp \from \QT \to \QM$. We denote both actions
(on $\QT$ and $\QM$)  by $g_{t}$. The subspaces $\QTs$ and $\QMs$ 
are invariant under the \Teich geodesic flow. Moreover, we have 
(\cite{Veech}, \cite{Masur1}):

\begin{theorem}[Veech-Masur] \label{Thm:VM}
Each connected component $\Stratum$ of a stratum 
$\cQ^{1}\mathcal{M}(\sigma)$ carries a unique probability measure 
$\mu$  in the Lebesgue measure class such that:
\begin{itemize}
\item the action of $\SL(2,\reals)$ is volume preserving and ergodic;
\item \Teich geodesic flow is mixing.
\end{itemize}
\end{theorem}

\section{Geometry of a quadratic differential}\label{Sec:geometry}

In this section, we recall some of the basic geometric properties of a 
quadratic differential $(X,q)$. We describe how the extremal length of a curve, 
which can be calculated from the conformal structure of $X$, relates
to the singular Euclidean metric associated to $(X,q)$. We also
define the notion of a $(q, \tau)$--regular triangulation, where $\tau>0$ is a large
constant. This is a partial triangulation of $(X,q)$ using the saddle connections 
that captures the geometry of the singular Euclidean metric associated to $q$. 
The main statement of the section is \lemref{Exist} which shows the existence 
of such triangulations. In the rest of the section, we establish some basic properties 
of $(q, \tau)$--regular triangulations which are used in section 5. 

\subsection{Intersection number}
In the hyperbolic metric of $X$, the geodesic representatives of any
two curves $\alpha$ and $\beta$ intersect minimally. Hence, the geometric
intersection number between homotopy classes of curves is equal to 
the intersection number between their geodesic representatives. 

In the singular Euclidean metric $|q|$, this is not true.  First, as mentioned in 
\ref{Sec:Geodesic}, the geodesic representative might \emph{pass through} the poles 
even though the poles are removed from the surface.
Also, the $q$--geodesic representatives of curves $\alpha$
and $\beta$ that have geometric intersection number zero may intersect. 
However, these intersections are \emph{tangential}. That is, 
$\alpha$ and $\beta$ may share an edge, but they do not \emph{cross}.
By this, we mean that any lifts $\tilde \alpha$ and $\tilde \beta$ to
the universal cover $\tilde q$ of $q$ have end points in the boundary
that do no interlock. To simplify the exposition, when we say $\alpha$ and $\beta$ 
intersect, we always mean that they have an essential intersection not tangential. 

We also talk about the intersection number between two saddle connections. 
Here, we say two saddle connections are \emph{disjoint} if they have
disjoint interiors or if they are equal. The intersection number between two saddle 
connections is the number of interior intersection points. The intersection number
between a saddle connection and itself is zero. In both cases, 
(saddle connections and curves) the intersection number is denoted by 
$\I(\param, \param)$. 

If $A$ is an embedded annulus, we distinguish between a curve $\alpha$ 
\emph{intersecting} $A$ and \emph{crossing} it. To intersect $A$, $\alpha$ needs 
only to enter the interior of $A$. The curve $\alpha$ crosses $A$ if
$\alpha$ enters one side of $A$ and exits the other. To be more precise,
in the annular cover $\tilde X_A$ of $X$ associated to $A$, there is a
lift of $\alpha$ connecting the two boundary components of $\tilde X_A$. 

\subsection{Extremal lengths and flat lengths of simple closed curves.} 
One can give an estimate for the extremal length of a simple closed curve 
$\alpha$ in $X$ by examining the singular Euclidean metric $|q|$. 
As mentioned before, $\alpha$ may not have a unique geodesic representative;
different geodesic representatives of $\alpha$ are parallel and foliate 
a flat cylinder that we refer to as $F_\alpha$.  Denote the two boundary
curves of $F_\alpha$ by $\alpha_E$ and $\alpha_G$. When 
$F_\alpha$ is degenerate, $\alpha_E = \alpha_G$.

We say an annulus is regular if its boundary curves are equidistant. 
Let $E_\alpha$ be the largest embedded regular annulus with boundary curve
$\alpha_E$ and let $G_\alpha$ be the largest embedded regular annulus with 
boundary curve $\alpha_G$. Note that $E_\alpha$ and $G_\alpha$ may
intersect $F_\alpha$ and each other. In a degenerate case, the interior of some 
or all of these annuli could be empty, for example, the interior of $F_\alpha$
is empty when $\alpha$ has a unique geodesic representative.

We call $\alpha_E$, the shared boundary of $E_\alpha$ and $F_\alpha$, the 
inner boundary of $E_\alpha$ (and similarly $\alpha_G$ is the inner boundary of 
$G_\alpha$). The annuli $E_\alpha$ 
and $G_\alpha$ are called \emph{expanding} because the equidistance curves 
parallel to the inner boundary get longer as they span $E_\alpha$ and $G_\alpha$. 
Let $l= \ell_q(\alpha)$ and let $e,f$ and $g$ be the $q$--distances 
between the boundaries of $E_\alpha$, $F_\alpha$ and $G_\alpha$
respectively. According to \cite{Rafi:HT}, when $\Ext_X(\alpha) \leq \ep_0$, 
(see \secref{Constants} for the discussion of the choice of $\ep_0$) we have the 
following estimates
\begin{equation} \label{Eq:EFE}
\frac 1{\Ext_X(\alpha)} \emul \Mod(E_\alpha)+ \Mod(F_\alpha)+ \Mod(G_\alpha)
\end{equation} 
where
\begin{equation} \label{Eq:Moduli} 
\Mod(E_\alpha) \emul \Log \frac el
\qquad  \Mod(F_\alpha)= \frac fl,
\qquad\text{and}\qquad \Mod(G_\alpha) \emul \Log \frac gl . 
\end{equation}
Here $\Log(\param)$ is a modified logarithm function:
$$
\Log(t) = \max \big\{ \log (t), 1 \big\}. 
$$
We intend $\Log$ to apply only to large numbers. 
Of course, the value of either $e$, $f$ or $g$ could be zero and the second
line will be $-\infty$. We use the modified logarithm to avoid this issue. 

Note that, a simple closed curve that has a short flat length may not have a small 
extremal length. We need to measure what is the largest neighborhood of 
$\alpha$ that still has a simple topology. Later, we use this idea to define
a notion of extremal length for a saddle connection. 

\subsection{Short simple closed curves}
As in \secref{bounded}, we say a curve
$\alpha$ is {\it short} in $q$ if $\Ext_X(\alpha) \leq \ep_0$.  
Denote the set of short curves in $q$ by $\cS_q$. We say $\alpha$ is a 
\emph{cylinder curve} if the interior of $F_{\alpha}$ is not empty. 
In what follows, the cases when $\alpha \in \cS_q$ is a cylinder curve and 
$F_\alpha$ has a large enough modulus will need special treatments. 
When the modulus of $F_\alpha$ is extremely small, $\alpha$ behaves
essentially like a non-cylinder curve. We make this precise:

\begin{definition} \label{Def:c-n}
Let $\tau$ be a positive real number and let $\CNM = e^{-2\tau}$. 
We say a curve $\alpha \in \cS_q$ is a \emph{large-cylinder curve} if 
$\Mod(F_\alpha) \geq \CNM$. Denote the set of large-cylinder curves by 
$\Sqc$ and define 
$$\Sqn = \cS_q \setminus \Sqc.$$ 
For $\alpha \in \Sqc$, the size $s_\alpha$ of $F_\alpha$ is defined to be
the distance between the boundaries of $F_\alpha$. 
\end{definition}

\begin{remark}\label{Rem:tau}
The constant $\tau$, which is determined in \secref{thin:part}, 
is the distance between steps of a random walk trajectory. 
We use $\CNM$ instead of just writing $e^{-2\tau}$ to 
highlight the fact that $\CNM$ is a bound for modulus. 
There is an implicit assumption that $\tau$ is \emph{large enough}
(say, $\tau \ge \tau_0$ for some uniform constant $\tau_0$).
That is, unless otherwise stated, all statements hold with uniform constants 
independent of $\tau$ as long as $\tau\ge \tau_0$.  
\end{remark}

Along \Teich geodesics, the length of a curve $\alpha \in \Sqn$ changes 
slowly while the modulus of $F_\alpha$ remains small. 
More precisely, let 
$$(X_t,q_t)= g_t(X,q),$$
where $g_t$ is the \Teich geodesic flow. Assuming $\alpha \in \Sqn$ and 
$0\leq t \leq \tau$, we have 
$\Mod_{q_t}(F_\alpha) \lmul 1$.  As a consequence of 
Equations~\eqref{Eq:EFE} and \eqref{Eq:Moduli}, $\Mod_{q_t}(G_\alpha)$
and $\Mod_{q_t}(E_\alpha)$ change at most linearly and we have
\begin{equation}\label{Eq:ext:curve}
 \frac1{\Ext_X(\alpha)}- t \lmul \frac 1{\Ext_{X_t} (\alpha)}
 \lmul  \frac 1{\Ext_X(\alpha)} + t.
\end{equation}

\subsection{The thick-thin decomposition of quadratic differentials} 
 \label{Sec:thinthick}
We call the components of $S \setminus \cS_q$ the {\it thick subsurfaces}
of $q$. The homotopy class of each such subsurface $Q$ of $S$ has a 
representative with $q$--geodesic boundaries. There is, in fact, a unique 
such representative that is disjoint from the interior of cylinders associated 
to the boundary curves of $Q$. This can also be described
as the smallest representative of $Q$ with $q$--geodesic boundaries. 
We denote this subsurface by $Q$ as well.
Define the {\it size} $s_Q$ of $Q$ to be the $q$--diameter of this representative.
The following theorem states that the geometry of the subsurface $Q$ is 
essentially the same as that of the thick hyperbolic subsurface of $X$ 
in the homotopy class of $Q$ but scaled down to a size $s_Q$: 

\begin{theorem}[\cite{Rafi:TT}] \label{Thm:TT}
For every essential closed curve $\gamma$ in $Q$, 
$$
\ell_X(\gamma) \emul \sqrt{\Ext_X(\gamma)} \emul \frac{\ell_q(\gamma)}{s_Q}.
$$
In particular, the $q$--length of shortest essential curve in $Q$ 
is on the order of $s_Q$. 
\end{theorem}

\begin{example} \label{Exmp:TT}
A quadratic differential can be described as a singular flat structure of a surface
plus a choice of a vertical direction. For example, the surface obtained from the 
polygon in \figref{(X,q)} with the given edge identifications is a once punctured 
genus 
$2$ surface.  Assume that the edges $2,3,5$ and $6$ have a comparable
lengths, the edge $1$ is significantly shorter and the edge $4$ is significantly 
longer than the others. Choose an arbitrary vertical direction and let $(X,q)$ be 
the associated quadratic differential. 

\begin{figure}[ht]
\setlength{\unitlength}{0.01\linewidth}
\begin{picture}(100, 22)
\put(3,3){\includegraphics[width=40\unitlength]{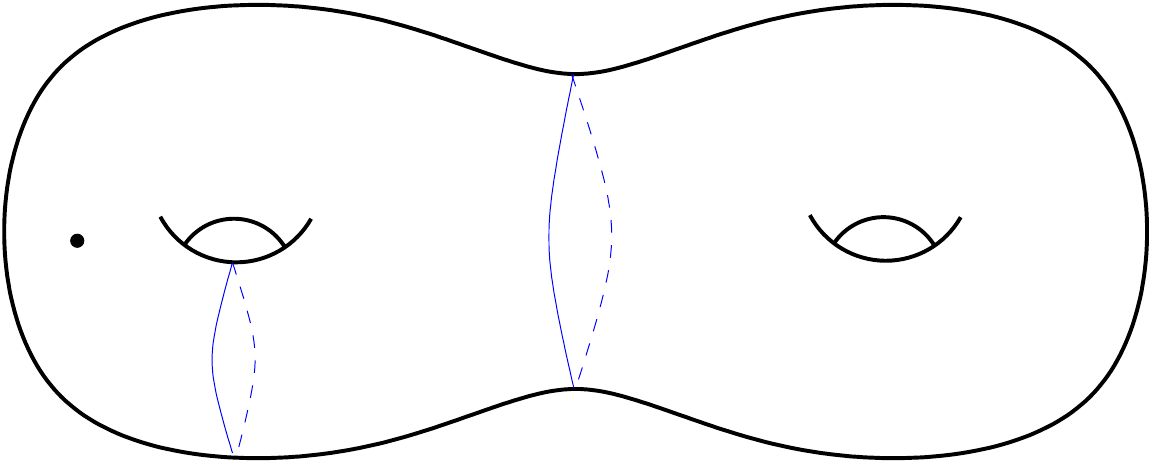}}
\put(20,11){$\beta$}
\put(8,6){$\alpha$}

\put(48,0){
\begin{tikzpicture}
  [thick, 
    scale=.04\unitlength,
    vertex/.style={circle,draw,fill=black,thick,
                   inner sep=0pt,minimum size=1mm}]

   \node[vertex] (l0) at (0,0)   {};
   \node[vertex] (l1) at (0,1)   {};
   \node[vertex] (l2) at (0,7)   {};
   \node[vertex] (l3) at (0,13)  {};
   \node[vertex] (m0) at (10,0)  {};
   \node[vertex] (m1) at (10,1)  {};
   \node[vertex] (m2) at (10,13) {};
   \node[vertex] (r0) at (40,1)  {};
   \node[vertex] (r1) at (40,3)  {};
   \node[vertex] (r2) at (40,7)  {};
   \node[vertex] (r3) at (40,9)  {};
   \node[vertex] (r4) at (40,13) {};

   \draw (l0) -- (l1) node at (-1, -.1) {1};
   \draw (l1) -- (l2) node[midway,left] {2};
   \draw (l2) -- (l3) node[midway,left] {2};
   \draw (l3) -- (m2) node[midway,above] {3};
   \draw (m2) -- (r4) node[midway,above] {4};
   \draw (r4) -- (r3) node[midway,right] {5};
   \draw (r3) -- (r2) node[midway,right] {6};
   \draw (r2) -- (r1) node[midway,right] {5};
   \draw (r1) -- (r0) node[midway,right] {6};
   \draw (r0) -- (m1) node[midway,below] {4};
   \draw (m1) -- (m0) node at (11,-.1) {1};
   \draw (m0) -- (l0) node[midway,below] {3};
       
   \draw[draw=blue]  (0,3) arc (90:0:2) -- (2,0)
            node at (2.5,3) {$\alpha$};
   
   \draw [draw=blue] (0,11) arc (270:360:2);    
   
      \draw [draw=blue] (18,13) -- (18,1)
         node[midway,right] {$\beta$};    
\end{tikzpicture}}
\end{picture}

\caption{Quadratic differential $(X,q)$ and short curves of $X$.}
\label{Fig:(X,q)}
\end{figure}

Then the hyperbolic metric on $X$ has two short simple closed curves; 
$\cS_q = \{ \alpha, \beta \}$. The curve $\beta$ is a cylinder curve and has a 
small extremal length because the flat annulus $F_\beta$ (\figref{EFG}) has a 
large modulus. In fact, $\beta$ is a large-cylinder curve ($\Sqc= \{\beta \}$).
The curve $\alpha$ is a non-cylinder curve and it has a small extremal length 
because the expanding annuli $E_\alpha$ and $G_\alpha$ (\figref{EFG}) have 
large moduli ($\Sqn= \{\alpha \}$). Note that the $q$--geodesic representative 
of $\alpha$ is the saddle connection $1$ (the end points of arc $1$ are identified). 

\begin{figure}[ht]
\setlength{\unitlength}{0.01\linewidth}
\begin{picture}(100, 18) 

\put(1,0){
\begin{tikzpicture}
   [thick, 
    scale=.04\unitlength,
    vertex/.style={circle,draw,fill=black,thick,
                   inner sep=0pt,minimum size=1mm}]

   \filldraw[fill=purple!70!white, draw=black] 
   (0,0)  -- (6,0) -- (6,1) arc (0:90:6) -- cycle;
   
   \filldraw[fill=purple!70!white, draw=black] 
   (0,13) -- (6,13) arc (360:270:6) -- cycle;
   
   \filldraw[fill=blue!70!white, draw=black] 
   (10,0) -- (10,1) -- (16,1) arc (0:180:6) -- (4,0) -- cycle;
   
   \filldraw[fill=blue!70!white, draw=black] 
   (10,13) -- (16,13) arc (360:180:6) -- cycle;

   \node[vertex] (l0) at (0,0)   {};
   \node[vertex] (l1) at (0,1)   {};
   \node[vertex] (l2) at (0,7)   {};
   \node[vertex] (l3) at (0,13)  {};
   \node[vertex] (m0) at (10,0)  {};
   \node[vertex] (m1) at (10,1)  {};
   \node[vertex] (m2) at (10,13) {};
   \node[vertex] (r0) at (40,1)  {};
   \node[vertex] (r1) at (40,3)  {};
   \node[vertex] (r2) at (40,7)  {};
   \node[vertex] (r3) at (40,9)  {};
   \node[vertex] (r4) at (40,13) {};

   \draw (l0) -- (l1);
   \draw (l1) -- (l2); 
   \draw (l2) -- (l3); 
   \draw (l3) -- (m2); 
   \draw (m2) -- (r4); 
   \draw (r4) -- (r3); 
   \draw (r3) -- (r2); 
   \draw (r2) -- (r1); 
   \draw (r1) -- (r0); 
   \draw (r0) -- (m1); 
   \draw (m1) -- (m0); 
   \draw (m0) -- (l0); 
      
\end{tikzpicture}}

\put(51,0){
\begin{tikzpicture}
   [thick, 
    scale=.04\unitlength,
    vertex/.style={circle,draw,fill=black,thick,
                   inner sep=0pt,minimum size=1mm}]

   \filldraw[fill=purple!70!white, draw=black] 
   (10,1) -- (40,1)  -- (40,13) -- (10,13) -- cycle;  

   \node[vertex] (l0) at (0,0)   {};
   \node[vertex] (l1) at (0,1)   {};
   \node[vertex] (l2) at (0,7)   {};
   \node[vertex] (l3) at (0,13)  {};
   \node[vertex] (m0) at (10,0)  {};
   \node[vertex] (m1) at (10,1)  {};
   \node[vertex] (m2) at (10,13) {};
   \node[vertex] (r0) at (40,1)  {};
   \node[vertex] (r1) at (40,3)  {};
   \node[vertex] (r2) at (40,7)  {};
   \node[vertex] (r3) at (40,9)  {};
   \node[vertex] (r4) at (40,13) {};

   \draw (l0) -- (l1);
   \draw (l1) -- (l2); 
   \draw (l2) -- (l3); 
   \draw (l3) -- (m2); 
   \draw (m2) -- (r4); 
   \draw (r4) -- (r3); 
   \draw (r3) -- (r2); 
   \draw (r2) -- (r1); 
   \draw (r1) -- (r0); 
   \draw (r0) -- (m1); 
   \draw (m1) -- (m0); 
   \draw (m0) -- (l0); 
   
 \end{tikzpicture}
 } 
 
 \put(4,17){$E_\alpha$} 
 \put(15,17){$G_\alpha$} 
  \put(59, 8){$F_\beta$} 
 
 \end{picture}
\caption{The maximal expanding annuli $E_\alpha$ and $G_\alpha$ and the
maximal flat annulus $F_\beta$.}
\label{Fig:EFG} 
\end{figure}
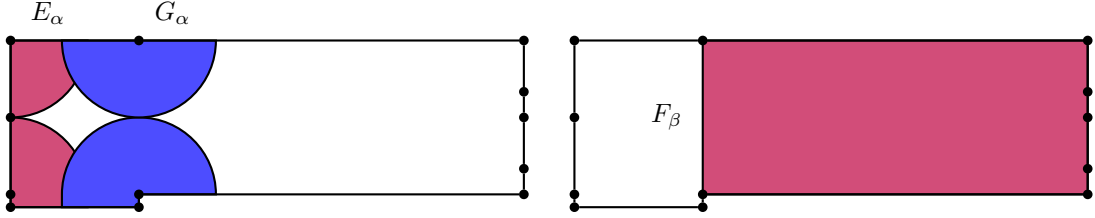
\end{example}

There are two thick subsurfaces. There is a once punctured torus with a
boundary curve $\beta$ whose $q$--representative is degenerate and is
represented in $q$ with a graph of area zero (the union of arcs $5$ and
$6$). The other is a pair of pants whose boundaries consist of two copies
of $\alpha$ and one copy of $\beta$. The maximal expanding annuli
$E_\alpha$ and $G_\alpha$ do not necessarily stay inside of the
$q$--representative of this pair of pants and they may overlap. 

The size of a thick subsurface $Q$ is related to the radii of annuli $E_\alpha$,
$F_\alpha$ and $G_\alpha$ for every boundary curve $\alpha$. We make a few 
observations that will be useful later.

\begin{lemma} \label{Lem:size:radius} 
Let $Q$ be a thick subsurface of $(X,q)$, $\alpha$ be
a boundary component of $Q$ and $E_\alpha$ be 
the expanding annulus in the direction of $Q$. Using the
notation of \eqnref{EFE} we have
\begin{enumerate} 
\item $l \leq 2 s_Q$. 
\item $e \leq s_Q$. 
\item $\max(e,f,g) \geq \ell_q(\alpha)$. 
\item $(e+ l )\emul s_Q$.
\item If $\Mod(E_\alpha) \gmul 1$ then $e \emul s_Q$. 
\end{enumerate}
\end{lemma}

\begin{proof}
Since $\alpha$ is part of $Q$, its length is less than twice the diameter of
$Q$ which is the first assertion. To see part two, note that if
$e$ is larger than $s_Q$, then $Q$ is contained in $E_\alpha$ which
is an annulus. This is a contradiction. Part (3) follows from \eqnref{Moduli}
and the fact that $\alpha$ is $\ep_0$--short. Parts (1) and (2) imply $(e+ l)\lmul s_Q$. 
Hence, to prove part (4), we need to show $(e+ l)\gmul s_Q$.

Since $E_\alpha$ is maximal, its outer boundary self-intersects. 
Let $\gamma$ be the curve constructed as a concatenation 
a sub arc of $\alpha$ and two arcs connecting $\alpha$ to the boundary 
points of $E_\alpha$ associated to the self intersection of $E_\alpha$.  
Note that the inner boundary of $E_\alpha$ is a geodesic and its outer boundary
has positive curvature, therefore, the interior of $E_\alpha$ is convex,
and the curve $\gamma$ must be essential. 

Then $l+e \gmul \ell_q(\gamma)$. If $\gamma$ is contained in 
$Q$ and is essential in $Q$, then $\Ext_X(\gamma) \gmul 1$ ($Q$ is a thick subsurface). 
From \thmref{TT} we get, 
$$
\frac{\ell_q(\gamma)}{s_Q} \gmul 1
\quad\text{and hence}\quad 
(e+l) \gmul s_Q.
$$
If $\gamma$ is not contained in $Q$, we show that there exists a closed
curve $\gamma'$ in $Q$ whose length is not much longer than $\gamma$. 

Assume that $\gamma$ exists $Q$ by intersecting a boundary curve $\alpha'$ 
and returns via a boundary curve $\alpha''$ ($\alpha'$ and $\alpha''$ 
maybe the same curve). By part (3),  $\max(e',f',g')$ is larger than $l'$, 
$\max(e'',f'',g'')$ is larger than $l''$ and $\ell_q(\gamma)$ is larger than both. 
There is a sub-arc $\omega$ of $\gamma$ connecting $\alpha'$ to $\alpha''$,
in particular, $\ell_q(\omega) \leq \ell_q(\gamma)$. If $\alpha' \not = \alpha''$, let 
$\gamma'$ be the curve obtained as a concatenation of two copies of $\omega$ and 
a copy of $\alpha'$ and $\alpha''$ each. This curve is essential in $Q$ unless $Q$ is
a pair of pants, in which case, we take $\gamma'$ to be the curve that
wraps around $\alpha'$ twice. If $\alpha' = \alpha''$, then let $\gamma'$
be the curve obtained as a concatenation of $\omega$ and a sub-arc of $\alpha'$. 
Again, this curve is essential in $Q$ unless $Q$ is a pair of pants, in which 
case, we take $\gamma'$ to be the curve that wraps around $\alpha'$ twice. 
The curve $\gamma'$ resides in $Q$ and 
$\ell_q(\gamma') \lmul \ell_q(\gamma)$. We have
\[
(e+l) \gmul \ell_q(\gamma) \gmul \ell_q(\gamma') \gmul s_Q.
\]
To see part (5), we note that, if 
$$
\Log \frac el \emul \Mod(E_\alpha) \gmul 1
\quad\text{then}\quad 
e\emul (e+l).
$$ 
Now, part (5) follows from part (4). 
%
%
%

\end{proof}

As a corollary we get the following analogue of the collar lemma: 

\begin{corollary} \label{Cor:CollarLemma}
Let $\alpha \in \cS_q$ be the boundary of a thick subsurface $Q$ and let 
$\gamma$ be any curve crossing $\alpha$. Then
$$
\ell_q(\gamma) \gmul s_Q.
$$
\end{corollary}

\begin{proof}
We have $\ell_q(\gamma) \geq \max(e,f,g)$ and by part (3) of \lemref{size:radius}, 
$\max(e,f,g) \geq l$. Hence, $\ell_q(\gamma) \gmul (e+l )$. The corollary now follows 
from part (4) of \lemref{size:radius}.
\end{proof}

\subsection{Extremal lengths and flat lengths of saddle connections}
\label{Sec:ext:saddle}

As mentioned above, we can also define a notion of extremal length for
saddle connections. Let $\omega$ be a saddle connection
connecting two distinct critical points in $(X,q)$. Let $E_\omega$ be the 
annulus obtained by taking the largest regular neighborhood of $\omega$ that 
is still a topological disk and then cutting a slit open along $\omega$. 
Let $l=\ell_q(\omega)$ and $e$ be the radius of $E_\omega$ (the $q$--distance
between $\omega$ and the boundary of $E_\omega$). Then, we define
(the second inequality follows from \eqnref{Moduli}) 
$$
\Ext_q(\omega) :=\frac 1{\Log (e/l)} \emul \frac 1{\Mod(E_\omega)}. 
$$
Another interpretation of this notion of extremal length, that would provide 
roughly the same result, is to compute the extremal length in a ramified double cover 
of $(X,q)$. Denote the end points of $\omega$ by $p_1$ to $p_2$. There exists 
a unique ramified double cover 
$\phi\from X_\omega \rightarrow X$ with simple ramification points at only $p_{1}$ 
and $p_{2}$. Note that $\alpha_\omega=\phi^{-1} \omega$ is a simple 
closed curve on $X_\omega$. 

\begin{lemma} \label{Lem:EXT}
If $\Ext_q(\omega) \leq \ep_0$, then
$$
\Ext_{X_\omega}( \alpha_\omega) \emul \Ext_q(\omega).
$$
\end{lemma}

\begin{proof}
Let $q_\omega$ be the lift of $q$ to $X_\omega$. 
Note that $\alpha_\omega$ has a unique geodesic representative in $q_\omega$ 
($\Mod(F_{\alpha_\omega})=0$) and $E_{\alpha_\omega}$ and 
$G_{ \alpha_\omega}$ are conformally equivalent to $E_\omega$. Hence,
by \eqnref{EFE}
\begin{equation*} 
\frac 1{\Ext_{X_\omega}( \alpha_\omega)} \emul 
\Mod(E_{ \alpha_\omega}) + \Mod(G_{\alpha_\omega}) 
=  2 \Mod(E_\omega) \emul \frac 1{\Ext_q(\omega)}.  \qedhere
\end{equation*}
\end{proof}

Since $l$ and $e$ change at most exponentially fast along a \Teich
geodesic, similar to \eqnref{ext:curve}, for $q_t=g_t(q)$ we have
\begin{equation}\label{Eq:ext:saddle}
 \frac 1{\Ext_q(\omega)} -t \lmul \frac 1{\Ext_{q_t}(\omega)}
   \lmul  \frac 1{\Ext_q(\omega)} +t.
\end{equation}

\begin{definition} \label{Def:Omega_q}
For any $0<\ep \leq \ep_0$, let $\Omega_q(\ep)$ be the set of saddle connections 
$\omega$ of $q$ so that, either
\begin{itemize} 
\item $\Ext_q(\omega) \leq \ep$, or
\item $\omega$ lies on a geodesic representative for $\alpha$ with 
$\Ext_X(\alpha) \leq \ep$. 
\end{itemize}
Later in the text, we will add further restrictions on the value of 
$\ep$ depending on $\tau$ (see \lemref{Exist} and \lemref{quad}). 
We note however that, in all the proofs, making $\ep$ smaller or making 
$\tau$ larger does not effect the constants involved in any of our estimates. 
\end{definition}

In general, knowing $\ell_q(\omega)$ is small does not imply that
$\omega$ has a small extremal length. However, we have the following
lemma which is enough to show that \thmref{short:saddle} 
follows from \thmref{j:short}. 

\begin{lemma}\label{Lem:ext:saddle}
Assume that $(X,q)$ has a saddle connection $\omega$ with 
$\ell_q(\omega) \ll 1$. Then, either 
$$
\frac 1{\Ext_q(\omega)} \emul \Log \frac 1{\ell_q(\omega)} \quad\text{or}\quad
\frac 1{\Ext_X(\alpha)} \gmul \Log \frac 1{\ell_q(\omega)},
$$
for some simple closed curve $\alpha$. In particular, $\Omega_q(\ep)$ is 
non-empty. 
\end{lemma}

\begin{proof}
Let $l= \ell_q(\omega)$ and $e$ be the radius of $E_\omega$. Since
the boundary of $E_\omega$ self intersects ($E_\omega$ is maximal), there 
is a simple closed curve $\gamma$,
obtained by a concatenation of a sub arc of $\omega$ and two arcs connecting
$\omega$ to the boundary of $E_\omega$, with $\ell_q(\gamma) \lmul  (e+l)$. 

Assume first that $\cS_q$ is empty. Then, $\ell_q(\gamma) \gmul 1$. 
Since, $e \lmul 1$, we have 
\[
\frac el \emul \frac {(e+l)}l \emul \frac 1{\ell_q(\omega)}
\qquad\text{and}\qquad
\frac 1{\Ext_q(\omega)} \emul \Log \frac{e}{l} \emul \Log \frac 1{\ell_q(\omega)}.
\]
That is, the first inequality holds. Otherwise, we show that, there is a curve 
$\alpha_1 \in \cS_q$ with $\ell_q(\alpha_1) \lmul (e+l)$. This is because, either 
$\gamma \in \cS_q$ and we can take $\alpha_1=\gamma$ or $\gamma$ intersects
a thick subsurface $Q$ in which case we let $\alpha_1$ be any boundary
component of $Q$. Using \corref{CollarLemma} and part one of 
\lemref{size:radius}, we get:
$$
(e+l) \gmul \ell_q(\gamma) \gmul s_Q \gmul \ell_q(\alpha_1). 
$$

Since the total area of $q$ is 1, there is always a thick subsurface of size 
comparable to 1. Let $Q_1, \ldots, Q_k$ be a sequence of distinct subsurfaces of 
sizes $s_1, \ldots, s_k$ respectively, where $\alpha_1$ is a boundary component 
of $Q_1$, $Q_{i-1}$ and $Q_i$ share a boundary curve $\alpha_i$ and 
$s_k \emul 1$. Let $l_i = \ell_q(\alpha_i)$ and let $s_0=l_1$. 

Consider $G_{\alpha_i}$, the expanding annulus with inner boundary $\alpha_i$
in the direction of $Q_i$ with radius $g_i$. For $i\geq 1$, part (4) of
\lemref{size:radius} implies, $(g_i+l_i) \emul s_i$ and by part (1)
$s_{i-1} \gmul l_i$. Hence, from \eqnref{EFE}, we know that 
$$
 \frac 1{\Ext_q(\alpha_i)} \gmul 
 \max \left( \Log \frac{g_i}{l_i}, 1\right)
 \gadd \Log \frac{g_i+l_i}{l_i} \gadd \Log \frac{s_i}{s_{i-1}}.
$$ 
That is, the common boundary curve of two surfaces of very different
size has a very small extremal length. Also, (recall that $s_0=l_1 \lmul (e+l)$):
\begin{equation} \label{Eq:s_i}
\left(\prod_{i=1}^k  \frac{s_i}{s_{i-1}}\right) \frac{e+l}{l}
\emul \left( \frac 1{s_0} \right) \frac{e+l}{l} \gmul \frac 1{l}. 
\end{equation}
Here, the maximum value of $k$ depends only on the topology of $S$. 
Therefore, taking the logarithm of both sides of \eqnref{s_i}, we conclude that either 
$$
  \text{there is some $i$ where,}\quad
  \frac 1{\Ext_q(\alpha_i)} \gmul \Log \frac 1l 
  \quad\text{or}\quad 
  \Log \frac{e+l}{l} \gmul  \Log \frac 1l.
$$ 
In the first case, the lemma holds for $\alpha=\alpha_i$. 
In second case,
\begin{equation*}
\frac 1{\Ext_q(\omega)} \emul \Log \frac{e}{l} 
     \emul \Log \frac{e+l}{l} \emul  \Log \frac 1{\ell_q(\omega)}. \qedhere
\end{equation*}
\end{proof}
\begin{remark} Note that in both \lemref{EXT} and \lemref{ext:saddle}
the implied constants only depend on the topology of $S.$
\end{remark}

\subsection{A $(q, \tau)$--regular triangulation}\label{Sec:triangulation}
We would like to mark a quadratic differential $q$ by a triangulation 
where the edges have a bounded length. However, the notion of having a bounded
length should depend on which thick subsurface we are in. That is, we would like
the $q$--length of an edge to not be longer than the size of the thick subsurfaces
it intersects. The complication comes from the fact that a saddle connection may 
intersect several thick subsurfaces of various sizes. 

Also,  as mentioned before, large-cylinder curves will require a 
special treatment. Hence, we triangulate only the complement of large-cylinders.
Recall that two saddle connections are said to be disjoint
if they have disjoint interiors but they may share one or two end points. 
\begin{definition}\label{Def:nice:triangulation}
Let $(X,q)$ be a quadratic differential. 
Given a cylinder curve $\alpha$, let $\upsilon_\alpha$ be an
arc connecting the boundaries of $F_\alpha$ that is perpendicular to $\alpha$. 
By a \emph{$(q, \tau)$--regular triangulation} 
$T$ of $q$ we mean a collection of disjoint saddle connections 
satisfying the following conditions:
\begin{enumerate}
\item  For $\alpha \in \Sqc$, denote the interior of a cylinder $F_\alpha$ 
by $F_\alpha^\circ$. Then, $T$ is disjoint from $F_\alpha^\circ$ and 
it triangulates their complement 
$$
q \setminus \bigcup_{\textstyle \alpha \in \Sqc} F_\alpha^\circ.
$$
That is, the complement of $T$ is a union of triangles and large-cylinders
$F_\alpha^\circ$, $\alpha \in \Sqc$. In particular, $T$ contains the boundaries 
of $F_\alpha$.
\item If an edge $\omega$ of $T$ intersects a thick subsurface $Q$ of $q$
then $\ell_q(\omega)\lmul s_Q$.
\item If $\alpha$ is a cylinder curve in $\Sqn$ then $\upsilon_\alpha$ intersects 
$T$ a uniformly bounded number of times. 

\end{enumerate}
\end{definition}

We shall see that condition $3$ means that the triangulation $T$ does not twist 
around short simple closed curves.

\begin{remark}
It is important to choose the implied constants in conditions $2$ and $3$ in 
\defref{nice:triangulation} large enough 
so that every quadratic differential $q$ has a $(q, \tau)$--regular triangulation. 
In fact, we choose the constants so that the key \lemref{Exist} below holds.
\end{remark}

\begin{lemma} \label{Lem:Exist} 
For every $\tau$ there is $\ep_1(\tau)$ so that for $\ep<\ep_1(\tau)$ the following holds.
Let $\Omega$ be a subset of $\Omega_q(\ep)$ consisting of pairwise
disjoint saddle connections. Then $\Omega$ can be extended 
to a $(q, \tau)$--regular triangulation $T$. 
\end{lemma}

\begin{proof}
We would like to triangulate each thick piece Q separately and 
let $T$ be the union of these triangulations. However, saddle connections
in $\Omega$ may intersect a boundary curve $\alpha$ of $Q$. 
To remedy this, we perturb $\alpha$ slightly to a curve $\alphabar$ 
that is a union of saddle connections, lies in a small neighborhood of 
$\alpha$ and is disjoint from $\Omega$ (see Claim 1). 
These curves divide the surface into subsurfaces
with nearly geodesic boundaries. We denote the surface associated to $Q$ with 
$\Qbar$. We then extend $\Omega$ to a triangulation in each $\Qbar$ so that 
the edge lengths are not much longer than the diameter of $\Qbar$
which is comparable to $s_Q$ (see Claim 3) and let $T$ be the union of these 
triangulations. However, one needs to be careful that $\Qbar$ does not 
intersect any subsurface of size much smaller that $s_Q$, otherwise
the resulting triangulation would not be $(q, \tau)$--regular. 

\subsection*{Claim 1:} For every $\alpha \in \cS_q$, there is a
representative $\alphabar$ of $\alpha$ that is a union of saddle connections,
lies in a $(\ell_q(\alpha)/2)$--neighborhood of $\alpha$ and is disjoint from 
$\Omega$. For $\alpha, \beta \in \cS_q$, $\alphabar$ and $\betabar$ do not 
intersect. Furthermore, if $\alpha$ is a boundary of $Q$ then $\alphabar$ intersects
only surfaces that are larger than $Q$, namely, if $\alphabar$ intersects 
a thick subsurface $Q'$ we have:
$$
s_{Q'} \gmul s_Q.
$$

\begin{proof}[Proof of Claim 1:]
Let $\alpha \in \cS_q$ be a common boundary of thick subsurfaces $Q$ and 
$R$. Recall that $\CNM = 2^{-2\tau}$. If $\Mod(F_\alpha) \geq \CNM$, 
we can choose $\ep_1$ small enough to ensure that $\alpha$ is disjoint 
from $\Omega$. This is because, if $\omega$ is part of a short curve $\alpha'$, 
then $\omega$ is disjoint from $\alpha$ because short curves $\alpha$
and $\alpha'$ do not intersect. Otherwise, $\omega$ has to satisfy
the first assumption in \defref{Omega_q}. But, $F_\alpha$ does not contain any 
singular points and any arc $\omega \in \Omega$ intersecting $\alpha$ has to cross 
$F_\alpha$. Therefore, $\ell_q(\omega) \geq f_\alpha$ ($f_\alpha$ is the distance 
between the boundaries of $F_\alpha$) and, for the radius $e_\omega$ of $E_\omega$, 
we have $e_\omega \leq \ell_q(\alpha)$ (otherwise $\alpha$ would be
contained in $E_\omega$). But $\Mod(F_\alpha) = \frac{f_\alpha}{\ell_q(\alpha)}\geq \CNM$ and thus (the second inequality follows from \eqnref{Moduli}) 
$$
\frac 1{\ep_1} \leq \frac 1{\Ext_q(\omega)} 
\emul \Log \frac {e_\omega} {\ell_q(\omega)} 
\leq \Log \frac{\ell_q(\alpha)}{f_\alpha}
\leq \Log \frac 1\CNM = 2\tau.
$$
which is not possible if $\ep_1$ is chosen to be small enough. 
To summarize, if $\Mod(F_\alpha) \geq \CNM$, 
then $\alpha$ is already disjoint from $\Omega$, we can take $\alphabar=\alpha$.

If $\Mod(F_\alpha) \leq \CNM$, then either $E_\alpha$ or $G_\alpha$ has a large 
modulus. The annulus with the larger modulus is in the direction of the
thick surface with the larger size (\lemref{size:radius}). Assume $E_\alpha$, 
the annulus in the direction of $Q$, has a large modulus. Let $e_\alpha$ be the distance 
between the boundaries of $E_\alpha$. By part (5) of \lemref{size:radius} and the 
previous assumption we have 
$$
e_\alpha \emul s_Q \geq s_R.
$$

Denote the $(\ell_q(\alpha)/2)$--neighborhood of $\alpha$ in $E_\alpha$ with 
$\Ebar_\alpha$. The annulus $E_\alpha$ may not be contained entirely in $Q$ 
and may intersect some thick subsurfaces with very small size. But $\Ebar_\alpha$
does not intersect any small subsurfaces. To see this, assume $Q'$ intersects 
$\Ebar_\alpha$. Since $Q'$ is disjoint from $\alpha$, it has to enter $E_\alpha$ 
intersecting the outer boundary of $E_\alpha$. But $e_\alpha$ is much larger than 
$\ell_q(\alpha)$, and hence:
$$
s_{Q'} \gmul \ell_q(\partial Q') > e-\ell_q(\alpha)/2 \gmul s_Q. 
$$
Thus, the last condition of the claim is satisfied as long as $\alphabar$
stays in $\Ebar_\alpha$. 

Note that no arc in $\Omega$ can cross $\Ebar_\alpha$ (intersect both boundaries). 
This is because, if $\omega$ is an arc in a curve $\beta \in \cS_q$, then it does
not intersect $\alpha$ since $\beta$ and $\alpha$ have intersection number zero. 
Otherwise, $\Ext_q(\omega)$ is small, which implies that its length is
much less than the injectivity radius of any point along $\omega$. 
But the injectivity radius of any point in $\Ebar_\alpha$ is less that 
$2\ell_q(\alpha)$. Hence, (by choosing $\ep$ small enough) 
$\ell_q(\omega)$ is less than the distance between the boundaries of 
$\Ebar_\alpha$ with is equal to $\ell_q(\alpha)/2$. 

Consider the union  of $\alpha$ and the set $\Omega_\alpha$ of arcs in 
$\Omega$ that intersect $\alpha$. The convex hull $H_\alpha$ of this set in 
$\Ebar_\alpha$ is an annulus (perhaps degenerate). 
We observe that the interior of $H_\alpha$ does not
contain any singular points. Otherwise, there would be a geodesic quadrilateral,
where two edges are subsegments of arcs in $\Omega_\alpha$ and one edge is 
a subsegment of $\alpha$, that contains a singular point in its interior. But this violates
the Gauss-Bonnet theorem. Let $\alphabar$ be the boundary component of $H_\alpha$
that is not $\alpha$. Then $\alphabar$ is in the homotopy class of $\alpha$ and lies inside 
$\Ebar_\alpha$. Also, because the interior $H_\alpha$ does not contain any
singular points, $\alphabar$ is disjoint from every saddle connection in $\Omega$. 
Furthermore, by the triangle inequality, any  saddle connection $\omegabar$ 
that appears in $\alphabar$ has a $q$--length less than or equal to $2\ell_q(\alpha)$. 

It remains to show that for $\alpha, \beta \in \cS_q$, $\alphabar$ and $\betabar$ 
are disjoint. Assume $\ell_q(\beta) \geq \ell_q(\alpha)$. Then, 
$\alpha$ is disjoint from $\Ebar_\beta$, otherwise, $\alpha$ would be
contained in $E_\beta$ which is an annulus an does not contain any curve
non-homotopic to $\beta$. This means $\alpha$ is disjoint from $\betabar$
which is contained in $\Ebar_\beta$. Also, since $H_\beta$ contains 
no singular points, if a saddle connection $\omega \in \Omega_\alpha$ intersects 
$\betabar$ then  it also intersects $\beta$. But then $\omega$ is in $\Omega_\beta$ 
and hence it is disjoint from $\betabar$. Therefore, $\betabar$ is disjoint from 
the convex hull $H_\alpha$ and thus also from $\alphabar$. 
This finishes the proof of claim 1. 
\end{proof}

Next, let $\overline{\Omega}$ be the set of edges that appear in curves 
$\alphabar$ for every $\alpha \in \cS$. We have shown that saddle connections 
in $\overline{\Omega}$ are disjoint from those in $\Omega$. After removing the 
interiors of large cylinders from the quadratic differential $(X,q)$ and cutting along 
curves $\alphabar$, $\alpha \in \cS_q$, we obtains a collection of subsurfaces with 
nearly geodesic boundaries. Denote the representative of a thick subsurface 
$Q$ that is disjoint from curves $\alphabar$ by $\Qbar$. 

For each $\alpha \in \Sqn$, if $F_\alpha^\circ$ is disjoint from every saddle 
connection in $\Omega \cup \overline{\Omega}$, we choose a saddle connection 
$\omega_\alpha$ that crosses $F_\alpha$, is disjoint from $\nu_\alpha$ 
(does not twists around $\alpha$). In particular, $\omega_\alpha$ is disjoint from 
every saddle connection $\Omega \cup \overline{\Omega}$ and has a length 
that is comparable with $\ell_q(\alpha)$. Let $\Omega^{n}$ denote 
the set of such saddle connections $\omega_\alpha$. 

\subsection*{Claim 2:} Saddle connections in 
$$
T_{0}=\Omega \cup \overline{\Omega}\cup \Omega^n
$$ 
satisfying conditions (2-3) of \defref{nice:triangulation}. 

\begin{proof}[Proof of Claim 2:]
All the conditions follow
immediately from the construction, but the argument is long since
we have to look at all the cases. 
We have already shown that these edges satisfy condition (1) and
arcs in $\alphabar$ satisfy condition (2). To see that an arc 
$\omega \in \Omega$ satisfies condition (2) note that if it did not, 
$\omega$ would intersect a thick subsurface $Q$ with 
$\ell_q(\omega) \geq s_Q$. The radius of $E_\omega$ is much larger than 
length of $\omega$ ($\log \frac{e_\omega}{l_\omega} \geq \frac 1{\ep_1}$), 
which implies $E_\omega$ contains $Q$. This is a contradiction.
 
We show that arcs in $T_0$ satisfy condition (3). Namely, if $\omega \in \Omega$
intersects a cylinder $F_\alpha$, we need to show that $\omega$ intersects
$\upsilon_\alpha$ a bounded number of times. In fact, if they
intersect more than once, then $\ell_q(\omega) \ge \ell_q(\alpha)$. 
But then $E_\omega$ would contain the curve $\alpha$ which is a 
contradiction ($E_\omega$ is a topological disk). Also, the curve
$\alphabar$ is a convex hull of the union of the curve $\alpha$ which
is disjoint from $F_\alpha$ and a bounded number of arcs in $\Omega$, 
each of which intersect $\upsilon_\alpha$ at most once. Hence $\alphabar$ 
intersects $\upsilon_\alpha$ at most a bounded number of times and
thus arcs in $\alphabar$ satisfy condition (3).  
 
Since, for every $\alpha \in \Sqn$, there is a saddle connection in $T_0$
crossing $F_\alpha$, any triangulation containing $T_0$ is guaranteed 
to satisfy the condition (3). 
\end{proof}

In the next claim, we describe how to add the remaining edges to $T_0$ 
while still satisfying conditions (1) and (2). 

\subsection*{Claim 3:} A partial triangulation of $\Qbar$ where the
length of edges are less than a fixed multiple of $s_Q$ can be 
extended to a triangulation using saddle connections of length less than a 
larger fixed multiple of $s_Q$.

\begin{proof}[Proof of claim 3:] 
We prove the claim by induction. Start by cutting $\Qbar$ along the
given edges. Each cutting increases the diameter by at most
twice the length of edge being cut. Hence, in the end, we have
several components each with diameter comparable to $s_Q$.
If all components are triangles, we are done. 
Otherwise, some component contains a saddle connection that is not part of 
its boundaries or the given triangulation, the shortest such saddle connection has 
a length less than the diameter of the component it is in, which is comparable to $s_Q$
(again, see \cite[Proposition 3.1]{Vorobets}). The claim follows from the fact that 
this process ends after a uniformly bounded number of times. The diameter grows 
at most multiplicatively each time but still it is uniformly bounded multiple of $s_Q$ in 
the end. We choose the constant in the second condition of a $(q, \tau)$--regular
triangulation large enough so that the outcome of this algorithm is
in fact a $(q, \tau)$--regular triangulation. 
\end{proof}

The triangulation $T$ is now defined to be the union of all the saddle connections 
in $T_{0}$ and those coming from claim $3$. The newly added edges in $\Qbar$
have a $q$--length less than a fixed multiple of $s_Q$ and, for any thick subsurface $R$ 
that $\Qbar$ intersects, we have $s_Q \lmul s_R$. Hence, the condition (2) in 
\defref{nice:triangulation} is satisfied. Therefore, the resulting triangulation $T$ is 
$(q, \tau)$--regular. 
\end{proof}
 
 \subsection{Twisting and extremal lengths}\label{Sec:Twist}
In this section we define several notions of twisting and discuss how
they relate to each other. This is essentially the definition introduced
by Minsky extended to a slightly more general setting. We denote the relative 
twisting of two objects or structures around a curve $\alpha$ by 
$\twist_\alpha(\param, \param)$. This is often only coarsely defined, 
that is, the value of $\twist_\alpha(\param, \param)$ is determined
up to a uniformly bounded additive error. 

In the simplest case, let $A$ be an annulus with
core curve $\alpha$ and let $\beta$ and $\gamma$ be homotopy classes
of arcs connecting the boundaries of $A$ (here, homotopy is relative
to the end points of an arc). The relative twisting of $\beta$ and $\gamma$
around $\alpha$, $\twist_\alpha(\beta,\gamma)$, is defined to be the 
geometric intersection number between $\beta$ and $\gamma$. 

Now consider a more general case where $\alpha$ is a curve on the surface $S$
and $\beta$ and $\gamma$ are two transverse curves to $\alpha$.
Let $\tilde S_\alpha$ be the annular cover of $S$ associated to $\alpha$
and denote the core curve of $\tilde S_\alpha$ again by $\alpha$.
Let $\tilde \beta$ and $\tilde \gamma$ be the lifts of $\beta$ and
$\gamma$ to $\tilde S_\alpha$ (respectively) that connect the boundaries of 
$\tilde S_\alpha$.  Note that freely homotopic curves lift to arcs that
are homotopic relative their endpoints. The arc $\tilde \beta$ is not
uniquely defined, however any pair of lifts are disjoint. 
We now define 
$$
\twist_\alpha(\beta,\gamma) = \twist_\alpha(\tilde \beta,\tilde \gamma),
$$ 
using the previous case. This is well defined up to an additive error of 2
(see \cite{Minsky}).

We can generalize this further and define twisting between any two structures 
on $S$ as long as the structures in question provide a (nearly) canonical choice 
of a homotopy class of an arc $\tilde \beta$ connecting the boundaries of 
$\tilde S_\alpha$. Then we say the given structure defines a notion of 
\emph{zero twisting} around $\alpha$. The relative twisting between
two structures is the relative twisting between the associated arcs in 
$\tilde S_\alpha$. Here are a few examples:

\begin{itemize}
\item Let $X$ be a Riemann surface. Then $\tilde \beta$ can be taken to be 
the geodesic in $\tilde X_\alpha$ that is perpendicular to $\alpha$ in the
Poincare metric of $\tilde X_\alpha$. Alternatively, we can pick a
shortest curve $\beta$ transverse to $\alpha$ and let $\tilde \beta$ be the
lift of $\beta$ that connects the boundaries of $\tilde X_\alpha$.
In any case, the choice of $\tilde \beta$ is not unique, but any two such
transverse arcs have bounded geometric intersection number (see \cite{Minsky})
and the associated relative twisting $\twist_\alpha(\param, X)$ is well defined 
up to an additive error.
\item Let $q$ be a quadratic differential. As before, $\tilde \beta$ can be taken to be 
the geodesic in $\tilde q_\alpha$ that is perpendicular to $\alpha$ in the
Euclidean metric coming from $q$ or a lift of a $q$--shortest
curve $\beta$  transverse to $\alpha$ (see \cite{Rafi:minima}).
We denote the associated relative twisting with  $\twist_\alpha(\param, q)$.
\item Let $T$ be a $(q, \tau)$--regular triangulation of $(X,q)$ and
$\alpha \in \Sqn$. Then we can choose a curve $\beta$ transverse to $\alpha$
that is carried by $T$ and has a bounded combinatorial length in $T$
and let the lift of $\beta$ to the annular cover of $\alpha$ define zero twisting.
Since curves with bounded combinatorial length intersect a bounded
number of times, the associated relative twisting $\twist_\alpha(\param, T)$,  
is again well defined up to an additive error.  
\end{itemize}

The expression ``fix a notion of zero twisting around $\alpha$'' for a curve 
$\alpha$ in $S$ means ``choose a homotopy class of arcs connecting the 
boundaries of $\tilde S_\alpha$.''

\subsection{Intersection and twisting estimates} 
In this section we establish some statements relating Extremal length,  
twisting and intersection number. We start with a theorem of Minsky
giving an estimate for the extremal length of a curve. For a $X \in \T(S)$, 
let $\cS_X$ be a set of $\ep_0$--short simple closed curves in $X$. There is a 
uniform constant $B$ depending on $\ep_0$ and the topology of $S$ so that, 
for every $X$, any curve $\beta$ not in $\cS_X$ intersects a curve $\gamma$ with 
$\Ext_X(\gamma) \leq B$. That is, the curves with 
extremal length at most $B$ fill every complementary component of $\cS_X$. 
Let $\cB_X$ be the set of curves with extremal length at most $B$. 

\begin{theorem}\emph{(Minsky, \cite[Theorem 5.1]{Minsky})}
\label{Thm:minsky:ext}
Given $X \in \T(S)$ and a simple closed curve $\gamma \not \in \cS_X$, 
\begin{equation}\label{Eq:es}
\Ext_X(\gamma)  \emul 
\max_{\alpha \in \cS_X} \I(\gamma,\alpha)^2 \left[\frac{1}{\Ext_X(\alpha)}+
\twist_\alpha^{2}(\gamma,X)\Ext_X(\alpha) \right]
+ \max_{\alpha \in \cB_X} \I(\gamma, \alpha)^2.
\end{equation}
The multiplicative constant depends only on the topology of $S$.
\end{theorem}

It follows from the definition of twisting and elementary hyperbolic geometry that
if $\twist_\alpha(\beta, X)$ is large (that is, if $\beta$ twists around $\alpha$ 
a lot), then $\Ext_X(\beta) \gmul \Ext_X(\alpha)$. 

\begin{corollary} \label{Cor:the:curve}
For every curve $\gamma$ and any $X\in \cT(S)$, 
there is a curve $\beta$ so that,
$$
\sqrt{\Ext_X(\gamma) \Ext_X(\beta)} \lmul \I(\gamma, \beta)
\qquad\text{and}\qquad
\twist_\gamma(X, \beta) =O(1).
$$
\end{corollary}

Note that the reverse of first inequality always holds (\eqnref{intersection}). 

\begin{proof}
If $\gamma \in \cS_X$, then we choose $\beta$ to be a curve that
intersects $\gamma$ once or twice, is disjoint from other curves
in $\cS_X$, where $\twist_\gamma(\beta, X)$ is bounded and  where
$\I(\beta, \alpha)=O(1)$ for $\alpha \in \cB_X$. Applying \eqnref{es} 
to $\beta$ we have $\Ext_X(\beta) \emul \frac 1{\Ext_X(\gamma)}$ which 
implies that the corollary holds for $\beta$ and $\gamma$. 

If $\gamma$ is not short in $X$, \thmref{minsky:ext} applies to $\gamma$. 
Since the number of elements in $\cS_X$ and $\cB_X$ is uniformly bounded, 
$\Ext_X(\gamma)$ is comparable to one the following terms:
$$
\frac{\I(\gamma, \alpha)^2}{\Ext_X(\alpha)}, \quad
\I(\gamma, \alpha)^2 \twist_\alpha^{2}(\gamma,X)\Ext_X(\alpha)
\quad\text{or}\quad
\I(\gamma, \alpha)^2. 
$$
In the fist two cases $\alpha \in \cS_X$ and in the third case $\alpha \in \cB_X$.
We argue in 3 cases. 

If $\Ext_X(\gamma) \lmul \frac{\I(\gamma, \alpha)^2}{\Ext_X(\alpha)}$,
for $\alpha \in \cS_X$, then the corollary holds for $\beta = \alpha$
(the second conclusion follows from the fact that the twisting number of
a short curve around a long curve is uniformly bounded). 

In the second case, we take $\beta$ to be a curve transverse 
to $\alpha$ with (see above) $\Ext_X(\beta) \emul \frac 1{\Ext_X(\alpha)}$ 
and $\twist_\alpha(\beta,X) =O(1)$. In particular 
\begin{equation} \label{Eq:twists-same}
\twist_\alpha(\gamma, X) \eadd \twist_\alpha(\gamma, \beta). 
\end{equation}
The curve $\gamma$ also intersects $\alpha$ and hence 
$\Ext_X(\gamma) \gmul \frac 1{\Ext_X(\alpha)} \emul \Ext_X(\beta)$. 
Thus, $\beta$ twist around $\gamma$ at most a uniformly bounded number of times. 
Also, every strand of  $\gamma$ intersecting $\alpha$ intersects $\beta$ at least 
$\twist_\alpha(\gamma, \beta)$ times (up to an additive error). In this case
$\twist_\alpha(\gamma, \beta)$ is large and the additive error can be replaced
by a multiplicative error to obtain
\begin{equation} \label{Eq:int-times-twist}
\I(\gamma, \alpha) \twist_\alpha(\gamma, \beta)
\lmul \I(\gamma, \beta).
\end{equation}
Therefore,   
\begin{align*}
\Ext_X(\gamma) &\lmul 
\I(\gamma, \alpha)^2 \twist_\alpha^{2}(\gamma,X)\Ext_X(\alpha)
\tag{\text{Assumption on $\gamma$}}\\
&\lmul \frac{\I(\gamma, \alpha)^2 \twist_\alpha^2(\gamma, \beta)}
{\Ext_X(\beta)} \tag{\text{\eqnref{twists-same}}}\\
&\lmul \frac{\I(\gamma, \beta)^2}{\Ext_X(\beta)},
\tag{\text{\eqnref{int-times-twist}}}
\end{align*}
which implies the corollary.

The last case is when $\alpha \in \cB_X$ and 
$\Ext_X(\gamma) \lmul \I(\gamma, \alpha)^2$. In this case, we take 
$\beta= \alpha$. Since $\beta$ has bounded length in $X$, 
$$
\twist_\gamma(\beta, X)=O(1)
\qquad\text{and}\qquad 
\Ext_X(\beta) \emul 1.
$$ 
Again, the corollary follows. 
\end{proof}

We also recall the following lemma (\cite[Theorem 4.3]{Rafi:CM}):

\begin{lemma}[Rafi] \label{Lem:Twist-q-X}
For a quadratic differential $(X,q)$ and a Riemann surface $Y \in \T(S)$ with
$d_\T(X, Y) =O(1)$, we have
$$
\twist_\alpha(Y, q) \lmul \frac1{\Ext_X(\alpha)}. 
$$
\end{lemma} 

\subsection{Geometry of quadratic differentials and $(q, \tau)$--regular 
triangulations} As we mentioned at the beginning of the section, a 
$(q, \tau)$--regular triangulation is supposed to capture the geometry of $q$. 
We make this explicit in the following two lemmas. In \lemref{Length-Int}, we 
relate the length of a saddle connection to its intersection number with a 
$(q, \tau)$--regular triangulation. \lemref{Twist-q-T} shows that the notion of 
zero twisting coming from $q$ or $T$ is the same. These are used
to prove \lemref{up:to:twisting} but more essentially they are needed in 
\secref{Intersection-Bound}.

\begin{lemma} \label{Lem:Length-Int} 
Let $T$ be a $(q, \tau)$--regular triangulation and $\omega_T$ be an edge of
$T$. Let $s$ be the minimum of $s_Q$ where $Q$ is a thick subsurface of $q$
that intersects $\omega_T$. Let $\omega$ be any 
other saddle connection in $q$ so that, for every curve $\alpha \in \Sqn$,
$\twist_\alpha(\omega, q) =O(1)$. Then
$$
\I(\omega_T, \omega) \lmul \frac{\ell_q(\omega)}{s} +1.
$$
\end{lemma}

\begin{proof}
Condition (2) in the definition of a $(q, \tau)$--regular triangulation implies that
$\ell_q(\omega_T) \lmul s$. It is sufficient to prove the lemma for a
subsegment of $\omega_T$ with a $q$--length less than $s/7$, because 
$\omega_T$ can be covered but uniformly bounded number of such segments. 
Hence, without loss of generality, we assume $\ell_q(\omega_T) \leq s/7$.

Consider the $s/7$--neighborhood $N$ of $\omega_T$. Then 
$\omega \cap N$ has at most $O\left( \frac{\ell_q(\omega)}s \right)$ components. 
Hence, it is sufficient to show, for every component $\bar \omega$ of
$\omega \cap N$, that 
$$
\I(\omega_T, \bar \omega)=O(1).
$$

\begin{figure}[ht]
\setlength{\unitlength}{0.01\linewidth}
\begin{picture}(50, 42)
\put(0,0){\includegraphics[width=44\unitlength]{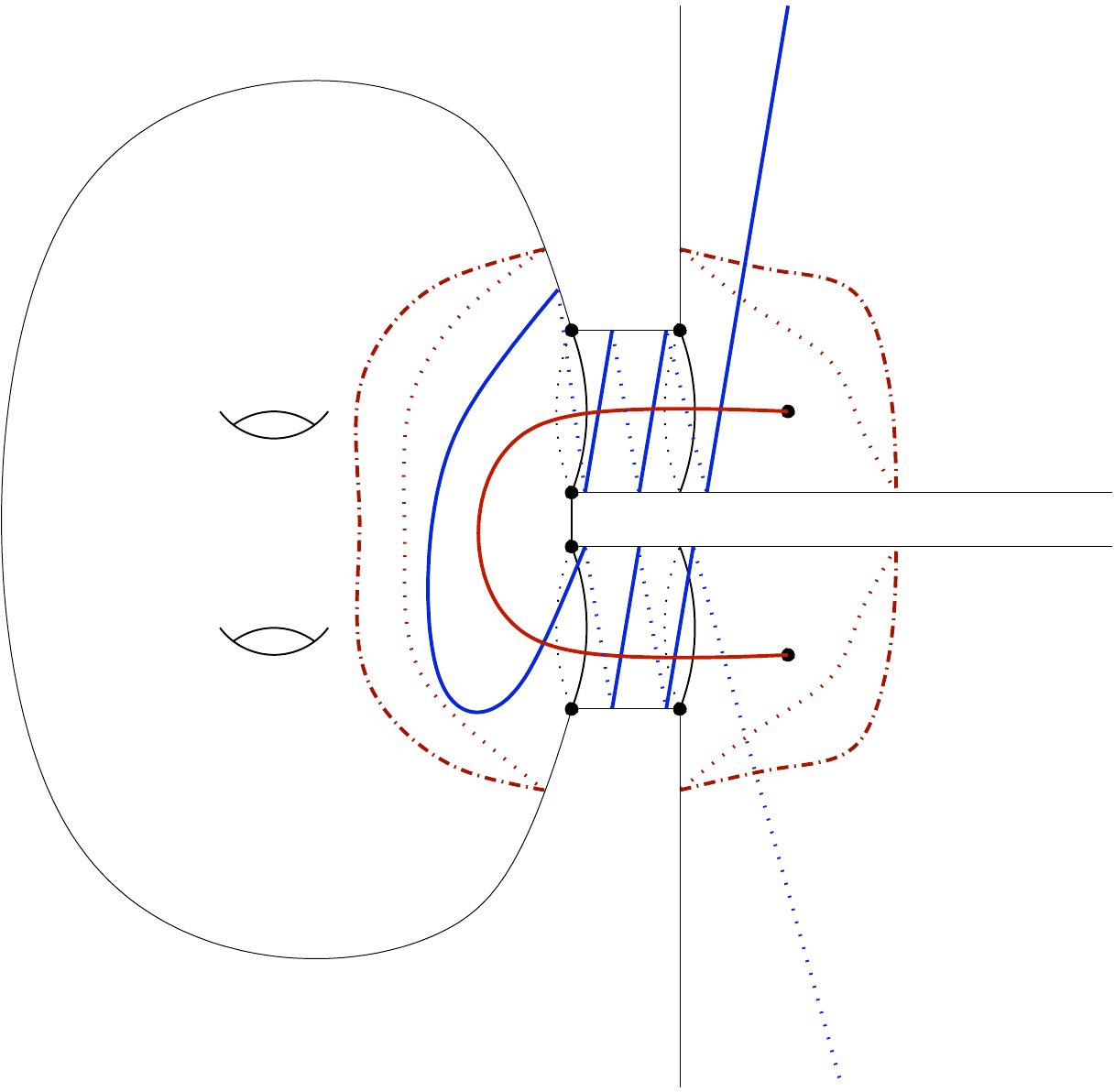}}
\put(0,5){$Q_1$}
\put(45,5){$Q_2$}
\put(45,35){$Q_3$}
\put(22.5,33){$F_{\alpha_1}$}
\put(22.5,12){$F_{\alpha_2}$}
\put(30,15){$\omega_T$}
\put(32,38){$\omega$}
\put(36,29){$N$}
\end{picture}
\caption{The arc $\omega_T$ intersects curves $\alpha_1, \alpha_2 \in \Sqn$
and thick subsurfaces $Q_1, Q_2$ and $Q_3$. Each component $\bar \omega$ of 
$\omega \cap N$ intersects $\omega_T$ only a bounded number of times outside of 
cylinders $F_{\alpha_1}$ and $F_{\alpha_2}$. The number of intersection
points inside of $F_{\alpha_i}$ is bounded because of the assumption on the 
twisting.}
\label{Fig:Intersection}
\end{figure}

First, we claim that any non-trivial curve in $N$ is homotopic to some curve
in $\cS_q$. This is because, any nontrivial loop $\gamma$ in $N$ has a $q$--length
of at most $3s/7$. By the definition of $s$, it can not be an essential curve in 
any subsurface $Q$ that $\omega_T$ intersects. Assume it intersects 
curves $\alpha_1, \alpha_2 \in \cS_q$ that are boundary curves of $Q_1$
($\alpha_1$ may equal $\alpha_2$). Then, $\ell_q(\alpha_1)$ and $\ell_q(\alpha_1)$  
are much smaller than $\ell_q(\gamma)$ which is at most $3s/7$. But, the sum of  
$\ell_q(\alpha_1)$,
$\ell_q(\alpha_2)$ and twice the distance between $\alpha_1$ and $\alpha_2$
(the sum is less than $s$) 
is an upper-bound for the size of $Q$ which is assumed to be larger than $s$. 
The contradiction proves the claim. 

We have shown that a closed curve in $N$ cannon intersect curves in $\cS_q$. 
However, the saddle connection $\omega_T$ may still intersects some curve 
$\alpha \in \Sqn$ (in fact more than one, see \figref{Intersection}). As before, 
let $\nu_\alpha$ be an arc in $F_\alpha$ that connects the boundaries of $F_\alpha$ 
and is perpendicular to them. 

First we observe that the number of intersection points between $\omega_T$ and 
$\bar \omega$ inside of $F_\alpha$ is uniformly bounded. This is because
both $\omega_T$ and $\bar \omega$ intersect $\nu_\alpha$ a uniformly
bounded number of times. (This follows from the definition $(q, \tau)$--regular
triangulation and the twisting assumption on $\omega$.) If two arcs inside
of a cylinder have a large intersection number, at least one of them 
has to twist around $F_\alpha$ a large number of times. 

It remains to show that the number of intersection points
outside of all cylinders $F_\alpha$ is bounded. To see this we observe that,
for any thick subsurface $Q$, it is not possible to have a subsegment of $\omega_T$ 
and a subsegment of $\bar \omega$ that are contained in $Q$ 
and have the same endpoint. Otherwise, the concatenation would create 
a two segment curve $\beta$ that is non-trivial in $N$. Hence, it has
to be homotopic to some curve $\alpha \in \cS_q$. Which means, $\alpha$
and $\beta$ create a cylinder with total negative curvature which contradicts 
the Gauss-Bonnet theorem. (See \cite[Lemma 5.6]{Rafi:minima} for a 
more detailed discussion.) 

Since the number of thick components $Q$ is uniformly bounded and $\omega_T$ and 
$\bar \omega$ can intersect at most once in each $Q$ we conclude that the total 
intersection number outside of cylinders $F_\alpha$ is uniformly bounded as well. 
This finishes the proof. 
\end{proof}

\begin{lemma} \label{Lem:Twist-q-T}
For a quadratic differential $(X,q)$, $\alpha \in \Sqn$ and a
$(q, \tau)$--regular triangulation $T$ we have
$$
\twist_\alpha(T, q) = O(1).  
$$
\end{lemma} 

\begin{proof}
Let $Q_1$ and $Q_2$ be the thick subsurfaces of $(X,q)$ glued along
the cylinder $F_\alpha$ (which by assumption, has a modulus at most 
$\CNM$), and let $\beta$ be an essential curve in $Q_1 \cup F_\alpha \cup Q_2$ 
that is transverse to $\alpha$ and has the shortest combinatorial $T$--length. 
A representative for the curve $\beta$ can be constructed using edges of $T$ that 
intersect either $Q_1$ or $Q_2$. Consider such a representative traversing the 
minimum possible number of edges. Let $\gamma$ be a curve transverse to $\alpha$ 
with the shortest $q$--length. From the definition of relative twisting,
$$
\twist_\alpha(T,q) \ladd \I(\beta, \gamma).
$$
Hence, it is sufficient to show that $\I(\beta, \gamma)$ is uniformly bounded. 

The curve $\gamma$ intersects $\alpha$ once if $Q_1=Q_2$ and twice otherwise.
Its restriction to $Q_i$ has a length bounded by $O(s_{Q_i})$ and its restriction 
to $F_\alpha$ has a length bounded by $\ell_q(\alpha)$ ($\Mod(F_\alpha)$
is bounded and there is no twisting around $\alpha$) which is less than both
$s_{Q_1}$ and $s_{Q_2}$. 
An argument similar to that of \lemref{Length-Int} implies that $\gamma$
intersects any edge of $T$ at most a bounded number of times. 

On the other hand, each edge of $T$ appears at most twice along the
representative of $\beta$, otherwise a surgery argument would reduce the 
length of $\beta$. Also, 
the total number of edges of $T$ is bounded by the topology of $S$. 
Hence, $\I(\beta, \gamma)$ is uniformly bounded. 
\end{proof}

\subsection{The number of $(q, \tau)$--regular triangulation} 
We now count the number of $(q, \tau)$--regular triangulations near a point 
in \Teich space. We can think of a $(q, \tau)$--regular triangulations on $(X,q)$ 
as topological objects on $S$, after being pulled back by the marking map 
$f_X \from S \to X$, up to homotopy. That is, we say a $(q, \tau)$--regular 
triangulation $T$ on $(X,q)$ is equivalent to a $q'$--regular triangulation $T'$ on 
$(X',q')$ if the pre images $f_X^{-1}(T)$ and $f_{X'}^{-1}(T')$ are homotopic 
on $S$. The homotopy does not have to fix the vertices of $T$. For a 
multi-curve $\cS_0$, we say $T$ is equivalent to $T'$ up twisting around $\cS_0$ 
if, $T$ is equivalent to $\phi(T')$ where $\phi$ is a multi-twist with
support on curves in $\cS_0$. 

\begin{lemma} \label{Lem:up:to:twisting} 
Let $U$ be a ball of radius one in $\T(S)$ centered at $X_0$. 
Then the number of equivalence classes, up to twisting around
$\cS_{X_0}$, of $(q, \tau)$--regular triangulations $T$ on a quadratic differential 
$(X,q)$ where $X \in U$ is uniformly bounded. 
\end{lemma}

\begin{proof}
We start with a topological counting statement. 
Let $\cS_0=\cS_0^c \cup \cS_0^n$ be a system of curves on $S$.
For every subsurface $Q$ in $S \setminus \cS_0$, let $\mu_Q$  
be a marking for the subsurface $Q$ in the sense of \cite{CCII}.
That is, $\mu_Q$ is a pants decomposition $\{\gamma_1, \ldots, \gamma_k\}$ 
for $Q$ together with a transverse curve $\overline{\gamma}_i$ for 
$0\leq i\leq k$. Each $\overline{\gamma}_i$ is contained in $Q$, 
intersects $\gamma_i$ once or twice and is disjoint from $\gamma_j$, 
$j \not = i$. Also,  for $\alpha \in \cS_0^n$,
let $\beta_\alpha$ be a curve transverse to $\alpha$ that is
disjoint from all other curves in $\cS_0$ and $\I(\beta_\alpha, \mu_Q)=O(1)$.
Define
$$
M= \bigcup_{Q} \mu_Q \cup \cS_0 \cup \{\beta_\alpha\}_{\alpha \in \cS_0^n}. 
$$
\subsection*{Claim:} Given a set $M$ as above, there is a uniformly bounded 
number of possibilities for the homotopy class of a triangulation $T$,
triangulating $S \setminus \cS_0^c$, where the curves in $M$ and $T$ have 
representatives with the following properties: 
\begin{enumerate}
\item curves in $M$ have no self intersections and intersect each other minimally. 
\item for any $\alpha \in \cS_0^c$, $\I(T, \alpha) = 0$.
\item for any $\gamma \in \mu_Q$, $\I(T, \gamma) = O(1)$.
\item for $\alpha \in \cS_0^n$, $\twist_\alpha(T, \beta_\alpha) =O(1)$, 
and $\I(T, \alpha) = O(1)$.
\end{enumerate}
To see the claim, note that the curves in $M$ divide $S$ into a uniformly 
bounded number of complementary regions, each one is either a polygon 
or an annulus parallel to a curve $\alpha \in \cS_0^c$. Choose a representative 
of the homotopy class of $T$ that intersects curves in $M$ minimally. 
There are a uniformly bounded number of possibilities for the location of 
vertices of $T$. Once the vertices of $T$ are fixed, there are a 
uniformly bounded number of possibilities for any given arc, with end points on these
vertices, that can appear as an edge of $T$. This is because there are a 
uniformly bounded number of possibilities for the intersection pattern of the given arc 
with the complementary regions. Also, each region is either a polygon where there is 
a unique arc (up to homotopy) connecting any two edges (or a vertex to an edge) 
or an annulus neighborhood of a curve $\alpha \in \cS_0^c$ where there are 
two possibilities (edges of $T$ are simple and disjoint from curves in $\cS_0^c$).  

It remains to show, that for every $(q, \tau)$--regular triangulation $T_q$
on $(X,q)$ where $X \in U$, there is a set of simple closed curves $M_q$ so that 
$T_q$ and $M_q$ satisfy the above properties and then to bound the number 
of possibilities for the set $M_q$. 

Let $(X, q)$ be a quadratic differential so that $X \in U$. We construct $M_q$ as 
follows: The curves $\cS_q=\Sqc \cup \Sqn$ have a uniformly bounded 
length in $X_0$ hence there are a uniformly bounded number of 
possibilities for these sets. For each thick subsurface $Q$ of $q$, choose
a $q$--short marking $\mu_Q$ in $Q$. Curves in $\mu_Q$
have a uniformly bounded length on $X$ and hence a uniformly bounded
length in $X_0$. Hence there are only a uniformly bounded number of choices for 
these as well. Now for each $\alpha$, let $\beta_\alpha^q$ be the shortest $q$
transverse curve to $\alpha$. \lemref{Twist-q-X} implies that 
$\twist_\alpha(X_0, \beta_\alpha^q) \lmul \frac 1{\Ext_{X_0}(\alpha)}$.
Hence the number of possible choices for $\beta_\alpha^q$ is of the order
of $\frac 1{\Ext_{X_0}(\alpha)}$. Define
$$
M_q= \bigcup_{Q} 
  \mu_Q \cup \cS_q \cup \{\beta^q_\alpha\}_{\alpha \in \Sqn}.
$$
Bu construction, the total number of possible sets $M_q$ chosen as above is
of the order of $G(X_0)$. However, up to twisting around $\cS_{X_0}$
there are only finitely many choices. For a $(q, \tau)$--regular triangulation $T_q$
in $(X,q)$, we need to check that the conditions (1)-(4) hold for $T_q$ and
$M_q$. Perturb the $q$--geodesic representative of curves in $M_q$ 
so that they have no self-intersections, intersect each other minimally
and the intersection number with $T$ does not increase. Condition (2)
follows from the construction of $(q, \tau)$--regular triangulations. Condition
(3) follows from \lemref{Length-Int}. The first part of condition (4)
is a consequence of \lemref{Twist-q-T} and the second part again follows 
from \lemref{Length-Int}.
\end{proof}

\section{Intersection bounds between regular triangulations}
\label{Sec:Intersection-Bound}

As before, let $\Qs$ be the stratum of quadratic differentials of type $\sigma$.
In this section, we establish some intersection bounds for $(q, \tau)$--regular
triangulations associated to a pair of quadratic differentials that appear
at the end points of a geodesic segment in $\Qs$. 

Recall, from \remref{tau}, that there is an implicit assumption that 
the constant $\tau$ is large. That is, there is a uniform constant $\tau_0$
so that all statements in this section hold as long as $\tau \ge \tau_0$.
In particular, the implied constant in our estimates do not get worst as $\tau$ 
gets larger. 
 
\subsection{Notation}\label{Sec:notation}
First we need to establish some notations. 

\subsection*{1}
For a fixed constant $r_0$, define $\cB (\Qs, X, \tau)$ to be the set of points 
$Z \in \cT(S)$ so that there is a \Teich geodesic 
$$
\cG_Z \from [a,b] \to \QTs, \qquad \cG_Z(t) =(X_t, q_t),
$$
such that 
$$
 d_\T(X_a, X) \leq r_0, \quad d_\T(X_b, Z) \leq r_0, \quad
 b-a \leq \tau.
$$
and
$$
  (X_t, q_t)  \in \Qs.
$$
One could think of $B(\Qs, X, \tau)$ as a ball of radius $\tau$
centered at $X$, except that one is allowed only to move in the direction of 
$\Qs$. Since $r_0$ is fixed, we refer to any constant that depends on $r_0$
as a uniform constant. The value of $r_0$ will be determined in \secref{Net}
depending on the choice of the net $\cN$. 

\subsection*{2} 
We use the notation of \eqnref{EFE} for $q_a$ and denote the flat
and expanding annuli associated to a curve $\alpha$ by 
$E^a_\alpha$, $F^a_\alpha$ and $G^a_\alpha$ and distances 
between their boundaries by $e^a$, $f^a$ 
and $g^a$. Let $\upsilon_a$ be an 
arc of length $f^a$ connecting the boundaries of $F_\alpha$. Also, let
$l^a= \ell_{q_a}(\alpha)$ and let $d^a= \max(e^a,f^a,g^a)$ be the maximum distance
between the boundaries of these annuli.  As a consequence of 
Equations~\eqref{Eq:EFE} and \eqref{Eq:Moduli} we have
\begin{equation} \label{Eq:collar} 
\frac 1{\Ext_X(\alpha)} \lmul \frac{d^a}{l^a}
\qquad\text{and}\qquad
\frac 1{\Ext_X(\alpha)} \gmul \frac{f^a}{l^a}.
\end{equation}

\subsection*{3} 
Let $T_a$ be a $(q_a, \tau)$--regular triangulation and $T_b$  be a 
$(q_b, \tau)$--regular 
triangulation. The geodesic flow induces a one-to-one correspondence between
saddle connections of $q_a$ and $q_b$. Hence, we can consider 
$T_b$ as a union of saddle connections in $q_a$. Then $T_a$ and $T_b$ have 
identical vertex sets and their edges are either identical or intersect transversally. 
The slope of a saddle connection in $q_a$ (or in $q_b$) is a well defined number 
in the interval $[0,\infty]$. 

\begin{definition} \label{Def:Positive}
Let $\omega_a$ be a saddle connection in $q_a$ and let $\omega_b$ be 
a saddle connection in $q_b$. We say $\omega_b$ intersects $\omega_a$ 
\emph{positively},  if when considering them both in $q_a$ (or $q_b$), 
the slope of $\omega_b$ is larger than the slope of $\omega_a$. We say 
$\omega_b$ intersects $\omega_a$ \emph{essentially positively} if 
either $\omega_b$ intersects $\omega_a$ positively or 
$\I(\omega_a,\omega_b)=O(1)$. We use similar terminology for intersection
between a saddle connection and a cylinder curve and two cylinder curves. 
\end{definition}

\subsection{Intersection and twisting bounds between $T_a$ and $T_b$} 
\label{Sec:IntersectioBound} 
For the rest of this subsection, we assume that $q_a$ and $q_b$,
 $T_a$ and $T_b$ are as described in the beginning of the section. 

\begin{lemma} \label{Lem:EdgeNonCylinder} 
Let $\alpha \in \Sn{q_a}$, $\omega_b \in T_b$ and 
$\alpha_b \in \cS_Z$, then 
$$
\twist_\alpha(q_a, \omega_b) =O(1)
\qquad\text{and}\qquad
\twist_\alpha(q_a, \alpha_b) =O(1).
$$
Similarly, let $\alpha \in \Sn{q_b}$, $\omega_a \in T_a$ and 
$\alpha_a \in \cS_X$, then 
$$
\twist_\alpha(q_b, \omega_a) =O(1)
\qquad\text{and}\qquad
\twist_\alpha(q_b, \alpha_a) =O(1).
$$
\end{lemma}

\begin{proof}
Let $\nu_\alpha$ be the arc connecting the boundaries of $F_\alpha^a$
and is perpendicular to them. Then, by definition of $\Sqn$, 
$$
\frac{\ell_{q_a}(\nu_\alpha)}{\ell_{q_a}(\alpha)}
= \Mod(F_\alpha^a) \lmul e^{-2\tau}. 
$$
Therefore
$$
\frac{\ell_{q_b}(\nu_\alpha)}{\ell_{q_b}(\alpha)} \lmul 1. 
$$
That is, $\nu_\alpha$ twists around $\alpha$ in $q_b$ a bounded
number of times. But the same is true for $\omega_b$. 
This gives a bound on $\I(\omega_b, \nu_\alpha)$ and thus on
$\twist_\alpha(q_a, \omega_b)$. Also, the curve $\alpha_b$
is short in $Z$ and hence in $q_b$. A short curve can not twist
around any other curve. Hence $\I(\alpha_b, \nu_\alpha)$
is uniformly bounded. Which means $\twist_\alpha(q_a, \alpha_b)$
is uniformly bounded. The proofs of the other two assertions are similar. 
\end{proof}

\begin{remark} \label{Rem:TwistCondition}
The main consequence of this lemma is that the twisting condition
of \lemref{Length-Int} is satisfied and can be applied freely. 
\end{remark} 

\begin{lemma} \label{Lem:EdgeEdge}
Let $\omega_a$ and $\omega_b$ be edges of $T_a$ and $T_b$ respectively.
Then $\omega_b$ intersects $\omega_a$ essentially positively and 
$$\I(\omega_a, \omega_b) \lmul e^\tau.$$
\end{lemma}

\begin{proof}
Let $Q_a$ be the thick subsurface of $q_a$ with the smallest size that 
intersects $\omega_a$ and let $s_a$ be the size of the subsurface $Q_a$. 
Recall that, by the definition of a $(q_a, \tau)$--regular triangulation, we have
$$
\ell_{q_a}(\omega_a) \lmul s_a. 
$$
We denote the horizontal and the vertical lengths 
of $\omega_a$ by $x_a$ and $y_a$. Let $Q_b$, $s_b$, $x_b$ and $y_b$ be 
similarly defined. The length of $\omega_a$ in $Q_b$ is 
$$
\sqrt{(x_a e^{\tau})^2 + (y_a e^{-\tau})^2}\emul x_a e^\tau + y_a e^{-\tau}.
$$
If $\I(\omega_a, \omega_b)=O(1)$ we are done. Otherwise, 
Considering $\omega_a$ and $\omega_b$ in $Q_b$, 
in view of \remref{TwistCondition}, \lemref{Length-Int} implies  that 
$$
\I(\omega_a,\omega_b) \lmul \frac{\ell_{q_b}(\omega_a)}{s_b} +O(1).
$$
However, since $\I(\omega_a,\omega_b)$ is large, 
$\frac{\ell_{q_b}(\omega_a)}{s_b}$ is large and we can 
incorporate the additive error into the multiplicative error. That is, 
\begin{equation} \label{Eq:Intersection1}
\I(\omega_a,\omega_b) \lmul \frac{\ell_{q_b}(\omega_a)}{s_b} 
\emul \frac{x_a e^\tau + y_a e^{-\tau}}{s_b}.
\end{equation}
Similarly, considering $\omega_a$ and $\omega_b$ in $Q_a$ we get
\begin{equation} \label{Eq:Intersection2}
\I(\omega_a,\omega_b) \lmul \frac{\ell_{q_a}(\omega_b)}{s_a} 
\emul \frac{x_b e^{-\tau} + y_b e^\tau}{s_a}
\end{equation}

Observing that $x_a, y_a \lmul s_a$ and $x_b, y_b\lmul s_b$, we can multiply
the two inequalities and take a square root to get 
$\I(\omega_a,\omega_b) \lmul e^\tau$.

Now assume that $\omega_b$ does not intersect $\omega_a$ positively. 
This means that the slope of $\omega_a$ in $q_b$ is larger than the slope of 
$\omega_b$. That is
$$
\frac{y_a e^{-\tau}}{x_a e^\tau} \geq \frac {y_b}{x_b}
\qquad \Longrightarrow \qquad
x_ay_b \, e^{2\tau} \leq x_by_a. 
$$
From the product of inequalities in \eqnref{Intersection1} and 
\eqnref{Intersection2}, 
we have
\begin{align*}
\I(\omega_a,\omega_b)^2 
 & \lmul \frac{x_a x_b + y_a y_b + x_a y_b \, e^{2\tau} 
                 + x_b y_a e^{-2\tau}}{s_a s_b} \\
  & \lmul  \frac{x_a y_b \, e^{2\tau}}{s_a s_b} 
     \lmul \frac{x_b y_a}{s_a s_b} =O(1). \qedhere
\end{align*}
\end{proof}

For a simple closed curve $\alpha$ and a triangulation $T$, we say $T$ intersects
$\alpha$ essentially positively if any saddle connection in $T$ intersects 
any saddle connection in the geodesic representative of $\alpha$ 
essentially positively.

\begin{lemma}\label{Lem:EdgeCylinder}
If $\alpha \in \cS_X$ and $\alpha \not \in \cS_Z$
then $\alpha$ intersects $T_b$ essentially positively and
$$
\twist_\alpha(X,Z) \I(\alpha, T_b) \lmul \frac{e^\tau}{\sqrt{\Ext_X(\alpha)}}.
$$
Similarly, if $\alpha \in \cS_Z$ and $\alpha \not \in \cS_X$ then 
$\alpha$ intersects $T_a$ essentially positively and
$$
\twist_\alpha(X,Z) \I(\alpha, T_a) \lmul \frac{e^\tau}{\sqrt{\Ext_Z(\alpha)}}.
$$
\end{lemma}

\begin{proof}
Let $\alpha$ be a simple closed curve in $\cS_X \setminus \cS_Z$. 
Applying \lemref{Twist-q-X} to the pair $X$ and $q_a$ and to
the pair $Z$ and $q_b$, we get 
\begin{equation} \label{Eq:Xtoq}
\twist_\alpha(X,Z) \lmul \twist_\alpha(q_a,q_b) + \frac 1{\Ext_X(\alpha)}.
\end{equation}
(The term $\frac 1{\Ext_Z(\alpha)}$ is omitted from the right hand side
because it is bounded and can be absorbed in the multiplicative error.) 
Hence, to prove the lemma, it is sufficient to show that the expression
$\frac{e^\tau}{\sqrt{\Ext_X(\alpha)}}$ is an upper-bounds for both
$$
\frac {\I(\alpha, T_b)}{\Ext_X(\alpha)} 
\qquad\text{and}\qquad
\twist_\alpha(q_a,q_b) \I(\alpha, T_b).
$$

Let $\omega_b$ be an edge of $T_b$ and let $Q_b$ be the thick subsurface
of $q_b$ with the smallest size intersecting $\omega_b$. Let $s_b$ be the size 
of $Q_b$ (thus $\ell_{q_b}(\omega_b) \lmul s_b$, by the definition of a
$(q_b, \tau)$--regular triangulation). Applying \lemref{Length-Int} to $\omega_b$ 
and $\alpha$ in $q_b$, we get
\begin{equation} \label{Eq:IntersectionBound}
\I(\alpha, \omega_b) - O(1) \lmul \frac{\ell_{q_b}(\alpha)}{s_b} \leq
 e^\tau \frac{l^a}{s_b}.
\end{equation}
Also, each subsegment of $\omega_b$ with end points in $\alpha$ 
has a length larger than $d^a$. Hence, 
$$
\I(\alpha, \omega_b) -1 \leq \frac {\ell_{q_a}(\omega_b)}{d^a}
\lmul e^\tau \frac{s_b}{d^a}.
$$
Multiplying these two equations, taking the square root we and summing
over all arcs in $T_b$ we get
$$
\I(\alpha, T_b) - O(1) \lmul e^\tau \sqrt{\frac{l^a}{d^a}}.
$$
In view of \eqnref{collar}, we obtain
$$
\I(\alpha, T_b) - O(1) \lmul e^\tau \sqrt{\Ext_X(\alpha)}.
$$
Dividing both sides by $\Ext_X(\alpha)$ we obtain
\begin{equation} \label{Eq:first:estimate} 
\frac{\I(\alpha, T_b)}{\Ext_X(\alpha)} 
\lmul \frac{e^\tau+O(1)}{\sqrt{\Ext_X(\alpha)}}
\lmul \frac{e^\tau}{\sqrt{\Ext_X(\alpha)}}.
\end{equation}
This is the first estimate we required.

We now find an upper-bound for $\twist_\alpha(q_a,q_b) \I(\alpha, T_b)$
by finding separate upper bounds for $\twist_\alpha(q_a,q_b)$ 
and $\I(\alpha, T_b)$. The argument involved in this new upper bound for
$\I(\alpha, T_b)$ is somewhat similar to above, but the two bounds do not imply
each other. We need to consider the image of $F^a_\alpha$ in $q_b$ under the 
\Teich geodesic flow. Denote this cylinder by $F_\alpha^b$, the distance between 
its boundaries by $f^b$ and let $\upsilon^b_\alpha$ be an arc of length $f^b$ 
connecting the boundaries of $F^b_\alpha$. Let $l^b= \ell_{q_b}(\alpha)$. 
Note that the area of $F_\alpha^a$ and $F_\alpha^b$ 
are equal, that is 
$$
l^a \, f^a = l^b\, f^b. 
$$

Consider again the arc $\omega_b$ in $T_b$ of $q_b$--length of order $s_b$. 
Then the $q_b$--length of every component of $\omega_b \cap F^b_\alpha$ 
is larger than $f^b$. Therefore
$$
\I(\alpha, \omega_b) \lmul \frac{s_b}{f_b} = 
\frac{s_b \, l^b}{l^a \, f^a}.
$$
As before, applying \lemref{Length-Int} to $\omega_b$ and $\alpha$ in 
$q_b$ we have
\begin{equation*}
\I(\alpha, \omega_b) \lmul \frac{l^b}{s_b} + O(1) \lmul \frac{l^b}{s_b}.
\end{equation*}
The reason we can ignore the additive errors here is that
since $\alpha$ is not short in $Z$, it has to either be an essential curve
in $Q_b$ or intersect some boundary curve of $Q_b$. In either case,
$l^b \gmul s_b$, in the first case by definition of the size and in the second case 
by \corref{CollarLemma}. Hence, the additive error can be absorbed into the fraction 
$\frac{l^b}{s_b}$. Multiplying the last two inequalities, taking the square root and 
summing over all arcs in $T_b$, we obtain
\begin{equation} \label{Eq:TwistBound}
\I(\alpha, T_b) \lmul \frac{l^b}{\sqrt{ l^a \, f^a  }}.
\end{equation}

We now argue that a component of $\omega_b \cap F_\alpha$ can intersect 
$\upsilon_\alpha^b$ at most a uniformly bounded number of times: since $\alpha$ is 
not short in $Z$, $\ell_{q_b}(\alpha) \gmul s_b$ and $\ell_{q_b}(\omega_b) \lmul s_b$, 
which means the intersection number between $\omega_b$ and $\upsilon_\alpha^b$
is at most $\frac{\ell_{q_b}(\omega_b) }{\ell_{q_b}(\alpha)}  = O(1)$. 
Therefore, the relative twisting of $q_a$ and $q_b$ around $\alpha$ is comparable 
to the intersection number between $\upsilon_\alpha^a$ and $\upsilon_\alpha^b$
which is at most the $q_a$--length of $\upsilon_\alpha^b$ divided by
the $q_a$--length of $\alpha$. That is
$$
\twist_\alpha(q_a, q_b) -O(1)\emul  \I(\upsilon_\alpha^a, \upsilon_\alpha^b)
\leq \frac{e^\tau f^b}{l^a}.
$$
Taking a product and using the second part of \eqnref{collar} we get:
\begin{align} \label{Eq:TwistIntersection}
\twist_\alpha(q_a,q_b) \I(\alpha, T_b) - O(\I(\alpha, T_b) )
 & \lmul e^\tau\left( \frac{l^b f^b}{l^a } \right)\frac{1}{\sqrt{l^a f^a}}  \\
 & \lmul e^\tau \sqrt{\frac{f^a}{l^a}} \lmul  \frac{e^\tau}
{\sqrt{\Ext_X(\alpha)}}. \notag
\end{align}
By \eqnref{first:estimate}, we have $\I(\alpha, T_b)$ is much smaller than 
$\frac{e^\tau}{\sqrt{\Ext_X(\alpha)}}$. Hence, 
\begin{equation} \label{Eq:second:estimate} 
\twist_\alpha(q_a,q_b) \I(\alpha, T_b) \lmul \frac{e^\tau}{\sqrt{\Ext_X(\alpha)}}. 
\end{equation}
The estimate in the Lemma follows from Equations \eqref{Eq:Xtoq}, 
\eqref{Eq:first:estimate} and \eqref{Eq:second:estimate}.

It remains to show that $\omega_b$ and $\alpha$ intersect essentially positively. 
Let $\omega_a$ be a saddle connection of $\alpha$ that intersects
$\omega_b$ many times. Then, by \lemref{Length-Int}, 
$\ell_{q_b}(\omega_a) \geq s_b$. However, $\ell_{q_b}(\omega_b) \lmul s_b$ 
and hence $\ell_{q_b}(\omega_a) \gmul l_{q_b}(\omega_b)$. 
If the slope of $\omega_\alpha$ was smaller than $\omega_b$ (say in $q_b$) 
then we would also have $\ell_{q_a}(\omega_a) \gmul \ell_{q_a}(\omega_b)$. 
Hence, $\omega_a$ intersects $\omega_b$ at most twice
(its length is less than $d^a$). This proves that $\omega_b$ intersects $\omega_\alpha$ 
essentially positively. 
But this is true for every saddle connection of $\alpha$. Thus
$\omega_b$ intersects $\alpha$ essentially positively. 
The case when $\alpha \in \cS_Z$ can be treated similarly.
\end{proof}

\begin{lemma}\label{Lem:curves}
Let $\alpha \in \cS_X$ and $\alpha_0 \in \cS_Z$. Then
\begin{enumerate}
\item If $\alpha_0 \not = \alpha$, then 
\begin{equation*} \label{Eq:IntersectionTwistTwist}
\I(\alpha,\alpha_{0}) \twist_{\alpha_{0}} (X, Z) 
\twist_{\alpha}(X, Z) \lmul \frac{e^\tau}{ \sqrt{\Ext_X(\alpha) 
\Ext_Z(\alpha_0)} }, 
\end{equation*} 
\item If $\alpha=\alpha_{0}$, then
\begin{equation*} \label{Eq:IntersectionTwist}
\twist_{\alpha} (X, Z) \lmul \frac{e^\tau}{ \sqrt{\Ext_X(\alpha) 
\Ext_Z(\alpha_0)} }.
\end{equation*}
\end{enumerate}
\end{lemma}

\begin{proof}
It is enough to prove that 
\begin{equation} \label{Eq:Twist}
\twist_{\alpha} (X, Z) \sqrt{\Ext_X(\alpha) \Ext_Z(\alpha)}  \lmul e^\tau.
\end{equation}
If $\alpha=\alpha_{0},$ this is equivalent to \eqnref{IntersectionTwist}. 
Also, if $\alpha \not = \alpha_0$ and $\alpha \in S_Z$, then the inequality (1) 
trivially holds (the left hand side is $0$). Otherwise, from \thmref{minsky:ext} (estimating 
$\Ext_Z(\alpha)$) we have
\begin{equation}\label{Eq:0}
\I(\alpha,\alpha_{0}) \sqrt{\Ext_Z(\alpha_0)} \twist_{\alpha_0}(X, Z) 
   \lmul \sqrt{\Ext_Z(\alpha)}.
\end{equation}
Multiplying the above equation to \eqnref{Twist} we obtain part (1) of the lemma. 

By \corref{the:curve}, (replace $X$ with $Z$, $\gamma$ with $\alpha$ and
$\beta$ with $\beta_0$) there always exits a simple closed curve $\beta_0$ so that
$\twist_{\alpha}(\beta_0, Z)=O(1)$ and 
\begin{equation}\label{Eq:1}
\sqrt{\Ext_Z(\beta_0) \Ext_Z(\alpha)} \lmul \I(\alpha,\beta_0). 
\end{equation}
On the other hand, from \thmref{ker} we have
\begin{equation}\label{Eq:2}
\sqrt{\Ext_X(\beta_0)} \lmul e^{\tau} \sqrt{\Ext_Z(\beta_0)}, 
\end{equation}
and from \thmref{minsky:ext} (this time estimating the length of 
$\beta_0$ is $X$) we have
\begin{equation}\label{Eq:3}
\I(\alpha,\beta_0) \twist_{\alpha}(X,\beta_0) \sqrt{\Ext_X(\alpha)}
   \lmul \sqrt{\Ext_X(\beta_0)}. 
\end{equation}
Since $\twist_{\alpha}(\beta_0, Z)=O(1)$, we can replace $\twist_{\alpha}(X,
\beta_0)$
with $\twist_{\alpha}(X,Z)$ in the above inequality. Now, \eqnref{Twist} is
obtained by successive substitution using Equations 
\eqref{Eq:0}, \eqref{Eq:1}, \eqref{Eq:2} and \eqref{Eq:3}. 
\end{proof}

\subsection{Relations between intersections numbers} \label{Sec:Relations}
So far, we have provided upper-bounds for the intersection numbers between 
the edges of $T_a$ and the edges of $T_b$. But these intersection numbers are 
not independent. The fact that the edges in $T_a$ intersect edges in $T_b$ 
essentially positively allows us to find relations between these intersection numbers. 
In this section we will describe these relations. There are two kinds of relations. 

\begin{lemma} \label{Lem:Triangle}
For every triangle in $T_a$ with edges $\omega_1$, $\omega_2$ and $\omega_3$ ,
there are sings $\vs_1, \vs_2, \vs_3 \in\{-1,+1\}$ so that, for every edge $\omega_b
$ 
in $T_b$ (respectively, for any $\alpha_b \in \Sc{q_b}$), we have the relation:
\begin{equation} \label{Eq:RelationOne}
\sum_{i=1,2,3}\vs_i \I(\omega_i, \omega_b) = O(1)
\qquad \left(\text{respectively,} \quad 
\sum_{i=1,2,3}\vs_i \I(\omega_i, \alpha_b) = O(1) \right).
\end{equation}
The additive error depends on the constant involved in the definition of essential positively. 
\end{lemma}

\begin{proof}
There is a leaf of the vertical foliation that passes through a vertex
of the given triangle before entering it. Assume this leaf intersects 
the interior of $\omega_3$ and makes an acute angle with $\omega_1$
inside of the triangle. We claim that, since $\omega_b$ intersects $\omega_1$
essentially positively, the number of sub-arcs of $\omega_b$ going
from $\omega_1$ to $\omega_2$ is uniformly bounded. This is because
either the slope of $\omega_b$ is larger than the slope of $\omega_1$ and 
every time $\omega_b$ intersects $\omega_1$ it has to intersect $\omega_3$ 
next, or it intersect $\omega_1$ a bounded number of times. 
Hence, we have
$$
 \I(\omega_1, \omega_b) + \I(\omega_2, \omega_b) =
  \I(\omega_3, \omega_b) +  O(1).
$$
Note that the signs $\vs_1=1$, $\vs_2=1$ and $\vs_3 = -1$ depend only
on the triangle and are independent of $\omega_b$. The proof for $\alpha_b$ 
is similar. 
\end{proof}

For each $\alpha \in \Sc{q_a}$, consider a saddle connection
$\beta_\alpha$ connecting the boundaries of $F_\alpha$.
Let 
$$
U_a= T_a \cup \bigcup_{\alpha \in \Sc{q_a}} \beta_\alpha.
$$  
We can choose the arcs
$\beta_\alpha$ so that $\twist_\alpha(q_a, U_a) = O(1)$. After orienting
the arcs in $U_a$, we can think of them as elements
of $H_1(S, \Sigma)$ where $\Sigma$ is the set critical points of $q_a$. 
In fact, arcs in $U_a$ generate $H_1(S, \Sigma)$.

\begin{lemma} \label{Lem:NonOrientable}
Assume that the vertical foliation of $q_a$ is not orientable. Then, there 
is a set $\cB$ of edges of $U_a$ and for $\omega \in \cB$ there is a 
sign $\vs_\omega \in \{-1, +1\}$ so that, for every $\omega_b$ in $T_b$ 
(respectively, for any $\alpha_b \in \Sc{q_b}$), we have the relation:
\begin{equation} \label{Eq:RelationTwo}
\sum_{\omega \in \cB} \vs_\omega \I(\omega, \omega_b) = O(1)
\qquad \left(\text{respectively,} \quad 
\sum_{\omega \in \cB} \vs_\omega \I(\omega, \alpha_b) = O(1) \right).
\end{equation}
Furthermore, this relation is independent of all the relations in 
\lemref{Triangle}. 
\end{lemma}

\begin{proof}
Choose a minimum number of edges of $U_a$ so that the complement is 
simply connected. Denote the set of all these edges by $\overline \cB$
and orient them in some arbitrary way. Minimality implies that the compliment 
$P$ is connected. We can visualize $P$ as a polygon in $\C$ with the vertical foliation 
parallel to the imaginary axis. Each edge of $\overline \cB$ 
has two representatives in the boundary of $P$. The two vectors are equal up 
to a multiplication by $\pm 1$. Let $\cB$ be the subset of $\overline \cB$ 
where the two representatives are negatives of each other (\figref{P}). 
Note that $\cB$ is non-empty since the vertical foliation in $q_a$ is not orientable. 

\begin{figure}[ht]
\setlength{\unitlength}{0.01\linewidth}
\begin{picture}(50, 24)
\put(0,3){\includegraphics[height=20\unitlength]{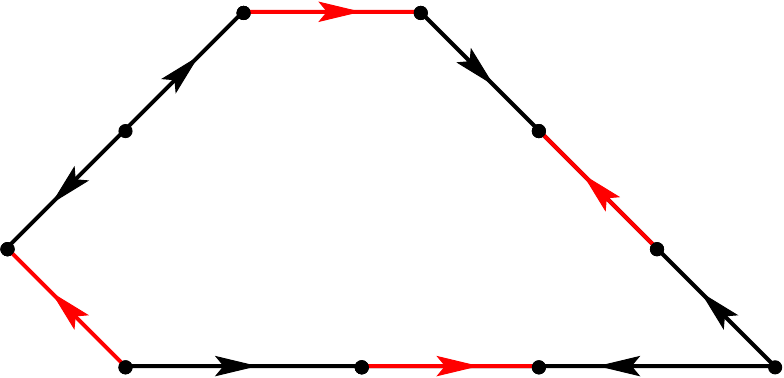}}
\put(1.5,15){$a$}
\put(7,20.5){$a$}
\put(17,24){$e$}
\put(23,0){$e$}
\put(26.5,20){$c$}
\put(39,7.5){$c$}
\put(33.5,13.5){$d$}
\put(1,4.5){$d$}
\put(13,0){$b$}
\put(34,0){$b$}
\put(17, 13){$P$}
\end{picture}
\caption{Polynomial $P$. The set $\cB= \{a,b,c\}$.}
\label{Fig:P}
\end{figure}

Now consider a double cover of $q_a$ constructed as follows. Take
a second copy $P'$ of $P$. Glue the edges that were not in $\cB$
as before and glue the edges in $\cB$ to the corresponding
edge in $P'$. Let $\cBt$ be the set of lifts of edges in $\cB$ to this cover. 
We now orient edges in $\cBt$ so that, for every $\omegat \in \cBt$, $P$ is in the 
same side of $\omegat$ (say, the left side). Denote this double cover by 
$\qt_a = P \cup P'$. 

Let $\St$ be the underlying surface for $\qt_a$ and $\Sigmat$ be 
the pre-image of $\Sigma$. Considering oriented saddle connections as 
elements of $H_1(\St, \Sigmat)$ we let $\hat \I (\param, \param)$ denote 
the algebraic intersection number. 
Note that $\qt_a$ is the unique double cover of $q_a$ where $\qt_a$ 
is a square of an abelian differential. Hence, for every two oriented saddle 
connections $\omegat$ and $\omegat'$ in $\qt_a$, all the intersection points 
have the same signature. That is, 
$$
\I(\omegat, \omegat')= |\hat \I (\omegat, \omegat')|.
$$

\begin{figure}[ht]
\setlength{\unitlength}{0.01\linewidth}
\begin{picture}(90, 24)
\put(0,3){\includegraphics[height=20\unitlength]{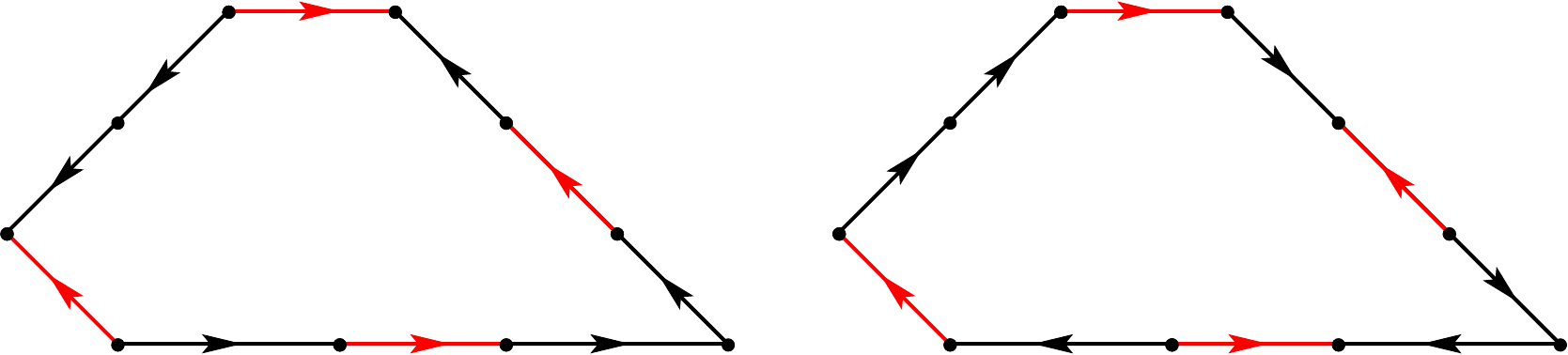}}
\put(1.5,15){$a$}
\put(7,20.5){$a'$}
\put(13,0){$b$}
\put(34,0){$b'$}
\put(26.5,20){$c$}
\put(39,7.5){$c'$}
\put(33.5,13.5){$d$}
\put(1,4.5){$d$}
\put(17,24){$e$}
\put(23,0){$e$}
\put(17, 13){$P$}

\put(48.2,15){$a$}
\put(54,20.5){$a'$}
\put(60,0){$b$}
\put(81,0){$b'$}
\put(74,20){$c$}
\put(87,7.5){$c'$}
\put(80,13.5){$d'$}
\put(48,4.5){$d'$}
\put(64,24){$e'$}
\put(70,0){$e'$}
\put(64, 13){$P'$}

\end{picture}
\caption{Polynomials $P$ and $P'$. The set $\cBt = \{a, a', b, b', c, c' \}$.}
\label{Fig:PP}
\end{figure}

Consider $\omega \in \cB$ and its lift $\omegat$. Note that $\omegat$
has an orientation and hence is identified with vector in $\C$. 
We define $\vs_\omega$ to be $+1$ if $\omegat$ has a positive $x$--coordinate 
and $-1$ otherwise. Let $\omega_b \in T_b$ and let $\omegat_b$ be a lift of 
$\omega_b$. We choose an orientation for $\omegat_b$ so it has a positive 
$y$--coordinate. We will show that
$$
\sum_{\omega \in \cB} \vs_\omega \I(\omega, \omega_b) = O(1).
$$

Consider an intersection point of $\omegat_b$ and 
$\omegat$ where $\vs_\omega=1$. If the absolute value of the slope
of $\omegat_b$ is larger than that of $\omegat$ then $\omegat_b$ is to the 
left of $\omegat$ and hence $\omegat$ intersects $\omegat_b$ with
a positive signature. Otherwise, $\omegat$ and $\omegat_b$ intersect
a uniformly bounded number of times ($\omegat_b$ and $\omegat$ 
intersect essentially positively). The opposite of this is true if $\vs_\omega=-1$;
either $\omegat$ intersects $\omegat_b$ with a negative signature or
a uniformly bounded number of times. If $\omegat, \omegat' \in \cBt$ are 
lifts of the same arc $\omega \in \cB$ then, choosing orientations for
$\omegat$ and $\omegat'$ as above, we have
\begin{equation} \label{Eq:lifts}
\I(\omega, \omega_b) = \I(\omegat, \omegat_b) + \I(\omegat', \omegat_b) \eadd 
\vs_\omega \left(\hat \I(\omegat, \omegat_b) + \hat \I(\omegat', \omegat_b) 
\right).
\end{equation}
To reiterate, this is because the number of intersection points that do not
have the same sign as $\vs_\omega$ is uniformly bounded.  

But arcs in $\cBt$ separate $\qt_a$. Thus, 
$$
\sum_{\omegat \in \cBt} \hat \I (\omegat, \omegat_b) \leq 1.
$$
This is because every time $\omegat_b$ exits $P$ it intersects the boundary
with the opposite signature than when it enters it. The sum is not necessarily 
zero because $\omegat_b$ may start inside $P$ and end in $P'$. 
Therefore, summing \eqnref{lifts} over $\omega \in \cB$, we get
$$
\sum_{\omega \in \cB} \vs_\omega \, {\hat \I}(\omega, \omega_b) =O(1). 
$$
The proof for the case of a simple closed curve $\alpha_b \in \Sc{q_b}$ is similar. 

Finally, we note that the relations of the type \eqref{Eq:RelationOne}
are also relations in the relative homology with $\Z_2$--coefficients. 
But the edges in $\cB$ are independent in $\Z_2$--relative homology. 
Hence, this new relation is independent from the previous ones. 
\end{proof}

\section{Main counting statement}\label{Sec:main:counting}
This section contains the main combinatorial counting arguments
with the goal of proving \thmref{LatticeCount}. 
Recall the definition of $B(\Qs, X, \tau)$ from \secref{notation}. Define 
$$
\Bj X, \tau) \subset B(\Qs, X, \tau)
$$
to be the set of points $Z \in \cT(S)$ so that, for the associated quadratic
differentials $q_a$ and $q_b$, there is a $(q_a, \tau)$--regular triangulation $T_a$
and a $(q_b, \tau)$--regular triangulation $T_b$ that have $j$ common 
homologically independent saddle connections. Now let,
$$
  B(\Qs, X,Y,\tau)= B(\Qs, X,\tau) \cap \big(\Gamma(S) \cdot Y \big),
$$
and
$$
  \Bj X,Y,\tau)= \Bj X,\tau) \cap\big(\Gamma(S) \cdot Y \big).
$$
That is, $\Bj X,Y,\tau)$ is the intersection of the orbit of $Y$ with $\Bj X,\tau)$.
Also, recall from \secref{bounded} that (when $\cS_X$ is empty, $G(X)=2$): 
$$
 G(X) = 1 + \prod_{\alpha \in \cS_X} \frac 1{\sqrt{\Ext_X(\alpha)}}
 \emul \prod_{\alpha \in \cS_X} \frac 1{\sqrt{\Ext_X(\alpha)}}.
$$
Notice that if $\gg \cdot Y \in \Bj X,Y,\tau)$ then $\gg^{-1} \cdot X \in \Bj Y,X,\tau)$. 
Thus, the number of points in $\Bj X,Y,\tau)$ is the same as the number 
of points in $\Bj Y,X,\tau)$. We prove the following upper-bound for the size of 
$\Bj Y,X,\tau)$:

\begin{theorem} \label{Thm:LatticeCount}
Consider the stratum $\Qs$. Given $X, Y \in \cT(S)$
$$
\big| \Bj X,Y,\tau) \big| \lmul \tau^{|\cS_X|+|\cS_Y|} e^{(h-j)\tau} G(X) G(Y),
$$
where $h=\frac{\dim \Qs}{2}$.
\end{theorem}

\begin{remark}
First we make a few remarks 
\begin{enumerate}
\item If, in the definition of $\Bj X,Y,\tau)$, we replace the assumption on the number 
of common homologically independent saddle connections with an assumption on the 
number of common homologically independent simple closed curves, the same 
statement would still holds. However, the theorem is strictly stronger. For example, 
assume $\cS_X \cap \cS_Y$ contains only one homologically trivial simple closed curves
$\alpha$.  We can still conclude that $j \geq 1$ because the geodesic representative 
of $\alpha$ in any quadratic differential $q$ contains a (homologically) non-trivial arc. 
That is, the number points $Y$, where the geodesic connecting $X$ to $Y$ follows 
$\Qs$ and contains a short curve throughout, is smaller than expected even when
$\alpha$ is a homologically trivial curve. 
\item The statement appears to be correct even without the term 
$\tau^{|\cS_X|+|\cS_Y|}$. However, the proof would become significantly 
more complicated. 
\end{enumerate}
\end{remark}

\subsection{Sketch of the proof \thmref{LatticeCount}}
Here is a an outline of our strategy :
\begin{enumerate}
\item We define a notion of a marking for the surface $S$ and what it
means for a marking to have a bounded length in a Riemann surface $X$.
A marking contains a partial triangulation of $S$, a set of short
simple closed curves with their lengths and some twisting information. 
Fixing a Riemann surface $X$, every quadratic differential $q$ 
where the underlying conformal structure is near $X$ defines a marking 
that has a bounded length in $X$. A marking takes the lengths of the 
short simple closed curves and the twisting information around short cylinder 
curves from $X$ and the triangulation and twisting around the non-cylinder 
short simple closed curves from $q$. Up to some twisting information, there are a 
uniformly 
bounded number of markings that have bounded length in a given Riemann 
surface $X$. 

\item Fixing a marking $\Delta_0$, a relation is a formal linear combination of
edges of $\Delta_0$ with integer coefficients. Given $\Delta_0$ and $\Delta_1$ 
and a set of relations $\cR$ we will define a set $M_\cR(\Delta_0, \Delta_1, \tau)$ 
consisting of all markings $\Delta$ such that $\Delta$ is a homeomorphic image of 
$\Delta_1$, its weighted intersection number with $\Delta_0$ is less than $e^\tau$ and 
so that the intersection patterns between $\Delta$ and $\Delta_0$ satisfy the relations in 
$\cR$. The weights depend on the length and the twisting information of each short 
simple closed curve. This is similar to assuming that there is a geodesic segment in 
a the stratum $\Qs$ starting near $X$ and ending near $Y$. 
\lemref{MarkingCount} provides and upper-bound for the number of 
elements in $M_\cR(\Delta_0, \Delta_1, \tau)$.

\item We then let $\cR$ be the set of relation of the type described in
 \lemref{Triangle} and \lemref{NonOrientable}. Each $Z \in \Bj X,Y,\tau)$ can 
 then be 
mapped to a marking in $\Delta \in M_\cR(\Delta_0, \Delta_1, \tau)$ for some 
marking $\Delta_1$ that has bounded length in $Y$ and some marking $\Delta_0$ 
that has both a bounded length and a bounded twisting in $X$. This map is 
finite-to-one except for some twisting information. 
An estimate for the number of possible markings $\Delta_0$ and $\Delta_1$ 
provides the desired upper-bound for the size of $\Bj X,Y, R)$.
\end{enumerate}

As is apparent from the outline, the main complication is to keep careful 
track of all the different twisting informations. Otherwise, the argument is
relatively elementary. 

\subsection{Markings on $S$}
Fix a set of points $\Sigma$ on $S$. A partial triangulation $T$ of $S$
with the vertex set $\Sigma$ is an embedding of a graph to $S$ where
vertices are mapped onto $\Sigma$ and the complementary components 
are either triangles or annuli. Even though the vertex set is fixed, we think of $T$ 
as representing a free homotopy class of triangulations. We say a curve $\gamma$ 
is carried by $T$ if the free homotopy class of $\gamma$ can be represented
by tracing the edges of $T$. We define a \emph{combinatorial} length of a simple 
closed curve $\gamma$ in $S$ to be the minimum number of arcs of $T$ that can 
appear in a representative of $\gamma$ and we denote it by $\ell_T(\gamma)$. 

Recall that a set of curves fill a subsurface $Q$ os $S$ if every essential
curve in $Q$ intersects one of these curves. We say a partial triangulation 
$T$ \emph{fills} a subsurface $Q$ of $S$ if, again, every essential curve
in $Q$ intersects $T$ (their free homotopy classes do not have disjoint
representatives). The two notions are related:

\begin{lemma}
There is a constant $\Bers$ such that, if $T$ fills a subsurface $Q$ of $S$, then the 
set of simple closed curves $\gamma$ carried by $T$ with $\ell_T(\alpha) \leq \Bers$ 
also fill the subsurface $Q$. \qed
\end{lemma}

\begin{definition} \label{Def:Marking}
A marking $\Delta= \Delta \big(\cS, \{E(\alpha)\}, T \big)$ for $S$ is:
\begin{itemize}
\item a free homotopy class of oriented curve system $\cS$ (pairwise disjoint curves) 
together with a notion of zero twisting for each curve $\alpha \in \cS$, (that is, the 
expression $\twist_\alpha(\Delta, \param)$ makes sense),
\item a length $E(\alpha)$ associated to each simple closed curve $\alpha \in \cS$, 
and
\item a homotopy class of a partial triangulation $T$ with the vertex
set $\Sigma$ such that the core curve of any annulus in the complement of $T$
is in $\cS$.
\item for each $\alpha \in \cS$ intersecting $T$, $\twist_\alpha(\Delta, T) = O(1)$.
\end{itemize}
We denote the set of simple closed curves that are disjoint from $T$ by $\cS^c$ and
the remaining curves in $\cS$ by $\cS^n$ (the set $\cS^c$ is a place holder for 
large cylinder curves and the set $\cS^n$ is a place holder for non-cylinder curves 
or small cylinder curves).

We say a marking $\Delta= \big(\cS, \{ E(\alpha)\}, T \big)$ has a 
\emph{bounded length} in $X$ if:
\begin{enumerate}
\item $\cS= \cS_X$. 
\item For $\alpha \in \cS$, $E(\alpha) = \Ext_X(\alpha)$. 
\item For $\alpha \in \cS^c$, $\twist_\alpha(\Delta, X) = O(1)$. 
\item For each simple closed curve $\gamma \not \in \cS_X$ that is disjoint from $\cS_X$, 
$\ell_X(\gamma) \emul \ell_T(\gamma)$.
\end{enumerate}
We say $\Delta$ has \emph{bounded length in $X$ with $\tau$--bounded
twist} if we further have
\begin{enumerate}
\item[(5)] For $\alpha \in \cS^n$, $\twist_\alpha(\Delta, X) =O(\tau)$.
\end{enumerate}
\end{definition}

\begin{example}
We continue \exref{TT} of a surface $(X,q)$ described by a 
gluing of a polygon in $\reals^2$. As it was discussed, there are two thick 
subsurfaces in the complement of curves $\alpha$ and $\beta$
(\figref{(X,q)}). A $(q, \tau)$--regular triangulation of $(X,q)$ is depicted in 
\figref{regular}. Here $\Sqn=\{\alpha\}$ and $\Sqc=\{\beta\}$.

\begin{figure}[ht]
\setlength{\unitlength}{0.01\linewidth}
\begin{tikzpicture}
  [
    thick, 
    scale=.05\unitlength,
    vertex/.style={circle,draw,fill=black,thick,
                   inner sep=0pt,minimum size=1mm}
    ]

   \node[vertex] (l0) at (0,0)   {};
   \node[vertex] (l1) at (0,1)   {};
   \node[vertex] (l2) at (0,7)   {};
   \node[vertex] (l3) at (0,13)  {};
   \node[vertex] (m0) at (10,0)  {};
   \node[vertex] (m1) at (10,1)  {};
   \node[vertex] (m2) at (10,13) {};
   \node[vertex] (r0) at (40,1)  {};
   \node[vertex] (r1) at (40,3)  {};
   \node[vertex] (r2) at (40,7)  {};
   \node[vertex] (r3) at (40,9)  {};
   \node[vertex] (r4) at (40,13) {};

   \draw (l0) -- (l1) 
       node at (-1, -.1) {1};
   \draw (l1) -- (l2) 
       node[midway,left] {2};
   \draw (l2) -- (l3) 
       node[midway,left] {2};
   \draw (l3) -- (m2) 
       node[midway,above] {3};
   \draw (m2) -- (r4) 
       node[midway,above] {4};
   \draw (r4) -- (r3) 
       node[midway,right] {5};
   \draw (r3) -- (r2) 
       node[midway,right] {6};
   \draw (r2) -- (r1) 
       node[midway,right] {5};
   \draw (r1) -- (r0) 
       node[midway,right] {6};
   \draw (r0) -- (m1) 
       node[midway,below] {4};
   \draw (m1) -- (m0) 
       node at (11,-.1) {1};
   \draw (m0) -- (l0) 
       node[midway,below] {3};
 
   \draw (l1) -- (m0); 
   \draw (l1) -- (m1);
   \draw (l2) -- (m1); 
   \draw (l2) -- (m2);
   \draw (m1) -- (m2); 
   
\end{tikzpicture}
\caption{A $(q, \tau)$--regular triangulation.} 
\label{Fig:regular}
\end{figure}

Here a marking $\Delta$ that has bounded length in $X$ can be obtained 
as follows: The set $\cS$ is the set $\{\alpha, \beta\}$ of short curves in $X$ 
(depicted as blue curves in \figref{Delta}), the triangulation $T$ is the 
$(q, \tau)$--regular triangulation (depicted as the red triangulation) and 
$E(\alpha)$ and $E(\beta)$ are the extremal length of $\alpha$ and $\beta$ in 
$X$ respectively. The condition $(4)$ for $\Delta$ to have a bounded length
in $X$ is a consequence of $T$ being a $(q, \tau)$--regular triangulation. 

\begin{figure}[ht]
\setlength{\unitlength}{0.01\linewidth}
\includegraphics[width=40\unitlength]{Short} \qquad
\includegraphics[width=40\unitlength]{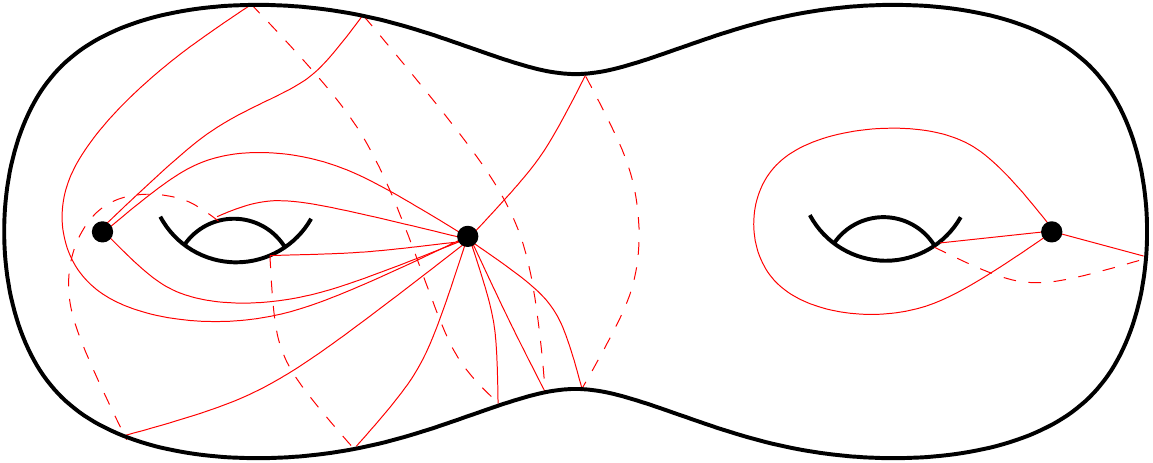}
\caption{The curves and the triangulation in the marking $\Delta$.}
\label{Fig:Delta}
\end{figure}

\end{example}

\begin{lemma} \label{Lem:TauBounded}
Let $M(X, \tau)$ be the set of markings $\Delta$ that have a bounded length 
in $X$ with $\tau$--bounded twist. Then
$$
|M(X,\tau)| \lmul \tau^{|\cS_X|}.
$$
\end{lemma}
\begin{proof}
The set $\cS$ and the lengths $\{ E(\alpha)\}$ and the twisting around 
curves in $\cS^c$ are determined by definition. By \lemref{up:to:twisting},
there is a uniformly bounded number of possibilities for $T$ up to
twisting around curves in $\cS^n$.
But each of these twisting parameters is bounded by multiple of $\tau$
(condition (5) in the definition \defref{Marking}).
This finishes the proof.
\end{proof}
\noindent

\begin{definition}
Consider the markings 
$$
\Delta= \{ \cS, \{ E(\alpha)\}, T \} 
\quad\text{and}\quad
\Delta_0=  \{ \cS_0, \{ E_0(\alpha_0)\}, T_0 \}.
$$ 
Recall that $T$ and $T_0$ have the same vertex set $\Sigma$. 
For every $\alpha \in \cS^c$, let $\beta_{\alpha}$
be an arc with end point in $\Sigma$ and disjoint from $T$
that crosses $\alpha$ so that $T \cup \beta_\alpha$ has
bounded twisting around $\alpha$. Denote
$$
U = T \cup \bigcup_{\alpha \in \cS^c} \beta_\alpha.
$$
Let $\R[U]$ be the vector space of formal sums with real coefficient
of edges in $U$. Let $\cR$ be a finite subset of $\R[U]$ with integer 
coefficients. We define the set
$$
M_\cR(\Delta,\Delta_0, \tau)
$$
to be the set of markings $\oDelta= \{ \oS, \{ \oE(\oalpha)\}, \oT \}$ such that:

\begin{enumerate}[(I)]
\item $\oDelta$ is a homeomorphic image of $\Delta_0$, and for every 
$\oalpha \in \oS^{c}$ that is the image of $\alpha_0\in\cS_0^{c}$, 
we have $E_0(\alpha_0)=\oE(\oalpha)$. 
\item For every element $\sum a_\omega \omega \in \cR$ and every
arc $\oomega \in \oT$ (respectively, $\oalpha \in \oS^c$), we have
$$
\sum_{\omega \in U} a_\omega \I(\omega, \oomega)= O(1),
\qquad \left( \text{respectively,} \quad
\sum_\omega a_\omega \I(\omega, \oalpha)= O(1) \right).
$$
\item Given $\alpha \in \cS^{c},$ $\oalpha \in \oS^c$, $\omega \in T$ 
and $\oomega \in \oT,$ we have the following bounds on the intersection 
numbers:
\begin{align}
\I(\omega, \oomega) & \lmul e^{\tau}  \label{Eq:one} \\
\twist_{\oalpha}(\Delta, \oDelta) \sqrt{E(\oalpha)}\, \I(\oalpha, T) 
      & \lmul e^\tau \label{Eq:two} \\
\twist_\alpha(\Delta, \oDelta) \sqrt{E(\alpha)} \, \I(\alpha, \oT) 
      & \lmul e^\tau  \label{Eq:three} \\
\I(\alpha,\oalpha)\, \twist_\alpha (\Delta, \oDelta) 
\twist_\oalpha(\Delta, \oDelta) \, \sqrt{\oE(\oalpha) E(\alpha)} 
      &\lmul e^\tau \label{Eq:four} \\
\intertext{and finally if $\alpha=\oalpha \in \cS^{c}\cap \oS^{c}$ we have:}
\twist_\alpha(\Delta, \oDelta) \sqrt{E(\alpha)\oE(\oalpha)} 
      &\lmul e^\tau. \label{Eq:five}
\end{align}
\end{enumerate}
Note that the partial triangulations in $\Delta$ and $\oDelta$ are defined
up to homotopy. By above intersection bounds we mean that the
homotopy class of two partial triangulations have representations
with vertex set $\Sigma$ so that the above bounds
hold simultaneously. 
\end{definition}

Let $\langle \cR\rangle$ be the subspace of $\R[U]$ generated
by elements in $\cR$. We give the following upper bound for such markings: 
\begin{lemma} \label{Lem:MarkingCount}
Let $h_\cR=\dim(\R[U]/ \langle \cR \rangle)$. Then
$$
| M_\cR(\Delta, \Delta_0, \tau)| \lmul e^{h_{\! \cR} \tau} 
\prod_{\alpha \in \cS^c} \frac 1{\sqrt{E(\alpha)}}
\prod_{\alpha_0 \in \cS_0^c} \frac 1{\sqrt{E_0(\alpha_0)}}.
$$
\end{lemma}
\begin{proof}
For $\oDelta \in M_\cR(\Delta,\Delta_0, \tau)$ consider the weighted graphs
$$
\oW= \oT+ \sum_{\oalpha \in\oS^{c}}  m(\oalpha, \oDelta) \, \oalpha,
$$ 
where the weights on the edges of $\oT$ are $1$ and the weight 
$m(\oalpha, \oDelta) \in \natls$ are defined to be 
$$ 
m(\oalpha, \oDelta)
= \left\lfloor \twist_\oalpha(\Delta, \oDelta) \sqrt{\oE(\oalpha)} \right\rfloor.
$$
Define $\cW$ to be the set of weighted graphs induced by elements of 
$M_\cR(\Delta,\Delta_0, \tau)$:
$$
\cW=\left\{ \oW \ST \oDelta \in M_\cR(\Delta,\Delta_0, \tau) \right\} .
$$
The weighted graph $\oW$ essentially determines $\oDelta$ except that,
for $\oalpha \in \oS^c$, the value of $m(\oalpha, \oDelta)$ determines
$\twist_\oalpha(\Delta, \oDelta)$ only up to 
$\frac 1{\sqrt{\oE(\oalpha)}} = \frac 1{\sqrt{E_0(\alpha_0)}}$ 
possibilities (we have used the
floor function in defining $m(\oalpha, \oDelta)$). Hence, 
\begin{equation}\label{Eq:ine}
 |M_\cR (\Delta, \Delta_0, \tau)| 
    \leq \prod_{\alpha_0 \in \cS_0^c} \frac 1{\sqrt{E_0(\alpha_0)}}\; |\cW|.
\end{equation}
We proceed in two steps:
\subsection*{Step 1} 
Consider the set $\cE \subset U$ that forms a basis for the 
space $\R[U]/ \langle \cR \rangle$.  
First, we claim that the map 
$$
\cI\from \cW \rightarrow \natls^{h_\cR}, 
\qquad
\oW \rightarrow \Big(\I(\oW,\omega) \Big)_{\omega \in \cE}
$$
is finite to one, where $\I(\oW,\omega)$ is defined to be
$$
\I(\oW,\omega)=\sum_{\oomega \in \oT} \I(\oomega,\omega)+ 
\sum_{\oalpha \in \oS^{c}} m( \oalpha, \oDelta)  \, \I(\oalpha,\omega).
$$ 
Note that in general a weighted graph $\oW$ is determined by the intersection 
numbers of its edges with all the edges of $U$. The map $\cI$
records the intersection number with arcs in $\cE$.
To prove the claim, we need to show that, there are only finitely many 
possibilities for the intersection number of $\oW$ with the other edges of $U$. 

We can consider an $\upsilon \in U$ as the element 
$1 \cdot \upsilon \in \R[U]$. Then $\upsilon$ can be written as a linear combination 
elements in the generating set $\cE$ (which generates $\R[U]/ \langle \cR \rangle)$)
and $\cR$ (the relations). That is, there are constants $c_\omega$ 
and $d_R$ so that
$$
\upsilon = \sum_{\omega \in \cE} c_\omega \omega 
           + \sum_{R \in \cR} d_R R. 
$$
But the intersection number is linear hence, for every $\oomega \in \oW$, we have
$$
\I(\upsilon, \oomega) = 
\sum_{\omega \in \cE} c_\omega \I(\omega, \oomega) 
             + O \left( \sum_{R \in \cR} d_R \right). 
$$
But the constants $d_R$ depend only on the set $\cR$ and otherwise
are uniformly bounded. Hence, there are only finite number
of possibilities for $\I(\upsilon, \oomega)$. This proves the claim. 

\subsection*{Step 2} We bound the size of 
$\mathcal{I}(\cW) \subset \natls^{h}$ by obtaining upper bounds on intersection 
numbers of $\oW$ with arcs $\omega \in \cE$.
\begin{itemize}
\item First, if $\omega \in T$, \eqnref{one} implies that 
\begin{equation}\label{Eq:ArcArc}
\I(\oT, \omega) \lmul e^{\tau}.
\end{equation}
Also, for $\oalpha \in \oS^{c}-\cS^{c}$ the \eqnref{two} implies
$$
m(\oalpha, \oDelta) \I(\oalpha, \omega) \lmul
\twist_\oalpha(\Delta, \oDelta) \sqrt{\oE(\oalpha)}  \I(\oalpha, \omega)
\lmul e^{\tau}.
$$
Hence, 
\begin{equation} \label{Eq:ArcCurve}
\I(\oW, \omega) \lmul e^\tau.
\end{equation}

\item For arc $\beta_\alpha \in U$ where $\alpha \in \cS^c$, and arc
$\oomega \in \oW$ by \eqnref{three} we have
\begin{equation}\label{Eq:CurveArc}
\I(\beta_\alpha, \oomega) 
\lmul \I(\alpha, \oomega) \twist_\alpha(\Delta, \oDelta)
\lmul  \frac{e^\tau}{\sqrt{E(\alpha)}}.
\end{equation}
And for $\oalpha \in \oS^{c}-\cS^{c}$, by \eqnref{four} we have
\begin{equation}\label{Eq:CurveCurve} 
m (\alpha,\Delta) \I(\alpha, \oalpha) \lmul 
\twist_\oalpha(\Delta, \oDelta) \sqrt{\oE(\oalpha)} \I(\alpha, \oalpha)
\lmul  \frac{e^\tau}{\sqrt{E(\alpha)}}.
\end{equation}
Finally, if $\alpha= \oalpha \in \cS^c \cap \oS^c$, by \eqnref{five}
\begin{equation}\label{Eq:SameCurve} 
m(\oalpha,\oDelta) \I(\oalpha, \beta_\alpha) \lmul 
\frac{e^\tau}{\sqrt{E(\alpha)}}.
\end{equation}
\end{itemize} 

Now from Equations \eqref{Eq:ArcCurve}, \eqref{Eq:CurveArc},
\eqref{Eq:CurveCurve} and \eqref{Eq:SameCurve}, we get: 
\begin{align*}
|\cW| \lmul |\mathcal{I}(\cW)| 
&\lmul  \prod_{\begin{subarray}{c} \alpha \in \cS^c \\ 
          \beta_\alpha \in \cE \end{subarray}} 
         \frac {e^\tau}{\sqrt{E(\alpha)}} \times 
\prod_{\omega \in T\cap \cE} e^{\tau} \\
&\leq e^{|\cE|  \tau}  \prod_{\alpha \in \cS^{c}} \frac {1}{\sqrt{E(\alpha)}}.
\end{align*}
Now, applying \eqnref{ine}, we get
$$
| M_\cR(\Delta, \Delta_0,\tau)| \lmul e^{h_\cR \tau} 
\prod_{\alpha \in \cS^c} \frac 1{\sqrt{E(\alpha)}}
\prod_{\alpha_{0} \in \cS_0^c} \frac 1{\sqrt{E_0(\alpha_{0})}},
$$
which is as we claimed. 
\end{proof}

\begin{proof}[Proof of \thmref{LatticeCount}]
Let $Z \in \Bj  X,Y,\tau)$ and let $(X_a,q_a)$ and $(X_b,q_b)$ be the 
initial and the terminal quadratic differentials for the \Teich geodesic 
in $\Qs$ starting near $X$ and finishing near $Z \in \Gamma \cdot Y$, as before. 
There may be many choices for these quadratic differentials. We 
need to be a bit careful. 

\subsection*{Claim:} We can choose $(X_a, q_a)$ and $(X_b, q_b)$ so that 
for any $\alpha \in \Sn{q_a}$,
\begin{equation} \label{Eq:NonCylinder}
\twist_\alpha(X, q_a) =O(\tau).
\end{equation}

\begin{proof}[Proof of claim]
Assume $(X_{\hat a},q_{\hat a})$ and $(X_{\hat b},q_{\hat b})$
are some choice of initial and terminal points with associated regular triangulations
$T_{\hat a}$ and $T_{\hat b}$ that have $j$ common saddle connection.
But, assume that they do not satisfy \eqnref{NonCylinder}. 
We define $(X_a, q_a)$ to be the image of $(X_{\hat a},q_{\hat a})$
under an appropriate number of Dehn twists around curves in 
$\Sn{q_{\hat a}}$ to ensure \eqref{Eq:NonCylinder}
and let $(X_b, q_b)$ be the image of $(X_{\hat b},q_{\hat b})$ under the
same homeomorphism. We will show that $X_a$ and $X_b$ are still near
$X$ and $Z$. 

For $\alpha \in \Sn{q_{\hat a}}$,
if $\frac 1{\Ext_X(\alpha)} \lmul \tau$, by \lemref{Twist-q-X}. 
$$
\twist_\alpha(X_{\hat a}, q_{\hat a}) \lmul \frac 1{\Ext_X(\alpha)}.
$$
Hence, using the triangle inequality and \thmref{PRT}
\begin{equation} \label{Eq:Triangle} 
\twist_\alpha(X, q_{\hat a}) \ladd \twist_\alpha(X, X_{\hat a}) 
+ \twist_\alpha(X_{\hat a}, q_{\hat a})
=O \left(\frac 1{\Ext_X(\alpha)} \right) = O(\tau).
\end{equation}
Therefore, \eqref{Eq:NonCylinder} is holds and no modification is required. 

Now, assume $\frac 1{\Ext_X(\alpha)} \gmul \tau$. 
Since $\alpha$ is a non-cylinder curve, $\frac 1{\Ext_{X_t}(\alpha)}$ changes at most linearly with time (\eqnref{ext:curve}). Hence, 
for $\tau$ large enough, we have
\begin{equation} \label{Eq:SameLength}
\frac 1{\Ext_X(\alpha)} \emul \frac 1{\Ext_Z(\alpha)}.
\end{equation}
Again by \lemref{Twist-q-X}, the number of Dehn twists $n_\alpha$
around $\alpha$ that needs to be applied to $q_{\hat a}$ to ensure 
\eqnref{NonCylinder} is at most $O(1/\Ext_X(\alpha))$. That is, 
$$
X_a = \prod_{\alpha \in \Sn{q_{\hat a}}} D_\alpha^{n_\alpha} X_{q_0},
$$
where $D_\alpha$ is a Dehn twist around $\alpha$ and 
$n_\alpha \lmul \frac 1{\Ext_X(\alpha)}$. By, \thmref{PRT}
$$
d_\T(X_{\hat a}, X_a) \emul \sum_{\alpha \in \Sn{q_{\hat a}}} 
n_\alpha \Ext_{X}(\alpha) \lmul 1.
$$
and 
$$
d_\T(Y_{\hat b}, X_b) \emul \sum_{\alpha \in \Sn{q_{\hat a}}} 
n_\alpha \Ext_{Z}(\alpha)\lmul 1.
$$
Hence, $(X_a, q_a)$ and $(X_b, q_b)$ are as desired. Also, 
the images $T_a$ and $T_b$ of $T_{\hat a}$ and $T_{\hat b}$ are
still regular triangulations and have $j$ arcs in common. 
\end{proof}

For the rest of the proof, we assume \eqnref{NonCylinder}
holds. To the pair $(X_a, q_a)$ we associate the marking 
$\Delta=\{ \cS, \{ E(\alpha)\} , T \}$ as follows:
\begin{itemize}
\item Let $\cS$ be the set of short curve in $X$ and set
$E(\alpha) = \Ext_X(\alpha)$.
\item Let $T$ be the $(q_a, \tau)$--regular triangulation $T_a$
which has $j$ edges in common with the triangulation $T_b$. 
\item If $\alpha \in \Sc{q_a}$ then set the twisting around $\alpha$
in $\Delta$ so that
$$
\twist_\alpha(\Delta, X) = O(1).
$$
\item If $\alpha \in \Sn{q_a}$ then set the twisting around 
$\alpha$ in $\Delta$ so that 
$$\twist_\alpha(\Delta, T) = O(1).$$
\end{itemize}
The result is a marking that has bounded length in $X$ and 
(by \eqnref{NonCylinder}) has $\tau$--bounded twist 
in $X$. Also, note that $\cS^c = \Sc{q_a}$ and $\cS^n = \Sn{q_a}$.

We can similarly associate a marking $\oDelta$ to the pair $(X_b, q_b)$.
Here we can only conclude that $\oDelta$ is bounded in $Z$
(not with bounded twist); this is because the inequality \eqref{Eq:NonCylinder} 
does not necessarily hold for $Z$ and $q_b$. Instead, similar
to  \eqnref{Triangle}, we have
\begin{equation} \label{Eq:TwistZ}
\twist_\oalpha(Z, q_b) \lmul \frac 1{\Ext_Z(\oalpha)}.
\end{equation}
Assume $Z = \gg(Y)$, for $\gg \in \Gamma(S)$. 
Let $\Delta_0 = \gg^{-1}(\oDelta)$. Then $\Delta_0$
in bounded in $Y$. Also, let $\cR$ be the elements in $\R[U]$
coming from \lemref{Triangle}, (and \lemref{NonOrientable}
in case quadratic differentials in $\Qs$ are not orientable) 
and the $j$ edges in $T$ that are present in the $(q_b, \tau)$--regular 
triangulation $T_b$. Taking this $T_b$ is the partial triangulation in 
$\oDelta$, we have $\oDelta \in M_\cR(\Delta, \Delta_0, \tau)$. 
The number of possible choices for $\Delta$ is $O(\tau^{|\cS_X|})$ 
(\lemref{TauBounded}) and there are finitely many choices 
of for the homeomorphism type of $\Delta_0$. 
\lemref{MarkingCount} provides an 
upper-bound for the size of the set $M_\cR(\Delta, \Delta_0, \tau)$. 
Also, using the fact that $\oDelta$ is bounded in $Z$ and \eqnref{TwistZ}, 
similar to \lemref{TauBounded}, we can conclude that the association 
$Z \mapsto \oDelta$ is at most 
$O\left(\prod_{\alpha \in \Sn{q_b}}\frac 1{\Ext_Z(\alpha)}\right)$--to-one. 

To summarize, we have defined a map from $\Bj X,Y,\tau)$ to
the union of sets of markings $M_\cR(\Delta, \Delta_0, \tau)$, where
$\Delta$ is bounded $X$ with $\tau$--bounded twist and $\Delta_0$ is bounded
in $Y$. The map is not one-to-one but we have a bound on the 
multiplicity. 

The size of $\Bj X,Y,\tau)$ is comparable to the product of the following:
the number of choices for $\Delta$, the number of choices for the 
homeomorphism class of $\Delta_0$, the maximum multiplicity of the association 
$Z \mapsto \oDelta$ and the size of $M_\cR(\Delta,\Delta_0, \tau)$.
That is, 
\begin{align*}
|\Bj X,Y,\tau)| &\lmul |M(X,\tau)| \times O(1) \times
\prod_{\alpha \in \Sn{q_b}}\frac 1{\Ext_Z(\alpha)} \times
|M_\cR(\Delta, \Delta_0, \tau)|\\
& \lmul \tau^{|\cS_X|} 
\prod_{\alpha \in \Sn{q_b}}\frac 1{\Ext_Z(\alpha)} 
e^{h_\cR\tau} \prod_{\alpha \in \cS^c} \frac 1{\sqrt{E(\alpha)}}
\prod_{\alpha_0 \in \cS_0^c} \frac 1{\sqrt{E_0(\alpha_0)}}\\
&\lmul \tau^{|\cS_X|+ |\cS_Y|} e^{h_\cR \tau} 
\prod_{\alpha \in \cS} \frac 1{\sqrt{E(\alpha)}}
\prod_{\alpha_0 \in \cS_0} \frac 1{\sqrt{E_0(\alpha_0)}}.
\end{align*}
The last line follows from the previous line because, for every term in 
the product $\prod_{\alpha \in \Sn{q_b}}\frac 1{\Ext_Z(\alpha)}$
we either have $\frac 1{\Ext_Z(\alpha)} = O(\tau)$ or,
as in \eqnref{SameLength}, 
$$
\frac 1{\Ext_Z(\alpha)} \lmul \frac 1{\sqrt{\Ext_Z(\alpha)}} \frac 
1{\sqrt{\Ext_X(\alpha)}},
\quad\text{and}\quad
\alpha \in \Sn{q_a} \cap \Sn{q_b} = \cS^n \cap \cS_0^n.
$$ 
That is, each term can either be counted in the power of $\tau$ in
the beginning of last line or it can be divided into a term in each of
the last two products. The proof is finished after checking that
$h_\cR = (h-j)$. This is true because all the relations in 
\lemref{Triangle} are also relations in $H_1(S, \Sigma)$. The fact that the 
$j$ arcs we have fixed in $T_a$ are homologically independent implies that 
these arcs and the other relations in homology are independent
in $\R[U]$. In fact, \lemref{NonOrientable} is used only when $\vs = -1$. 
But this is accounted for in the definition of $h$ (see \secref{Period}).
Hence, the dimension of $\R[U]/\langle \cR \rangle$ is exactly $j$ less
than $h= \frac{\dim \Stratum+1}{2}$. 
\end{proof}

\section{Geodesics in the thin part of moduli space}
\label{Sec:thin:part}

In this section we prove \thmref{short:saddle} and \thmref{j:short}.
The main idea, which is due to Margulis, is to prove an
inequality, which shows that the flow (or more precisely an
associated random walk) is biased toward a compact part of  the space.
Consider the stratum $\Qs$.  We discretize the projection
$$
  \pi(\Qs) \subset \cT(S),
$$
by fixing an appropriate net $\Nt$ in $\cT(S)$. 
Then, we consider the random walk $\{\lambda_{i}\}_{i\geq 0}$ on the points 
in $\Nt$ and apply \thmref{LatticeCount} to show that the projection of this 
random walk in $\Moduli$ is biased towards the compact subset of $\Moduli$. 
Moreover, we show that quadratic differentials 
$\{q(\lambda_{i},\lambda_{i+1})\}_{i\geq 0}$ (see \secref{Tt})
tend not to have short saddle connections. See \lemref{est:tilde:P:j} for the precise 
formulation. 

These estimates imply \thmref{short:saddle}; this is because, roughly speaking, 
every closed geodesic in $\Stratum$ can be approximated by a path along the 
net points.

\subsection{Short saddle connections and simple closed curves}
\label{Sec:ssc}
For a quadratic differential $(X,q) \in \QT$, recall the set of
\emph{short saddle connections} $\Omega_q(\ep)$ (\defref{Omega_q}). 
Define $s(q,\ep)$ to denote the maximum number of homologically independent 
disjoint saddle connections in $\Omega_q(\ep)$. Given the tuple $\sigma$, define
$$
\Qje(\sigma)= 
  \Big\{ (X,q) \in \Qs \ST s(q ,\ep) \geq j \Big\} 
  \subset  \QT.
$$
For the rest of this section, with fix $\sigma$ and denote $\Qje(\sigma)$
simply by $\Qje$. Also, recall the definition of $B(\Qs, X, \tau)$ 
from \secref{notation} and $\Bj X, \tau)$ from \secref{main:counting}. 
We would like to refine the definition of $\Bj X, \tau)$. Roughly speaking, 
we are interested in a ball of
radius $\tau$ centered at $X$ that is allowed to move in the
direction $\Qje$ only. Namely, define
$B(\Qje, X, \tau)$ to be the set $Z \in \T(S)$ so that 
\begin{itemize}
\item $Z \in B(\Qs, X, \tau)$
\item for the associated quadratic differential $q_a$, we have $s(q_a, \ep) \geq j$. 
\end{itemize} 
One can similarly define $B(\Qje, X, Y, \tau)$ as in \secref{main:counting}. 
Recall the choice of $\ep_1(\tau)$ from \lemref{Exist}. 
\begin{lemma}\label{Lem:quad}
For any $\tau>0$, there is  $\ep_2(\tau) < \ep_1(\tau)$ such that 
for $\ep< \ep_{2}(\tau)$, any integer $j \geq 0$, 
and any $X, Y \in \cT(S)$, we have
 \begin{equation}\label{Eq:qu1}
B(\Qje, X, \tau) \subset \Bj X, \tau),
 \end{equation}
 and 
 \begin{equation} \label{Eq:qu2}
B(\Qje, X, Y, \tau) \subset \Bj X, Y, \tau).
\end{equation}

\end{lemma}

\begin{proof}
It is enough to let $\ep_{2}(\tau)=e^{-2\tau} \ep_1$. 
Assume, $Z \in B(\Qje, X,\tau)$, $q_a$ and $q_b$ are the associated
quadratic differentials and $(b-a) < \tau$. Let $\omega_1,\ldots, \omega_j$ 
be disjoint homologically independent saddle connections counted
in $s(q_a, \ep)$. Then, for each $i$, by \eqnref{ext:saddle}, 
$$
\left|\frac 1{\Ext_{q_a}(\omega_i)} - \frac 1{\Ext_{q_b}(\omega_i)}\right| 
  \leq \tau, 
$$
and by \thmref{ker} the extremal length of any short curve $\alpha$ 
containing $\omega_i$ changes by at most a factor of at most $e^{2\tau}$. 
That is, $\omega_i \in \Omega_{q_b}(\ep_1)$. 
The arcs $\omega_i$ are still disjoint and homologically independent in $q_b$. 
Hence, the set $\{\omega_i\}$ can be extended to both a $(q_a, \tau)$--regular 
triangulation $T_a$ and a $(q_b, \tau)$--regular triangulations $T_b$ 
(\lemref{Exist}). Thus, by the definition $Z \in \Bj X, \tau)$. The proof of 
\eqnref{qu2} is similar.
\end{proof}

\subsection{Choosing a net} \label{Sec:Net} 
By a $(c,2c)$--separated net $\cN \subset \Moduli$ we mean a set of 
points in $\Moduli$ so that:
\begin{itemize}
\item the \Teich distance between any two net points in $\cN$ is at least
$c$, and 
\item any point in $\Moduli$ is within distance $2c$ of a point in $\cN$.
\end{itemize}
Let 
$$
\cN(X, \tau) = \bp(B (X, \tau)) \cap \cN.
$$ 

Then, it is easy to check (see Lemma 3 in \cite{EM:geodesic-counting}): 

\begin{lemma} \label{Lem:polynomial}
There exists a constant $c_{0}>0$ such that for any $c>c_{0}$, 
and $(c,2c)$ net $\Nt$ as above, we have 
\begin{equation}
\label{Eq:netpoints}
 | \cN(X, \tau) | \lmul \tau^{3g-3+p}.
\end{equation}
\end{lemma}

Let $\Nt = \bp^{-1}(\cN)$. We assume the $r_0>2c$, where $r_0$ is
the constant used to define $B (\Qs,  X, \tau)$ (see \secref{notation}).
We denote the intersection of a ball in \Teich space, $B(\param)$, with $\Nt$ by 
$\Nt(\param)$. That is, for $X,Y \in \T(S)$, 
\begin{align*}
\Nt(\Qs, X, \tau)&= B (\Qs,  X,\tau) \cap \Nt,\\
\Nt(\Qje, X, \tau)&= B (\Qje,  X,\tau) \cap \Nt,\\
\Nt(\Qs, X, Y, \tau)&= B (\Qs,  X, Y, \tau) \cap \Nt,\\
\intertext{and}
\Nt(\Qje, X, Y, \tau)&= B (\Qje,  X, Y, \tau) \cap \Nt.
\end{align*}

\subsection{The main inequality.} \label{Sec:ineq}
For a real-valued function $f \from \cM(S) \to \reals$, consider the
average function
$$
\big(\Aj  f \big)\from \cT(S) \rightarrow \reals,
$$
defined by 
\begin{displaymath}
\big(\Aj f\big)(X) = e^{-h\tau} \sum_{Z \in  \Nt(\Qje, X ,\tau)} f(Z).
\end{displaymath}
Here, as before
$$ h=\frac{\dim \Stratum +1}{2}.$$

Our main tool is the following ($\ep_{2}(\tau)$ is as in \lemref{quad}):
\begin{proposition}
\label{Prop:inequalities:saddles}
Given $\tau>0$, and $\ep< \ep_{2}(\tau)$ we have
\begin{equation}
\label{Eq:Main2}
\big(\Aj G\big)(X) \lmul \tau^m  e^{-j\tau}G(X).
\end{equation}
where $G$ is as in \eqnref{GG} and $m$ depends only on the topology of $S$.
\end{proposition}

\begin{proof}
Enumerate the elements of  $\cN(X, \tau)$ as $y_1, \ldots, y_k$ and 
let $Y_i \in \T(S)$ be a pre-image of $y_i$, $i=1, \ldots, k$. 
By \lemref{count}, every net point in $Z \in B (\Qje, \Nt, X, Y_i, \tau)$ is near 
at most $G(Y_i)$ points in $B(\Qje,  X,Y_i, \tau)$. That is, 
\begin{equation}\label{Eq:size} 
\Big| \Nt(\Qje, X, Y_i, \tau)  \Big| \, G(Y_i)^2 \emul
 \Big| B(\Qje, X, Y_i, \tau)  \Big| 
\end{equation}
Hence, we have
\begin{align*}
\big(\Aj G\big)(X) &= e^{-h\tau} \sum_{Z \in \Nt(\Qje, X, \tau)} G(Z) \\
& =  e^{-h\tau} \sum_{i=1}^k \sum_{Z \in \Nt(\Qje, X, Y_i, \tau)} G(Y_i) \\
& \emul e^{-h\tau} \sum_{i=1}^k 
 \frac {\Big|B(\Qje, X, Y_i, \tau) \Big|}{G(Y_i)}
 \tag{\eqnref{size}} \\
& \lmul e^{-h\tau} \sum_{i=1}^k  \tau^{|\cS_X| + |\cS_Y|} 
e^{(h-j)\tau} G(X)  \tag{ \thmref{LatticeCount} and \eqref{Eq:qu2}}\\
& \leq e^{-j\tau} \tau^m G(X). \tag{\eqnref{netpoints}} 
\end{align*}
Here, $m =  (9g-9+3p) \geq |\cS_X| + |\cS_Y| + (3g-3+p)$. 
\end{proof}

\subsection*{Trajectories of the random walk.}
Suppose $R \GG \tau$ and let $n$ be the integer part of $R/\tau$. By a 
\emph{trajectory of the random walk} we mean a map 
$$
\lambda \from \{ 0, n\} \to \Nt \subset \cT(S)
$$ 
such that, for all $0 < k \le n $, we have 
$d_\cT(\lambda_k , \lambda_{k-1})\leq \tau$, where $\lambda_k=\lambda(k)$. 
Let $\cP_{\tau} (X, R)$ denote the set of all trajectories for which
$d_{\cT}(\lambda_0,X) = O(1)$. 
For $j \in \natls$, let $\Pjt X,R)$ denote the set of all trajectories 
$\lambda \in \cP_\tau(X,R)$ so that, 
\begin{itemize}
\item 
for $1\le k \le n$
$$
\lambda_k \in \Nt(\Qs, \lambda_{k-1}, \tau).
$$
\item 
$$ 
\Big| \big\{k \st 1 \leq k \leq n, \; \lambda_k \in B(\Qje, \lambda_{k-1}, 
  \Nt,\tau) \big\}\Big| \geq \theta \cdot n.
$$
\end{itemize}
Given $X, Y \in \cT$, let $\Pjt X,Y,R)$ denote the set of all trajectories 
$\lambda \in \Pjt X,R)$ such that $$
d_\cT\big(\pp(Y), \pp(\lambda_n)\big) =O(1). $$
We say that a trajectory is {\em almost closed in the quotient} if 
$$
d_\cT\big(\pp(\lambda_0), \pp(\lambda_n)\big) =O(1). 
$$
Finally, let $\Pjht X,R)= \Pjt X, X, R)$ denote the subset of these trajectories 
starting from $X$ which are almost closed in the quotient.
Let $\ep_{2}(\tau)$ be as in \lemref{quad} and \propref{inequalities:saddles}.

\begin{lemma}
\label{Lem:est:tilde:P:j}
For any $\delta_0> 0$ there is $\tau_0>0$ so that
for $\tau>\tau_0$, $0\le\theta\leq1$ and $\ep<\ep_{2}(\tau_0)$ we have
\begin{equation}
\label{Eq:est:tilde:P:j:2}
\big|\Pjt X,Y, R)\big| \lmul e^{(h-j \theta+\delta_0) R}\; \frac{G(X)}{G(Y)}.
\end{equation}
In particular,
\begin{equation}\label{Eq:est:tilde:P:j}
\big|\Pjht X,R)\big| \lmul e^{(h-j \theta+\delta_0) R}.
\end{equation}
\end{lemma}

\begin{proof}
Define
$$ 
q_{k}(\lambda)=  \Big| \big\{i  \, \big| \, 1 \leq i \leq k, \; \lambda_i \in 
\Nt(\Qje,  \lambda_{i-1},\tau) \big\} \Big|.
$$
This keeps track of the number of steps in the trajectory $\lambda$ 
(amount the first $k$ steps) that can be approximated by a segment in $\Qje$. 
For $0 < r=k\tau < R$, let $\Pjht X,Y, R, r)$ be the set of trajectories 
obtained from a trajectory $\lambda \in \Pjht X,Y, R)$ but truncated after 
$k = r/\tau$ steps. Define
\begin{displaymath}
V_{\tau}(R,r) = \sum_{\lambda \in \Pjt X,Y, R, r)} 
 G(\lambda_k) e^{j \, q_k(\lambda)\tau}.
\end{displaymath}
Also, let $R = n \tau$, $q(\lambda) = q_n(\lambda)$ and
\begin{displaymath}
V _{\tau}(R) = \sum_{\lambda \in \Pjt X,Y, R)} 
   G(\lambda_n) e^{j q(\lambda) \tau}.
\end{displaymath}
Note that $G(Y) \emul G(\lambda_n)$ and $q(\lambda) \tau \geq \theta R$. 
Therefore,
\begin{equation}
\label{Eq:frac:estimate}
|\Pjt X,Y, R)| \lmul  \frac{V_{\tau}(R)}{G(Y) e^{j \theta R}}.
\end{equation}

If $\lambda_{k+1} \in  \Nt(\Qje, \lambda_{k}, \tau)$
then $q_{k+1}(\lambda)=q_{k}(\lambda)+1$ and 
$q_{k+1}(\lambda)=q_{k}(\lambda)$ otherwise. Hence,

\begin{align*}
V_{\tau}(R,r+\tau) 
& = \sum_{\lambda \in \Pjt X,Y, R, r+\tau)} 
   G(\lambda_{k+1}) e^{j q_{k+1}(\lambda)\tau} \\
& \leq  \sum_{\lambda \in \Pjt X,Y, R, r) }
 \left( \sum_{\lambda_{k+1}  \in  \Nt(\Qje,  \lambda_k,\tau)} 
   G(\lambda_{k+1}) e^{j (q_k(\lambda)+1) \tau}+ \right. \\  
 & \left. \qquad \qquad  \qquad \qquad \qquad \qquad \quad+
 \sum_{\lambda_{k+1} \not \in  \Nt(\Qje,  \lambda_k,\tau)} 
 G(\lambda_{k+1}) e^{j q_k(\lambda)\tau} \right).\\
\end{align*}
The two summands inside of the parenthesis are similar to the average
defined above. Using \eqnref{Main2}, the first term is less than
(up to a multiplicative error) 
$$
e^{j (q_k(\lambda)+1)\tau} e^{h\tau}(\Aj G)(\lambda_k)
\lmul e^{j (q_k(\lambda)+1)\tau} e^{h\tau} \tau^m e^{-j\tau} 
G(\lambda_k).
$$ 
and the second term is less than (again, up to a multiplicative error)
$$
e^{j q_k(\lambda)\tau} e^{h\tau} (A_{\tau,0} G)(\lambda_k)
\lmul e^{j q_k(\lambda)\tau} e^{h\tau} \tau^m G(\lambda_k).
$$ 
Note that the right hand sides of the above two equations are the same. 
Hence, 
\begin{align} \label{Eq:q:Delta:iter}
V_{\tau}(R,r+\tau) &\lmul  \tau^m e^{h\tau} 
\sum_{\lambda \in \cP_{\tau}(\Stratum, X,R,r)}
 e^{j q_k(\lambda)\tau} G(\lambda_k) \notag \\
 & = \tau^m \, e^{h \tau} \, V_{\tau}(R,r).
 \end{align}
Now iterating \eqref{Eq:q:Delta:iter} $n= R/\tau$ times we get
\begin{equation}
\label{Eq:qj:estimate}
V_{\tau}(R) \le (C\; \tau)^{mn} \, G(X) \, e^{h n\tau}
= G(X) e^{(h+ \frac{m (\log (\tau)+\log(C))}{\tau}) R},
\end{equation}
where $C>0, $ and $m \in \natls$ are uniform constants.
We choose $\tau$ large enough so that 
$$
\frac{m \log(\tau)+\log(C)}{\tau} < \delta_0.
$$ 
The lemma follows from \eqnref{frac:estimate} and \eqnref{qj:estimate}.
\end{proof}

Let $\Njt X,Y, R)$ be the number points $Z \in B(\Qs, X, Y, R)$ 
(see \secref{main:counting} for definition) 
so that associated geodesic $(X_t, q_t)$ spends $\theta$
proportion of time in $\Qje$. Similarly, for $x \in \Moduli$, let 
$N_\theta(\Cje, x,R)$ 
be the number of conjugacy classes mapping classes 
associated to closed geodesics $\bg$ in $\Stratum$ of length at most $R$ 
which pass within a uniformly bounded distance of the point $x$ and so that 
for at least $\theta$ fraction of the points $(x_t, q_t) \in \bg$, 
$s(q_t,\delta) \ge j$ (see \secref{ssc}). As we shall see in the proof of 
the lemma below, for $x=\bp(X)$, $\Njt X,X, R)$ may be much larger than 
$N_\theta(\Cje, x, R)$.

\begin{lemma} \label{Lem:geodeiscs:to:trajectories}
For any $\delta_1> 0$, there is $\tau_1$ so that, for $\tau>\tau_1$
$X \in \cT(S)$ and any sufficiently large $R$ 
(depending only on $\delta_1, \tau$) we have
\begin{equation}
\label{Eq:Nx:tilde:S}
N_\theta(\Cje, \bp(X), (1-\delta_1)R) \lmul  \Big| \Pjt X, R) \Big|,
\end{equation}
and
\begin{equation}
\label{Eq:Nx:tilde:S:1}
\Njt  X, Y, (1-\delta_1)R) \lmul  \Big| \Pjt X, Y, R) \Big| \; G(Y)^2.
\end{equation}

\end{lemma}

\begin{proof} 
Recall the definition of 
$$
  I_X=\big\{ \gg \in \Gamma(S) \,\big|\, d_{\cT}(X, \gg \cdot X) =O(1) \big\} .
$$ 
from \lemref{count}. Consider a closed geodesic $\bg$ in $\Stratum$ which 
intersects a uniformly bounded neighborhood of $x=\pp(X)$. Let $[\bg]$ denote 
the corresponding conjugacy class in $\Gamma(S)$. Then there are approximately
$|I_X|$ lifts of $[\bg]$ to $\cT_g$ which start within a bounded distance of $X$. 
Each lift $\cG$ is a geodesic segment of length equal to the length of
$\bg$.

We can mark points distance $\tau$ apart on $\cG$, and replace 
these points by the nearest net points in $\Nt$. (This replacement is the 
cause of the $\delta_1R$ error). This gives a map $\Psi$ from lifts of geodesics 
to trajectories. If the original geodesic $\bg$ has length at most $(1-\delta_1)R$
and has $s(q_t, \delta) \geq j$ for $\theta$ fraction of its points, then  
the resulting trajectory belongs to $\Pjt X, R)$.

If two geodesic segments map to the same trajectory, then the segments
fellow travel within $O(1)$ of each other. In particular if $\gg_1$ and
$\gg_2$ are the pseudo-Anosov elements corresponding to the two
geodesics, then $d_{\cT}(\gg_2^{-1} \gg_1 X, X) = O(1)$, thus $\gg_2^{-1} 
\gg_1 \in
I_X$.

We now consider all possible geodesics contributing to
$N_\theta(\Cje, x,(1-\delta_1)R)$; for each of these we consider all the
possible lifts which pass near $X$, and then for each lift consider
the associated random walk trajectory. We get:
\begin{displaymath}
|I_X|\, N_\theta(\Cje, x, (1-\delta_1)R)  \lmul  
|I_X| \, \big|\Pjt X,R)\big|.
\end{displaymath}
The factor of $|I_X|$ on the left hand side is due to the fact that
we are considering all possible lifts which pass near $X$, and the
factor of $|I_X|$ on the right is the maximum possible number of times
a given random walk trajectory can occur as a result of this process. Thus,
the factors of $|I_X|$ cancel, and the lemma follows.
Note that by \lemref{count} (See also, \thmref{LatticeCount}) 
 $|I(Y)| \lmul G(Y)^2$. An argument similar to the proof of the 
first part implies \eqnref{Nx:tilde:S:1}.
\end{proof}

We need he following lemma which is due to Veech \cite{Veech}. 
\begin{lemma}
\label{Lem:veech}
Suppose $\bg$ is a closed geodesic of length at most
$R$ on $\cM(S)$. Then for any $x \in \bg$, any $X$ so that $\bp(X)=x$
and every simple closed curve 
$\alpha$
$$
\Ext_X(\alpha) \gmul e^{-(6g-4+2p)R}.
$$
\end{lemma}

\begin{proof}[Proof of \thmref{j:short}]
Let $\delta>0$. Choose $\delta_0, \delta_1 \leq \delta/3$. 
Now choose $\tau\geq \max\{\tau_0, \tau_1\}$ and let $R$ be large enough 
so that Equations \eqref{Eq:est:tilde:P:j} and \eqref{Eq:Nx:tilde:S} hold. 
We get,
\begin{equation}
\label{Eq:Nx:Q}
N_\theta(\Cje, x, R) \lmul  e^{(h-j \theta +2\delta/3)R}.
\end{equation}
Finally 
$$
N_\theta(\Cje R) \leq \sum_{x \in \cN} \Njt  x,R),
$$
where $\cN$ is the net chosen above. In view of \lemref{veech} and 
\lemref{polynomial}, the number of relevant points 
in the net is at most polynomial in $R$. However, for $R$ large enough, this
polynomial is less than $e^{\delta R/3}$. Thus \thmref{j:short} follows.
\end{proof}

\begin{proof}[Proof of \thmref{short:saddle}]
Let $\bg$ be a closed geodesic in $\Stratum \setminus \cK$. 
By taking $\cK$ large enough we can assure that every quadratic differential 
along $\bg$ has an arbitrary short saddle connection. 
We choose $\cK$ so that \lemref{ext:saddle} implies that any
such quadratic differential $(x,q)$, $\Omega_q(\ep)$ is non-empty
for $\ep \leq \ep_{2}(\tau)$. Hence the number of disjoint homologically independent 
saddle connections in $\Omega_q(\ep)$ is at least one. That is, $\bg$ is counted in 
$N_\theta(\Cje, R)$ for $j=1$ and $\theta=1$. The theorem now follows from 
\thmref{j:short}. 
\end{proof}

\begin{proof}[Proof of \thmref{j:short:XY}]
We can use the argument applied in the proof of \thmref{j:short}. 
Let $0< \delta_0,\delta_1 \leq \delta/3$.  Choose a net  satisfying 
\lemref{polynomial}. Then choose $\tau\geq \max\{\tau_0, \tau_1\}$ and let 
$R$ be large enough so that Equations \eqref{Eq:est:tilde:P:j:2} and \eqref{Eq:Nx:tilde:S:1} hold. As in the proof of \thmref{j:short}, 
\eqnref{XYr} follows from \lemref{polynomial} and \lemref{veech}.
\end{proof}

\section{The Hodge Norm and the Hodge Distance.} \label{Sec:Hodge}
In this section, we use the Hodge norm \cite{forni} to show that in
any compact  subset of $\Stratum$ the geodesic flow is uniformly
hyperbolic: see \cite{ABEM} and \remref{hyperbolicity} below. 
There are many approaches to proving hyperbolic like behavior for
the \Teich geodesic flow in different settings, see for example \cite{AGY, AG, forni, 
Hamenstadt:BM, Veech}.

Let $\cH^1\cT(S)$ be the bundle of area one abelian differentials over
$\T(S)$. We also denote by $g_t$ the geodesic flow on 
$\cH^1\cT(S)$ (where we square an abelian differential to get 
a quadratic differential). 

\subsection{Hodge norm.}
Fix a point $(X,\phi)$ in $\cH^1\cT(S)$,
where  $X \in \T(S)$ and $\phi$ is an abelian differential on $X$. 
Let $\pi\from \cH^1\cT(S)  \to \cT(S)$ and 
$\bp \from \cH^1\cT(S) \to \cH^1\Moduli$ be natural maps as in \secref{maps}. 
Let $\|\param \|_{H,t}$ denote the Hodge norm on the surface 
$X_t = \pi(g_t \phi)$. Also, for each abelian differential $\phi$, 
let $\Re(\phi), \Im(\phi)\in H^1(X, R)$ be
forms obtained by the real part and the imaginary part of the holonomy.

The following fundamental result is due to Forni \cite[\S{2}]{forni}:
\begin{theorem}
\label{Thm:forni}
For any $\lambda \in H^1(X,\reals)$ and any $t \ge 0$,
\begin{displaymath}
\|\lambda\|_{H,t} \le e^{t} \| \lambda\|_{H,0}.
\end{displaymath}
If, in addition, $\lambda \wedge \Re(\phi) = \lambda \wedge
\Im(\phi) = 0$ and, for some compact subset $\cK$ of $\cH^1 \Moduli$,
the segment $[\phi,g_t \phi]$ starts and ends in $\bp^{-1}(\cK)$
and spends at least half the time in $\bp^{-1}(\cK)$, then we have
\begin{displaymath}
\|\lambda\|_{H,t} \le e^{(1-\alpha)t} \| \lambda\|_{H,0},
\end{displaymath}
where $\alpha > 0$ depends only on $\cK$.
\end{theorem}

 \thmref{forni} gives a partial hyperbolicity property of
 the geodesic flow on spaces of abelian differentials. In our
 application, we need a similar property for compact subsets of the spaces 
$\QMs$ of quadratic differentials.

\subsection{Quadratic and abelian differentials} \label{Sec:quad}
Here, we briefly treat the case when $q \in \mathcal{Q}\mathcal{M}(S)$ is 
not the global square of an abelian differential. 
A standard construction, given $X \in \T(S)$ and $q$ a quadratic 
differential on $X$, is to pass to the possibly ramified double cover on which the 
foliation defined by $q$ is orientable. More precisely, we consider the canonical 
(ramified) double cover 
$\phi \from \tilde{X} \rightarrow X$ such that $\phi^{*}(q)=\phi^2$.
(See the proof of \lemref{NonOrientable} for the explicit construction.)
The set of critical values of $\phi$ coincides with the set of zeros of $q$ 
with odd degree.  

This yields a surface $\tilde{X}$ with an abelian differential $\phi$. However, 
even if $\bp(X)$ belongs to a compact 
subset of $\Moduli$, there may be a curve that has a very small 
extremal length in $\tilde{X}$. This may occur since 
the flat structure defined by $q$ may have an arbitrarily short saddle connection
connecting distinct zeroes. Such a saddle connection lifts to a very short loop 
in the double cover. Let $\ell_{min}(q)$ denote the length of the shortest
saddle connection in the flat metric defined by $q$. We have, 
$$\ell_{min}(\phi) \geq \ell_{min}(q). $$
That is, if $q$ does not have any short saddle connection, then 
$\phi$ also does not have any short saddle connections.

\subsection{The Hodge norm on relative cohomology.}
Let $(X,q) \in \QTs$ and let $\Sigma$ be the set of singularities of $q$. 
Let $\tilde X$ be as before and $\tilde \Sigma$ be the pre-image
of $\Sigma$. On $\tilde{X}$, $\tilde q$ has a canonical square root which 
we denote by $\phi$. To simplify the notation, if $q$ is a square of an 
abelian differential, let $\tilde{X} =X,$ $\tilde{\Sigma}=\Sigma$. 

Let $\bj\from H^1(\tilde X,\tilde \Sigma,\reals) \to H^1(\tilde X,\reals)$ denote the
natural map. We define a norm $\| \param \|$ on the relative cohomology group
$H^1(\tilde X,\tilde \Sigma,\reals)$ as follows:
\begin{equation}
\label{Eq:df:relative:hodge}
\|\lambda\| = \| \bj(\lambda)\|_H + \sum_{(p,p') \in \Sigma \times \Sigma}
\left|\int_{\gamma_{p,p'}} (\lambda - \gh)\right|,
\end{equation}
where $\| \param \|_H$ denotes the Hodge norm on $H^1(\tilde X,\reals)$,
$\gh$ is the harmonic representative of the cohomology class
$\bj(\lambda)$ and $\gamma_{p,p'}$ is any path connecting the zeroes $p$
and $p'$. Since $\bj(\lambda)$ and $\gh$ represent the same class in
$H^1(\tilde X,\reals)$, the \eqnref{df:relative:hodge} does not
depend on the choice of $\gamma_{p,p'}$.

Let $q_t$, $X_t$ and $\phi_t$ be defined as usual and let $\| \param \|_t$ 
denote the norm \eqref{Eq:df:relative:hodge} on the
surface $\tilde X_t= \pi(g_t \phi)$. We have the following analogue 
of \thmref{forni}:

\begin{theorem}
\label{Thm:decay:relative:norm}
Let $\cK$ be a compact subset $\QMs$. Then there is $t_0>0$
so that for $t >t_0$ the following holds. Suppose $\bp(q_0), \bp(q_t) \in \cK$ 
and that the geodesic segment $[q_0, q_t]$ spends at 
least half the time in $\bp^{-1}(\cK)$. Suppose 
$\lambda \in H^1(\tilde{X},\tilde{\Sigma},\reals)$ with 
$$
\bj(\lambda) \wedge \Re (\phi) 
= \bj(\lambda) \wedge \Im (\phi) = 0.
$$
Then we have
\begin{displaymath}
\|\lambda\|_t \le e^{(1-\bar \alpha)t} \| \lambda\|_0,
\end{displaymath}
where $\bar \alpha > 0$ depends only on $\cK$.
\end{theorem}

This theorem is essentially in \cite{AF} (Lemma $4.4$).
We reproduce the proof here for the convenience of the reader.

\begin{proof}[Proof of \thmref{decay:relative:norm}]
Since $\cK$ is compact, quadratic differentials in $\cK$ have no short
saddle connections. Hence, for $u \in [0,t]$, $\bp(q_u) \in \cK$ implies that 
$\tilde X_u$ is thick (has no curves with short extremal lengths). 
Therefore, there exist a constant $\CK$ such that
for any $u$ with $\bp(q_u) \in \cK$, any harmonic 
$\psi \in H^1(\tilde X_u,\reals)$ and any arc $\gamma$ on $\tilde X_u$ with 
end points in $\tilde \Sigma$,
\begin{equation}
\label{Eq:trivial:estimate:integral}
\left|\int_{\gamma} \psi\right| 
  \le \CK \, \|\psi \|_{H, u} \big(1 + \ell_u(\gamma)\big),
\end{equation}
where $\ell_u(\gamma)$ is the length of $\gamma$ in flat metric
associated to $\phi_u$. 

Under the assumptions of \thmref{decay:relative:norm},
there exists $s \in [0.1 t, 0.9t]$ such that $\bp(q_s) \in \cK$.
Fix $p,p' \in \Sigma$. Since $\tilde X_0$ is thick, there exists
a path $\gamma_0$ connecting $p$ and $p'$ with 
$\ell_0(\gamma_0) = O(1)$. Similarly, since $\tilde X_s$ and $\tilde X_t$
are thick there are paths $\gamma_s$ and $\gamma_t$ connecting
$p$ and $p'$ such that $\ell_s(\gamma_s) = O(1)$, 
$\ell_t(\gamma_t) = O(1)$. Then, 
$$
\ell_0(\gamma_s) = O(e^s)\qquad\text{and}\qquad
\ell_s(\gamma_t) = O(e^{t-s}).
$$

Suppose $\lambda \in H^1(\tilde X, \tilde \Sigma,\reals)$ with 
$\bj(\lambda) \wedge \Re(\phi) =
\bj(\lambda) \wedge \Im(\phi) = 0$. Let $\psi = \bj(\lambda)$.
For $0 \le u \le t$, let $\psi_u$ denote the
harmonic representative of the cohomology class $\psi$ on $\tilde X_u$.
We have
\begin{align}
\label{Eq:lambdat:norm}
\|\lambda\|_t - \|\lambda\|_0 & \le \|\psi\|_{H,t} -
\|\psi\|_{H,0} + \sum_{p,p' \in \Sigma \times \Sigma }\left|
 \int_{\gamma_{p,p'}} \psi_t - \psi_0 \right| \notag \\
& \le e^{(1-\alpha)t}
\|\psi\|_{H,0} + \sum_{p,p' \in \Sigma \times \Sigma }\left|
 \int_{\gamma_{p,p'}} \psi_t - \psi_0 \right|,
\end{align}
where we have used \thmref{forni}. Since the integral in
\eqnref{lambdat:norm} is independent of the choice of $\gamma_{p,p'}$, 
we use $\gamma_{p,p'} = \gamma_s$. Then, by 
\eqnref{trivial:estimate:integral},
\begin{equation}
\label{Eq:int:gamma:s:theta:0}
\left|\int_{\gamma_s} \psi_0 \right| 
  \le \CK \, \|\psi\|_{H,0} (1+\ell_0(\gamma_s)) 
  \lmul \CK \,  \|\psi\|_{H,0}\, e^s.
\end{equation}
Also,
\begin{align*}
\left|\int_{\gamma_s} \psi_t \right| & =
\left|\int_{\gamma_s-\gamma_t} \psi_t + \int_{\gamma_t}
\psi_t \right| \\
& = \left|\int_{\gamma_s-\gamma_t} \psi_s + \int_{\gamma_t}
\psi_t \right| \\
& \le \left|\int_{\gamma_s} \psi_s\right| + \left|\int_{\gamma_t}
 \psi_s \right| + \left| \int_{\gamma_t} \psi_t \right| \\
& \lmul \CK 
 \Big(\|\psi\|_{H,s} + \|\psi\|_{H,s} e^{t-s} + \|\psi\|_{H,t}\Big).
\end{align*}
where to pass from the first line to the second we used the fact
that $\psi_s$ and $\psi_t$ represent the same cohomology class in
$H^1(\tilde X,\reals)$, and in the last line we used
\eqnref{trivial:estimate:integral} to estimate each term. Then,
using \eqnref{int:gamma:s:theta:0}, we have
\begin{align*}
\left|\int_{\gamma_s} \psi_t - \psi_0 \right| & \lmul
\CK \Big( \|\psi\|_{H,s} + \|\psi\|_{H,s} e^{t-s} 
+ \|\psi\|_{H,t} + \|\psi \|_{H,0} e^s \Big) \\
& \le \CK \Big(e^{(1-\alpha)s} + e^{(1-\alpha)s+t-s} + e^{(1-\alpha)t} 
+ e^s\Big ) \|\psi\|_{H,0} \\
&\lmul \CK \, e^{(1-0.1\alpha)t} \|\psi\|_{H,0},
\end{align*}
where in the second line we used \thmref{forni} and in the last line
we use the fact that $s \in [0.1t,0.9t]$. Substituting into 
\eqnref{lambdat:norm} we get
\begin{displaymath}
\|\lambda\|_t - \|\lambda\|_0 \le \CK \, e^{(1-0.1\alpha)t}
\|\psi\|_{H,0} \le \CK \, e^{(1-0.1\alpha)t} \|\lambda\|_0.
\end{displaymath}
Assuming $t$ is large enough, we can assume that the multiplicative error
is less than $e^{\alpha_0 t}$ for some $\alpha_0 \leq 0.1 \alpha$. 
The theorem then holds for $\bar \alpha \leq (0.1 \alpha - \alpha_0)$. 
\end{proof}

\subsection{The Hodge Distance.}
Let $g_t$ be the Teichm{\"u}ller flow on $\QMs$.
To each quadratic differential $q$, we associate
its imaginary and real measured foliations $\eta^-(q)$, and $\eta^+(q)$.

The flow $g_t$ admits the following foliations:
\begin{enumerate}
\item $\cF^{ss}$, whose leaves are sets of the form 
$\big\{(X,q) \st \eta^+(q)=\text{const}\big\}$;
\item $\cF^{uu}$, whose leaves are sets of the form 
$\big\{(X,q) \st \eta^-(q)=\text{const}\big\}$.
\end{enumerate}
In other words, for $(X_0,q_0)\in \Qs$, a leaf of $\cF^{ss}$ is given by
$$
\alpha^{ss}(X_0,q_0)=\{(X,q)\in\Qs \st \
\eta^+(q)=\eta^+(q_0)\},
$$
and a leaf of $\cF^{uu}$ is given by
$$
\alpha^{uu}(X_0,q_0)=\{(X,q)\in\Qs \st
\eta^-(q)=\eta^-(q_0)\}.
$$
Note that the foliations $\cF^{ss}$, $\cF^{uu}$ are invariant under
both $g_t$ and $\Gamma(S)$; in particular, they
descend to the moduli space $\QMs$.

We also consider the foliation $\cF^u$ whose leaves are defined by
$$
\alpha^u(q)=\bigcup_{t\in{\mathbb R}} g_t\alpha^{uu}(q)
$$
and $\cF^s$ whose leaves are defined by
$$
\alpha^s(q)=\bigcup_{t\in{\mathbb R}} g_t\alpha^{ss}(q).
$$
If $\Stratum$ is a subset of moduli space of abelian differentials, we can 
locally identify a leaf of $\cF^{ss}$ (or $\cF^{uu}$) with
a subspace $W_-$ (or $W^+$) of $H^1(X,\Sigma,{\mathbb R})$. 
In fact, for $\psi \in W_-$ (or $\psi \in W^+)$, we have
\begin{equation} \label{Eq:Normal}
\bj(\psi) \wedge \Im(\psi) = 0
\quad\text{and}\quad 
\bj(\psi) \wedge \Re(\psi) = 0.
\end{equation}
See $\S{1}$ and $\S{2}$ of \cite{forni} for more details. 

If $\gamma$ is a map from $[0,r]$ into some leaf of $\cF^{ss}$, then
we define the Hodge length $\ell(\gamma)$ of $\gamma$ as $\int_0^r
\|\gamma'(t)\| \, dt$, where $\| \param \|$ is the Hodge
norm.  Finally:
\begin{itemize}
\item If two abelian differentials $\phi$ and $\phi'$
belong to the same leaf of $\cF^{ss}$, then we define $d_H(\phi,\phi')$ to 
be the infimum of $\ell(\gamma)$ where $\gamma$ varies over
paths connecting $\phi$ and $\phi'$ and staying in the leaf of $\cF^{ss}
\subset \Qs$. We make the same definition if $\phi$ and $\phi'$
are on the same leaf of $\cF^{uu}$.
\item By taking a ramified double cover (see \secref{quad}), we can define 
$d_H(q,q')$ for any $q,q'$ on the same leaf of $\cF^{ss}$ in  
$\Qs$. 
\end{itemize}

\begin{lemma}
\label{Lem:hodge:distance:nonincrease}
Let $\cK$ be a compact subset of $\Stratum$.
Suppose $(X,q), (X',q') \in \bp^{-1}(\cK)$ are in the same leaf of $
\cF^{ss}$. 
Let $\gamma$ be a Hodge length minimizing path connecting $q$ to
$q'$. Suppose $t > t_0$ is such that for all $q'' \in \gamma$,
\begin{equation}
\big\{ s \in [0,t] \st g_s q'' \in \bp^{-1}(\cK) \big\} \ge t/2.
\end{equation}
Then
\begin{displaymath}
d_H(g_t q, g_t q') \le e^{-c \, t} d_H(q,q'),
\end{displaymath}
where $c$ depend only on $\cK$.
\end{lemma}

\begin{proof} This follows from \thmref{decay:relative:norm} and
\eqnref{Normal}.
\end{proof}

We now show that the above condition holds whenever the projections
of $g_t q$ and $g_t q'$ to $\Stratum$ are also close. See also Lemma $5.4$ of \cite{EM:geodesic-counting}.

\begin{lemma} \label{Lem:Good-Neighborhood}
Let $\cK$ be a compact subset of $\Stratum$. Then there
is a larger compact subset $\cK'\subset \Stratum$ and a covering of $\cK$ 
with a finite family of open sets $\cU$ so that the following holds. 
Let $U_1, U_2 \subset \Qs$ be connected open sets so that 
$\bp(U_i) \in \cU$, $i=1,2$. Let $(X,q), (X',q') \in U_1$ and $t>0$
be such that $g_t(q), g_t(q') \in U_2$.  Further, assume that
\begin{equation}
\label{Eq:stay:half:time:in:K}
\big\{ s \in [0,t] \st \bp(g_s q) \in \cK \big\} \ge t/2.
\end{equation}
Then, 
\begin{equation}
\label{Eq:stay:half:time:in:K'}
\big\{ s \in [0,t] \st \bp(g_s q') \in \cK' \big\} \ge t/2.
\end{equation}
\end{lemma}

\begin{proof}
Let $\rho>0.$ We can find an open cover $\cU$ of $\cK$ 
 so that the following holds. Let $U$ be connected open sets so that 
$\bp(U) \in \cU$, and let $(X_1,q_1), (X_2,q_2) \in U$. Then for any saddle 
connection $\omega$, we have 
\begin{equation} \label{Eq:Comparable}
\frac 1 \rho \ell_{q_1}(\omega) \leq \ell_{q_2}(\omega) 
  \leq  \rho \, \ell_{q_1}(\omega).
\end{equation}
Let $U_1, U_2 \subset \Qs$ be connected open sets so that 
$\bp(U_i) \in \cU$, $i=1,2$. Let $(X,q), (X',q') \in U_1$ and $t>0$
be such that $g_t(q), g_t(q') \in U_2$.  
We first claim that \eqref{Eq:Comparable} is true for quadratic differentials 
$q_s= g_s(q)$ and $q_s'= g_s(q')$ as well for a larger constant $\rho'=2\rho$.
Assume, for contradiction that 
$$
\ell_{q_s}(\omega) > \rho' \ell_{q_s'}(\omega).
$$
for some $s\in[0,t]$. Assume $\omega$ is mostly vertical in $q_s$. That is, 
$$
\Im (\hol_{q_s}(\omega)) > \frac 12 \ell_{q_s}(\omega). 
$$
Then
\begin{align*}
\ell_q(\omega) &\geq \Im (\hol_q(\omega)) \\
&= e^s  \Im (\hol_{q_s}(\omega)) \\
&>  \frac 12 e^s \ell_{q_s}(\omega)\\
& > \frac 12 e^s \rho' \, \ell_{q_s'}(\omega)
  \geq \frac 12 \rho' \, \ell_{q'}(\omega). 
\end{align*}
Which contradicts \eqnref{Comparable}. In case $\omega$ is mostly
horizontal, we move forward in time and argue the same way. This
proves the claim. 

Now let $\ep$ be such that the length of any saddle connection
in $q \in \cK$ is larger than $\ep$, and let $\cK'$ be the compact subset of 
$\Stratum$ consisting of quadratic differentials where the length of every saddle 
connection is larger than $\ep'=\ep/\rho'$. Then \eqref{Eq:stay:half:time:in:K'} 
follows from the above length comparison. 
\end{proof}

\begin{remark}
\label{Rem:hyperbolicity}
We have essentially shown that under the assumption \eqnref{stay:half:time:in:K}
we have exponential contraction along the foliation $\cF^{ss}$ (and similarly
exponential expansion along the foliation $\cF^{uu}$).
\end{remark}

\section{Outline of the proof of \thmref{asymp:all:geodesics}} \label{Sec:closing}

In this section, we prove \thmref{asymp:all:geodesics}. 
We only outline the arguments here since they are well known
a more detailed version is already present in \cite{Hamenstadt:BM}.
We essentially follow the work of Margulis \cite{Margulis:thesis}. 
First, we need a closing lemma. 

\begin{lemma}[Closing Lemma]
\label{Lem:closing}
Let $\cK$ be a compact subset of $\Stratum$ consisting of non-orbifold points. 
Given a quadratic differential $(x,q) \in \cK$ and $\delta > 0$, 
there exist constants $L_0 > 0$, and open neighborhoods 
$U \subset U'\subset \Stratum$ of $(x,q)$ with the following 
property. For $L>L_0$, suppose that $\bg: [0,L] \to \Stratum$ is a 
\Teich geodesic segment such that 
\begin{enumerate}[\quad (a)]
\item $\bg(0), \bg(L) \in U$ and 
\item $\bg$ spends more than half of its length in $\cK$. 
\end{enumerate}
Let  $\bg_1$ be the closed path in $\Stratum$ which is the union of 
$\bg$ and a segment connecting $\bg(L)$ to $\bg(0)$ in 
$U$. Then there exists a unique closed geodesic $\bg' \subset \Stratum$ 
with the following properties:
\begin{enumerate}[\quad (I)]
\item $\bg'$ and $\bg_1$ have lifts in $\T(S)$ which stay 
$\delta$--close with respect to the Teichm\"uller metric.
\item The length of $\bg'$ is within $\delta$ of $L$, 
\item $\bg'$ passes through $U'$.
\end{enumerate}
\end{lemma}

\begin{remark} 
We remark that in \lemref{closing} if we remove the assumption that $\cK$ consists 
of non-orbifold points then there are at most a uniformly bounded number of
closed geodesics satisfying conditions (I--III).
A version of the closing lemma can be found in \cite{Hamenstadt:BM}.
\end{remark}

\begin{proof}[Outline of the proof of \lemref{closing}]
Consider the stable and unstable foliations for the geodesic flow. Our goal is 
to show that if $U$ is small enough, the first return map on these foliations will 
define a contraction with respect to the Hodge distance function.
As a result, we find a fixed point for the first return map in $U'$; this is the same 
as a closed geodesic going through $U'$. 

In view of \lemref{hodge:distance:nonincrease} and \lemref{Good-Neighborhood}
 there is in fact a neighborhood of $(x,q)$ such that the time 
$L$ geodesic flow restricted to the neighborhood expands along the leaves of 
$\cF^{uu}$ and contracts along the leaves of $\cF^{ss}$. 

Then, the contraction mapping principle (applied first to the map on $\cF^{ss}$ 
and then to the inverse of the map on $\cF^{uu}$) allows us to find
a fixed point for the geodesic flow near $(x,q)$ (in a slightly bigger 
neighboorhood). In other words, there are neighborhoods
$U \subset U'$ of $(x,q)$ such that:

\begin{itemize}
\item if $\bg: [0,L] \to \Stratum$ satisfies properties $(a)$ and $(b)$ 
then in view of the hyperbolicity statement (\lemref{hodge:distance:nonincrease}) 

the time $L$ geodesic flow restricted
to $U$ expands along the leaves of $\cF^{uu}$
and contracts along the leaves of $\cF^{ss}$, in the metric $d_H$, 
\item for any $q_{1}, q_{2} \in U,$ if $q_1 \in \cF^{ss}(q_2)$ or 
$q_1 \in \cF^{uu}(q_2)$ then $d_{H}(q_{1},q_{2}) \leq \delta$.
\end{itemize}
We can apply the contraction mapping principle to $\cF^{ss}$ to find 
$(x_0,q_0) \in U'$ such that $g_{L}(q_{0})\in \cF^{uu} q_0$.
Now we can consider the first return map of the map $g_{-t}$ on 
$\cF^{uu}(q_0)$.
\end{proof}

\begin{proof}[Proof of \thmref{asymp:all:geodesics}]
Note that by the bound proved in \thmref{j:short}, 
we only need to consider the set of closed geodesics going through a fixed 
compact subset of $\Stratum$. We have 
\begin{itemize}
\item by \thmref{VM}, the geodesic flow on $\Stratum$ is mixing, and 
\item on a fixed compact subset of $\cQ^{1}\mathcal{M}(S,\sigma)$ 
the geodesic flow is uniformly hyperbolic.
\item every nearly closed orbit approximates a close orbit (\lemref{closing}). 
\end{itemize} 
Hence, all the ingredients are in place to drive \thmref{asymp:all:geodesics} 
 following the work of Margulis \cite{Margulis:thesis}. 
(See also $\S 20.6$ in  \cite{Katok:Hasselblat}.) 
\end{proof}

\end{document}